\documentclass[a4paper, 11pt, twoside, leqno]{article}

\usepackage[T1]{fontenc}
\usepackage[utf8x]{inputenc}
\usepackage[english]{babel}
\usepackage{lmodern}               
\usepackage{amsmath,amsthm,amssymb}
\usepackage{mathrsfs,esint,bbm,stmaryrd}
\usepackage{enumitem,graphicx,xcolor}
\usepackage[textwidth=15cm,centering,headheight=24pt]{geometry}
\usepackage{authblk}

\newtheorem{theorem}{Theorem}           

\newtheorem{lemma}[theorem]{Lemma}
\newtheorem{proposition}[theorem]{Proposition}

\newtheorem{mainthm}{Theorem}           

\theoremstyle{definition}              
\newtheorem{definition}{Definition}

\theoremstyle{remark}                  
\newtheorem{step}{Step}

\newtheorem{remark}{Remark}

\DeclareMathOperator{\dist}{dist}                                   
\let\div\relax
\DeclareMathOperator{\div}{div}                                     

\newcommand{\abs}[1]{\left| #1 \right|}                             
\newcommand{\norm}[1]{\left\| #1 \right\|}                          

\newcommand{\csubset}{\subset\!\subset}                             

\DeclareMathAlphabet{\mathpzc}{OT1}{pzc}{m}{it}

\newcommand{\T}{\mathrm{T}}
\renewcommand{\d}{\mathrm{d}}
\renewcommand{\o}{\mathrm{o}}
\newcommand{\N}{\mathbb{N}}       
\newcommand{\R}{\mathbb{R}}

\newcommand{\NN}{\mathscr{N}}     
\newcommand{\QQ}{\mathcal{Q}}
\renewcommand{\H}{\mathscr{H}}

\newcommand{\eps}{{\varepsilon}}

\newcommand{\nl}{{\mathrm{nl}}}

\renewcommand{\S}{\mathbb{S}}

\definecolor{lightblue}{rgb}{0.22,0.45,0.70}   
\definecolor{darkgray}{gray}{0.4}    
\definecolor{lightgray}{gray}{0.8}
\newcommand{\BBB}{}

\title{H\"older regularity and convergence for a non-local model \\
 of nematic liquid crystals in the large-domain limit}
\author{Giacomo Canevari\thanks{
 Dipartimento di Informatica --- Universit\`a di Verona,
 Strada le Grazie 15, 37134 Verona, Italy. \\
 \emph{E-mail address}: \texttt{giacomo.canevari@univr.it}} \mbox{}
 and Jamie M. Taylor\thanks{BCAM --- Basque Center for Applied Mathematics,
 Alameda de Mazarredo 14, 48009 Bilbao, Spain. \\
 \emph{E-mail address}: \texttt{jtaylor@bcamath.org}}}
\date{\today}

\begin{document}

\maketitle

\begin{abstract}
We consider a non-local free energy functional, modelling a competition between entropy and pairwise interactions reminiscent of the second order virial expansion, with applications to nematic liquid crystals as a particular case. We build on previous work on understanding the behaviour of such models within the large-domain limit, where minimisers
converge to minimisers of a quadratic elastic energy with manifold-valued constraint, analogous to harmonic maps.
We extend this work to establish H\"older bounds for (almost-)minimisers on bounded domains, and demonstrate stronger convergence of (almost)-minimisers away from the singular set of the limit solution. The proof techniques bear analogy with recent work of singularly perturbed energy functionals, in particular in the context of the Ginzburg-Landau and Landau-de Gennes models. 
\end{abstract}

\numberwithin{equation}{section}
\numberwithin{definition}{section}
\numberwithin{theorem}{section}
\numberwithin{remark}{section}
\numberwithin{goal}{section}

\section{Introduction}

\subsection{Variational models of liquid crystals}

Liquid crystalline systems are those which sit outside of the classical solid-liquid-gas trichotomy. While there are a plethora of different systems classified as liquid crystals, they can be broadly described as fluid systems where molecules admit a long range order of certain degrees of freedom. This is in contrast to classical fluids, which lack long range correlations between molecules. The fluidity of the systems makes them ``soft'', that is, easily susceptible to influence by external influences such as fields or stresses, whilst the long range ordering permits anisotropic electrostatic and optical behaviour. These two properties combined make them ideal for a variety of technological applications, as their anisotropy is exploitable whilst their softness makes them easy to manipulate. 

The simpliest liquid crystalline system is that of a nematic liquid crystal. These are systems of elongated molecules, often idealised as having axial symmetry, which form phases with no long range positional order, but where the long axes of molecules is generally well aligned over larger length scales. Even in the well studied case of nematics there are a variety of models that one may use to study their theoretical behaviour, where the choice of model is usually dependent on the length scales considered and the type of defects one wishes to observe. 

One of the earliest and most studied free energy functionals one may consider in continuum modelling is the Oseen-Frank model~\cite{frank1958liquid}. In the simpliest formulation, we consider a prescribed domain~$\Omega\subseteq\mathbb{R}^3$ and a unit vector field as our continuum variable $n\colon\Omega\to\mathbb{S}^2$, interpreted as the local alignment axis of molecules. As molecules are assumed to be (statistically) head-to-tail symmetric, we interpret the configurations~$n$, $-n$ as physically equivalent, {that is, they are two mathematically distinct representations of the same physical state.} In the simplified {\it one-constant approximation}, we look for minimisers of the free energy 
\begin{equation}\label{refOFOneConstant}
 \int_\Omega \frac{K}{2} \abs{\nabla n(x)}^2 \, \d x,
\end{equation}
subject to certain boundary conditions, although more general formulations are possible. The problem has attracted interest not only from the liquid crystal community, but also from the mathematical community as the prototypical harmonic map problem. If a prescribed Dirichlet boundary condition admits non-zero degree, then by necessity any $n$ satisfying it must admit discontinuities, meaning that defects/singularities are an unavoidable part of the model's study. 

More generally, one may consider an Oseen-Frank energy where different modes of deformation are penalised to different extents. Neglecting the saddle-splay null-Lagrangian term, this gives a free energy of the form 
\begin{equation}\label{eqOF}
 \int_\Omega \frac{K_1}{2}(\nabla \cdot n(x))^2
 + \frac{K_2}{2}(n(x)\cdot \nabla \times n(x))^2
 + \frac{K_3}{2}|n(x)\times \nabla\times n(x)|^2\,\d x. 
\end{equation}
The constants $K_1$, $K_2$, $K_3$ are known as the Frank constants, and represent the penalisations of splay, twist, and bend deformations respectively. In the case where $K_1=K_2=K_3=K$, we reclaim the one-constant approximation~\eqref{refOFOneConstant}.

It is natural to ask if such a free energy can be justified. While the original formulation was more phenomenological in nature and based solely on symmetry arguments and a small-deformation assumption, attempts have been made to identify the Oseen-Frank model as a large-domain limit of a more fundamental model, the Landau-de Gennes model \cite{de1995physics,mottram2014introduction}. In the Landau-de Gennes model, the continuum variable is the {\it Q-tensor}, corresponding to the normalised second moment of a one-particle distribution function. Explicitly, if the distribution of the long axes of molecules in a small neighbourhood of a point $x\in\Omega$ are described by a probability distribution $f(x, \, \cdot)\colon\mathbb{S}^2\to[0, \, +\infty)$, we define the Q-tensor at the point $x$ as 
\begin{equation}
 Q(x) = \int_{\mathbb{S}^2} f(x, \, p) \left(p\otimes p-\frac{1}{3}I\right)\,\d p.
\end{equation}
As molecules are assumed to be head-to-tail symmetric, a molecule is as likely to have orientation $p\in\mathbb{S}^2$ as $-p$, so that $f(x, \, p)=f(x, \, -p)$. For this reason the first moment of $f(x, \, \cdot)$ will always vanish, making the Q-tensor the first non-trivial moment, containing information on molecular alignment. Q-tensors are, following their definition, traceless, symmetric, $3\times 3$ matrices. We denote this set as 
\begin{equation}
 \text{Sym}_0(3)=\left\{Q\in\mathbb{R}^3\colon Q=Q^T,\, \text{Trace}(Q)=0\right\}.
\end{equation} 
The Q-tensor contains more information than the director field, namely that it does not force the interpretation of axially symmetric ordering about an axis (less symmetric configurations are permitted), and the degree of orientational ordering is permitted to vary. 

Depending on their eigenvalues, they come in one of three varieties. 
\begin{itemize}
\item If all eigenvalues are equal, $Q=0$, and we say that~$Q$ is {\it isotropic}, and representative of a disordered system. In particular, if $f$ is a uniform distribution on $\mathbb{S}^2$, $Q=0$.
\item If two eigenvalues are equal and the third is distinct, we say~$Q$ is {\it uniaxial}. A uniaxial Q-tensor can be written as $Q=s\left(n\otimes n-\frac{1}{3}I\right)$, for a scalar $s$ and unit vector $n$. We interpret $n$ as the favoured direction of alignment, and $s$ as a measure of the degree of ordering molecules about $n$. 
\item If all three eigenvalues are distinct, we say that $Q$ is {\it biaxial}.
\end{itemize}

The corresponding free energy to be minimised is
\begin{equation}
\int_\Omega \psi_b(Q(x)) + W(Q(x),\nabla Q(x))\,\d x.
\end{equation}
The function $\psi_b\colon\text{Sym}_0(3)\to\mathbb{R}\cup\{+\infty\}$ is a frame indifferent bulk potential, which may be taken as a polynomial or the Ball-Majumdar singular potential {(\cite{ball2010nematic}, and further discussion in \eqref{hp:Hfirst}--\eqref{hp:Hlast})}. Its main characteristic is that, in the cases considered, it is minimised on the set
\begin{equation} \label{NN}
 \NN =\left\{Q\in\text{Sym}_0(3)\colon \textrm{there exists }
 n\in \mathbb{S}^2 \textrm{ such that } 
 Q=s_0\left(n\otimes n-\frac{1}{3}I\right)\right\},
\end{equation}
with~$s_0$ a temperature, concentration and material dependent constant. The elastic energy~$W$ is minimised when $\nabla Q=0$. While many forms are possible, by symmetry the only frame-indifferent, quadratic energy that only depends on the gradient of~$Q$ is of the form 
\begin{equation}
 W(\nabla Q) = \frac{L_1}{2} Q_{ij,k}Q_{ij,k}
 + \frac{L_2}{2} Q_{ij,k}Q_{ik,j} 
 + \frac{L_3}{2}Q_{ij,j}Q_{ik,k},
\end{equation}
where Einstein summation notation is used. While Oseen-Frank represents nematic defects as discontinuities in the continuum variables, the Landau de-Gennes approach admits a different description, where nematic defects point defects are typically described as a melting of nematic order, that is~$Q=0$. This permits smooth configurations to describe defects.

In an appropriate large-domain limit of a rescaled problem, the contributions of the bulk energy become overwhelming, and we expect the minimisers to converge to minimisers of a constrained problem, where we minimise the elastic energy
\begin{equation}
 \int_\Omega W(\nabla Q)\,\d x,
\end{equation}
subject to the constraint that $Q(x)\in\NN$ almost everywhere. In the case where~$Q=s_0\left(n\otimes n-\frac{1}{3}I\right)$ almost everywhere for some~$n\in W^{1,2}(\Omega,\mathbb{S}^2)$, we say that~$Q$ is {\it orientable}, and the problem in the presence of Dirichlet boundary conditions that are $\NN$-valued almost everywhere becomes equivalent to that of the minimising the energy \eqref{eqOF} for~$n$. The constants~$L_i$ and~$K_i$ are related in the case of Dirichlet boundary conditions where null-Lagrangian terms may be neglected as 
\begin{equation}
\begin{split}
\frac{1}{s_0^2}K_1=& 2L_1+L_2+L_3,\\
\frac{1}{s_0^2}K_2=&2L_1,\\
\frac{1}{s_0^2}K_3=&2L_1.
\end{split}
\end{equation}
An energy purely quadratic in $\nabla Q$ cannot give rise to three independent elastic constants in the Oseen-Frank model, with the so-called ``cubic term'' $Q_{ij}Q_{kl,i}Q_{kl,j}$ often being used to fill the degeneracy. Such a term does not arise from the model we will consider, although a more complex variant taking into account molecular length scales has been proposed to avoid this issue~\cite{creyghton2018scratching}.

Studying the convergence of minimisers of Landau-de Gennes towards the Oseen-Frank limit has attracted interest, 
with Majumdar and Zarnescu showing global $W^{1,2}$~convergence 
and uniform convergence away from singular sets in
the one-constant case~\cite{majumdar2010landau}, 
Nguyen and Zarnescu proving convergence results
in stronger topologies~\cite{NguyenZarnescu},
Contreras, Lamy and Rodiac generalising the approach
to other harmonic-map problems~\cite{contreras2018convergence},
and further extensions by Contreras and Lamy~\cite{contreras2018singular}  and Canevari, Majumdar and Stroffolini~\cite{canevari2019minimizers} 
to more general elastic energies.
In other settings, the $W^{1,2}$-convergence does not 
hold globally but only locally, away from the singular sets, 
due to topological obstructions carried by the boundary data
and/or the domain 
(see e.g.~\cite{BaumanParkPhillips,GolovatyMontero, pirla3, IgnatJerrard}).
Recently, Di Fratta, Robbins, Slastikov and Zarnescu
found higher-order Landau-de Gennes corrections
to the Oseen-Frank functional, in two dimensional domains,
by studying the $\Gamma$-expansion of the
Landau-de Gennes functional in the large-domain
limit~\cite{diFrattaetal2020}.
The problem holds many parallels to the now-classical Ginzburg-Landau problem~\cite{bethuel1994ginzburg}.
Other singular limits and qualitative features
of Landau-de Gennes solutions
have been studied too; see, for instance,~\cite{ContrerasLamy-Biaxial,
Fratta16, HenaoMajumdarPisante, INSZ-Instability, kitavtsev2016liquid, INSZ-Hedgehog, INSZ-2dstability, INSZ-symmetry} and the references therein.

While Landau-de Gennes has proven an effective model in many situations, there are still open questions as to how one may justify the model in a rigorous way. While one may use Landau-de Gennes, in appropriate situations, to justify Oseen-Frank, a rigorous justification of Landau-de Gennes itself is lacking. Historically it was justified on a phenomenological basis, but other work has been able to provide Landau-de Gennes as a gradient expansion of a non-local mean field model~\cite{garlea2017landau,han2015microscopic}. Justification by formal gradient expansions leaves open the question as to the {consistency of minimisers of the original free energy with minimisers of its approximation, that is, are minimisers of the approximate model necessarily good approximations of the minimisers of the original problem?} To this end, recent work has been focused on rigorous asymptotic analysis of non-local free energies, which similarly produce the Oseen-Frank model in a large-domain limit~\cite{liu2017oseen,liu2018small,taylor2018oseen,taylor2019gamma}. These approaches ``bypass'' the intermediate and non-rigorous derivation of Landau-de Gennes. This is analogous to recent investigations into peridynamics, a formulation of elasticity based on non-local interactions. These formulations of elasticity bear mathematical similarity with the mean-field theory approach, where stress-strain relations are described in terms of non-local operators on the deformation map, rather than derivatives as in the more classical formulations of elasticity~\cite{bellido2015hyperelasticity,silling2008convergence}. The classical density functional theory we will consider in this work is based on a simplified competition between an entropic contribution to the energy, favouring disorder, and an interaction energy, favouring order. The models themselves are justified as a second order truncation of the virial expansion in the dilute regime based on long-range attractive interactions in the style of Maier and Saupe \cite{maier1959einfache} and with mathematical similarity to the model of Onsager \cite{onsager1949effects}.
 Explicitly, given the one-particle distribution function $f(x,\cdot)$ in a neighbourhood of $x$, we define a free energy functional 
\begin{equation}\label{eqMeanField}
 \begin{split}
  &k_BT\rho_0 \int_{\Omega\times\mathbb{S}^2} f(x, \, p)\ln f(x, \, p) \,\d x\,\d p \\
  &\qquad\qquad\qquad - \frac{\rho_0^2}{2} \int_{\Omega\times\mathbb{S}^2}
  \int_{\Omega\times\mathbb{S}^2}
  f(x, \, p)f(y, \, p) \mathcal{K}(x-y,\, p, \, q)\,\d x\,\d p\,\d y\,\d q.
 \end{split}
\end{equation}  
$\rho_0>0$ is the number density of particles in space, $k_B$~the Boltzmann constant, $T>0$~temperature and $\mathcal{K}(z, \, p, \, q)$ denotes the interaction energy of particles with orientations $p$, $q$ and with centres of mass separated by a vector~$z$. The entropic term on the left is convex and readily shown to be minimised at a uniform distribution, that is, an isotropic disordered system. The nature of the pairwise interaction energy on the right is that nearby particles will prefer to be aligned with each other. We see that temperature and concentration mediate the competition between these opposing mechanisms. Recent work has established the Oseen-Frank energy~\eqref{eqOF} in terms of a large-domain limit of the energy~\eqref{eqMeanField} under certain assumptions, in which the elastic constants~$K_i$ can be related to second moments of the interaction kernel. Previous work has established weaker modes of convergence, while in this work we will establish stronger convergence of minimisers away from defect sets, analogous to the approach taken by Majumdar and Zarnescu for the Landau-de Gennes model~\cite{majumdar2010landau}. 
{\BBB~The results however will require stronger assumptions on the regularity and decay of the interaction kernel than those of \cite{taylor2018oseen}, owing to the need for more precise control on the decay of various integral quantities.}

\subsection{Simplification of the model and non-dimensionalisation}

Here and throughout the sequel, we consider the more general case where molecules admit an internal degree of freedom $p$ in a manifold $\mathcal{M}$. We will employ a macroscopic order parameter $u\in\mathbb{R}^m$ to emphasise the analysis is not limited to the concrete case of nematic liquid crystals.

Through most of the paper, we consider the case
where~$f$ is prescribed on
$\left(\mathbb{R}^3\setminus\frac{1}{\eps}\Omega\right)\times\mathcal{M}$,
where~$\Omega$ is a non-dimensional reference domain 
and~$\eps>0$ is a small parameter, representative of 
the inverse of a large length scale of the domain.
In Section~\ref{sect:bd}, we relax this assumption
and study a minimisation problem where~$f$ is prescribed only
in a neighbourhood of the domain, of suitable thickness.
We consider the free energy 
\begin{equation}
 \begin{split}
  &\tilde{\mathcal{G}}_\eps(f) = k_BT\rho_0 
  \int_{\frac{1}{\eps}\Omega\times\mathcal{M}} f(x, \, p)\ln f(x, \, p)\,\d x\,\d p \\
  &\qquad\qquad\qquad -\frac{\rho_0^2}{2} \int_{\mathbb{R}^3\times\mathcal{M}}
  \int_{\mathbb{R}^3\times\mathcal{M}} f(x,\, p)f(y, \, q)
  \mathcal{K}(x-y, \, p, \, q)\,\d x\,\d p\,\d y\,\d q.
 \end{split}
\end{equation}
For simplification of the problem, we take the interaction energy to be of the form 
\begin{equation}
 \mathcal{K}(z, \, p, \, q)=K(z)\sigma(p)\cdot \sigma(q),
\end{equation}
where~$\sigma\in L^\infty(\mathcal{M},\mathbb{R}^m)$ 
is some ``microscopic order parameter'',
and $K\colon\mathbb{R}^3\to\mathbb{R}^{m\times m}$
is a symmetric tensor field, which will satisfy certain technical conditions 
(see~\eqref{hp:Kfirst}--\eqref{hp:Klast} in Section~\ref{sect:setting}). By applying Fubini we may then introduce a ``macroscopic order parameter'',
$u\in L^\infty(\mathbb{R}^3,\mathbb{R}^m)$ by 
\begin{equation}
 u(x) = \int_{\mathcal{M}}f(x, \, p)\sigma(p)\,\d p,
\end{equation}
and re-write the interaction energy as 
\begin{equation}
 \begin{split}
  &-\frac{\rho_0^2}{2} \int_{\mathbb{R}^3\times\mathcal{M}}
  \int_{\mathbb{R}^3\times\mathcal{M}} 
  f(x, \, p)f(y, \, p)\mathcal{K}(x-y, \, p, \, q) \,\d x\,\d p\,\d y\,\d q \\
  &\qquad\qquad\qquad = -\frac{\rho_0^2}{2} 
  \int_{\mathbb{R}^3}\int_{\mathbb{R}^3} K(x-y)u(x)\cdot u(y)\,\d x\,\d y.
 \end{split}
\end{equation}
While it is not possible to write the entropic term explicitly in terms of $u$, we may provide a lower bound by means of the singular Ball-Majumdar/Katriel potential and its extensions~\cite{ball2010nematic,katriel1986free,taylor2016maximum} by 
\begin{equation}
 \int_{\frac{1}{\eps}\Omega\times\mathcal{M}}
 f(x, \, p)\ln f(x, \, p)\,\d x\,\d p 
 \geq \int_{\frac{1}{\eps}\Omega}\psi_s(u(x))\,\d x,
\end{equation}
where the function $\psi_s\colon\mathbb{R}^m\to\mathbb{R}\cup\{+\infty\}$ 
is defined by 
\begin{equation} \label{BallMajumdar}
 \psi_s(u)=\min\left\{\int_{\mathcal{M}}f(p)\ln f(p)\,\d p \colon 
 f\geq 0 \textrm{ a.e.,} \, \int_{\mathcal{M}} f(p)\,\d p=1,
 \, \int_{\mathcal{M}} f(p)\sigma(p)\,\d p = u\right\} \! ,
\end{equation}
where by convention~$\psi_s(u)=+\infty$ when the constraint set is empty. Note that the minimisation problem \eqref{BallMajumdar} is strictly convex, thus solutions are necessarily unique, and we may define~$f_u$ to be the corresponding minimiser for $u\in \QQ=\left\{u: \psi_s(u)<+\infty\right\}$. That is, 
\begin{equation}
f_u=\text{arg min}\left\{\int_{\mathcal{M}}f(p)\ln f(p)\,\d p \colon 
 f\geq 0 \textrm{ a.e.,} \, \int_{\mathcal{M}} f(p)\,\d p=1,
 \, \int_{\mathcal{M}} f(p)\sigma(p)\,\d p = u\right\}.
\end{equation} 
The precise definition of~$\psi_s$ will be unimportant in this work, and we employ any function~$\psi_s$ satisfying certain technical assumptions in the sequel (see~\eqref{hp:Hfirst}--\eqref{hp:Hlast} in Section~\ref{sect:setting})

We in fact have the result 
that $f^*$ is a minimiser of~$\tilde{\mathcal{G}}_\eps$ if and only if, for $u^*(x)=\int_{\mathbb{S}^2}f^*(x, \, p)\sigma(p)\,\d p$, $u^*$ is a minimiser of 
\begin{equation} \label{energy-1}
\tilde{\mathcal{F}}_\eps(u)=k_BT\rho_0\int_{\frac{1}{\eps}\Omega}\psi_s(u(x))\,\d x-\frac{\rho_0^2}{2}\int_{\mathbb{R}^3}\int_{\mathbb{R}^3}K(x-y)u(x)\cdot u(y)\,\d x\,\d y,
\end{equation}
with $f^*=f_{u^*}$. This is readily seen by writing the minimisation as a two-step process, first minimising over all $f$ such that $f_u=f$, and later minimising over $u$ and noting that the first minimisation may be performed pointwise almost-everywhere in $\mathbb{R}^3$, as in \cite{taylor2018oseen}. 
That is to say, we have a simpler, macroscopic energy with equivalent minimisers. 
By introducing a change of variables, 
\[
 x=\frac{x^\prime}{\eps}, \quad y=\frac{y^\prime}{\eps}, \quad
 u^\prime(x^\prime)=u(x), \quad \eps^\prime := \frac{\eps}{\rho_0^{1/3}},
 \quad K^\prime(x^\prime)=\frac{1}{k_BT}K(\eps^\prime x),
\]
and a (non-dimensional) constant~$C_{\eps^\prime}$ 
to be specified later, we rescale the domain and obtain the free energy we will consider for the remainder of this work, so that 
\begin{equation} \label{energy0}
\begin{split} 
 E_{\eps^\prime}(u^\prime)
 &:= \frac{\eps}{k_BT\rho_0^{1/3}} \tilde{\mathcal{F}}_\eps(u) + C_{\eps^\prime}\\
 &= \frac{1}{{\eps^\prime}^2} \int_{\Omega} \psi_s(u^\prime(x^\prime))\,\d x^\prime 
 - \frac{1}{2{\eps^\prime}^5}\int_{\mathbb{R}^3} \int_{\mathbb{R}^3}
 K^\prime\left(\frac{x^\prime-y^\prime}{\eps^\prime}\right)u^\prime(x^\prime)\cdot u^\prime(y^\prime)\,\d x^\prime\,\d y^\prime + C_{\eps^\prime}. 
\end{split}
\end{equation}
The additive constant~$C_{\eps^\prime}$ 
is irrelevant for the purpose of minimisation;
however, we will make a specific choice of~$C_{\eps^\prime}$
(see Equation~\eqref{C_eps} below) for analytical convenience.
We will consider the regime as $\eps^\prime\to 0$ in this work. 
From the definition of $\eps^\prime$, this may be interpreted in two forms, one in which the characteristic length scale of the domain, $\frac{1}{\eps}$, becomes large, and one in which the density $\rho_0$ becomes large. However, as the energy we consider is based on the second order virial expansion which is explicitly a model for dilute regimes, we interpret the limit $\eps^\prime\to 0$ as the former, that is, a large-domain limit.
In the sequel we omit the primes and consider~\eqref{energy0} as 
our free energy functional to be minimised at scale~$\eps>0$.

\section{Technical assumptions and main results}
\label{sect:setting}

Let~$\text{Sym}(m)$ be the space
of~$(m\times m)$-symmetric matrices, with real coefficients.
Given an interaction kernel~$K\colon\R^3\to\text{Sym}(m)$
and~$\eps> 0$, we define~$K_\eps(z) := \eps^{-3} K(\eps^{-1}z)$
for any~$z\in\R^3$. Then, we may rewrite the functional~\eqref{energy0} as
\begin{equation} \label{energy}
 \begin{split}
  E_\eps(u) := - \frac{1}{2\eps^2}\int_{\R^3\times\R^3} 
     K_\eps(x-y)u(x)\cdot u(y) \, \d x \, \d y 
     + \frac{1}{\eps^2} \int_\Omega\psi_s(u(x)) \, \d x
     + C_\eps,
 \end{split}
\end{equation}
where~$u\colon\R^3\to\R^m$ is the macroscopic order parameter,
$\Omega\subseteq\R^3$ is a bounded, smooth domain, and
$\psi_s\colon\R^m\to [0, \, +\infty]$ is any convex potential
that satisfies the assumptions~\eqref{hp:Hfirst}--\eqref{hp:Hlast} below
(for instance, the Ball-Majumdar/Katriel potential defined by~\eqref{BallMajumdar}).

\paragraph*{Assumptions on the kernel~$K$.}

Our assumptions on the kernel~$K$ are reminiscent of~\cite{taylor2018oseen}.
We define~$g(z) := \lambda_{\min}(K(z))$ for any~$z\in\R^3$,
where~$\lambda_{\min}(K)$ denotes the minimum eigenvalue of~$K$.

\begin{enumerate}[label=(K\textsubscript{\arabic*}), 
ref=K\textsubscript{\arabic*}]
 \item \label{hp:K_decay} \label{hp:Kfirst}
  $K\in W^{1,1}(\R^3, \, \text{Sym}(m))$.
 
 \item \label{hp:K_even} $K$ is even, that is $K(z) = K(-z)$ for a.e.~$z\in\R^m$.
 
 \item \label{hp:g} $g\geq 0$ a.e. on~$\R^3$, and there exist 
 positive numbers~$\rho_1 < \rho_2$, $k_*$ such that~$g\geq k_*$
 a.e.~on~$B_{\rho_2}\setminus B_{\rho_1}$.
 
 \item \label{hp:g_decay} $g\in L^1(\R^3)$ and {\BBB there
 exists~$q>7/2$ such that
 $\int_{\R^3} g(z) \abs{z}^q \d z <+\infty$.}
 
 \item \label{hp:lambda_max} There exists a positive constant~$C$ such that 
 $\lambda_{\max}(K(z))\leq Cg(z)$ for a.e.~$z\in\R^3$
 (where~$\lambda_{\max}(K)$ denotes the maximum eigenvalue of~$K$).
 
 \item \label{hp:nabla_K} \label{hp:Klast} {\BBB There
 exists~$\nu > 1$ such that}
 \[
  \int_{\R^3} \norm{\nabla K(z)} \abs{z}^{\nu} \d z <+\infty,
 \]
 where~$\norm{\nabla K(z)}^2 := 
 \partial_\alpha K_{ij}(z) \, \partial_\alpha K_{ij}(z)$.
\end{enumerate}

In the case of physically meaningful systems the tensor $K$ will have to respect frame invariance. In the case of nematic liquid crystals, where the order parameter is a traceless symmetric matrix~$Q$, frame indifference implies that the bilinear form must necessarily be of the form
\begin{equation}
K(z)Q_1\cdot Q_2 = f_1(|z|)Q_1\cdot Q_2 + f_2(|z|)Q_1z\cdot Q_2z + f_3(|z|) (Q_1z\cdot z)(Q_2z\cdot z),
\end{equation}
for all $Q_1,Q_2 \in \text{Sym}_0(3)$, where~$f_1$, $f_2$, $f_3$ are real-valued functions defined on $[0, \, +\infty)$ \cite{taylor2018oseen}. It is clear that by appropriate choices of~$f_1$, $f_2$, $f_3$ which are~$C^1$ and with sufficient decay at infinity the previous assumptions can be satisfied. This family of bilinear forms includes the simplified interaction energy 
\begin{equation}
K(z)Q_1\cdot Q_2=\varphi(|z|)Q_1\cdot Q_2
\end{equation}
{\BBB for a suitable function $\varphi$, which includes the results of~\cite{liu2017oseen,liu2018small}, where $\varphi$ is taken to be rapidly decaying and $C^\infty$, which are stronger assumptions than we shall consider. Furthermore ~\cite[Equation (3.43)]{bowick2017mathematics} considers $K$ to have the same structure, albeit with a slower decay of $\varphi$ than our analysis would permit.} 

{\BBB~We remark that the integrability requirements in \eqref{hp:g_decay} and \eqref{hp:nabla_K} and regularity requirement in \eqref{hp:K_decay}, although weaker than the assumptions in the earlier work \cite{liu2017oseen}, are stronger than that of \cite{taylor2018oseen}, and permit more delicate control of integral estimates needed to show convergence of minimisers in a stronger sense.
(See also Remark~\ref{rk:assumptions} below.)}

\paragraph*{Assumptions on the singular potential~$\psi_s$.}

\begin{enumerate}[label=(H\textsubscript{\arabic*}), ref=H\textsubscript{\arabic*}]
 \item \label{hp:psi_s} \label{hp:Hfirst}
 $\psi_s\colon\R^m\to {\BBB (-\infty, \, +\infty]}$ is a convex function.
 
 \item \label{hp:Q} The domain of~$\psi_s$, 
 $\QQ := \psi_s^{-1}({\BBB\R})\subseteq\R^m$,
 is a non-empty, bounded open set {\BBB that contains~$0$}
 and~$\psi_s\in C^2(\QQ)$.
 
 \item \label{hp:unif_conv} There exists a constant~$c>0$ such that
 $\nabla^2\psi_s(y)\chi\cdot\chi\geq c\abs{\chi}^2$ for any~$y\in\QQ$ and any~$\chi\in\R^m$.
 
 \item \label{hp:blow-up} There holds $\psi_s(y)\to+\infty$
 as~$\dist(y, \, \partial\QQ)\to 0$.
\end{enumerate}

We define the ``bulk potential''~$\psi_b\colon\QQ\to\R$
in terms of~$K$ and~$\psi_s$, as
\begin{equation} \label{psib}
 \psi_b(y) := \psi_s(y) -
   \frac{1}{2}\left(\int_{\R^3} K(z)\, \d z\right)y\cdot y + c_0
   \quad \textrm{for any } y\in\QQ,
\end{equation}
where~$c_0\in\R$ is a constant, uniquely determined by
imposing that $\inf\psi_b = 0$.
We make the following assumptions on~$\psi_b$:

\begin{enumerate}[label=(H\textsubscript{\arabic*}), ref=H\textsubscript{\arabic*}, resume]
 \item \label{hp:NN} The set $\NN := \psi_b^{-1}(0)\subseteq\QQ$ 
 is a compact, smooth, connected manifold without boundary.
 
 \item \label{hp:non_degeneracy} \label{hp:Hlast}
 For any~$y\in\NN$ and any unit vector~$\xi\in\S^{m-1}$
 that is orthogonal to~$\NN$ at~$y$, we have $\nabla^2\psi_b(y)\xi\cdot\xi>0$.
\end{enumerate}

\begin{remark} \label{rk:trivial}
 If the norm of~$\int_{\R^3}K(z)\, \d z$ is smaller than the constant~$c$
 given by~\eqref{hp:unif_conv}, then the function~$\psi_b$
 is {\BBB strictly} convex
 and hence, its zero-set~$\NN$ reduces to a point.
 This happens, for example, in the sufficiently 
 high temperature regime, independently of the precise form of~$K$.
 Nevertheless, our arguments remain valid in this case, too.
\end{remark}

\begin{remark} \label{rk:BallMajumdar}
 The Ball-Majumdar/Katriel potential, defined by~\eqref{BallMajumdar},
 satisfies the conditions~\eqref{hp:Hfirst}--\eqref{hp:Hlast}.
 \eqref{hp:psi_s},~\eqref{hp:Q},~\eqref{hp:blow-up}, and~\eqref{hp:NN} follow from \cite{ball2010nematic}, apart from the $C^2$ smoothness of $\psi_s$ which is implicitly proven in \cite{katriel1986free} via an inverse function theorem argument, although not stated. \eqref{hp:unif_conv} is proven in \cite{taylor2016maximum}. 
 With this choice of the potential, the set~$\NN :=\psi_b^{-1}(0)$
 is either a point or the manifold given by~\eqref{NN} 
 (see~\cite[Section~4]{ball2010nematic}).
 In both cases, \eqref{hp:non_degeneracy} is satisfied
 (see~\cite[Proposition 4.2]{li2015local}).
\end{remark}

\paragraph*{The admissible class and an equivalent expression for the free energy.}

We complement the minimisation of the functional~\eqref{energy}
by prescribing~$u$ on~$\R^3\setminus\Omega$. We take
a map~$u_{\mathrm{bd}}\in H^1(\R^3, \, \R^m)$ such that
\begin{equation} \label{hp:bd} \tag{BD}
 u_{\mathrm{bd}}(x)\in\QQ \quad \textrm{for a.e. } x\in\R^3\setminus\Omega, 
 \qquad u_{\mathrm{bd}}(x)\in\NN \quad \textrm{for a.e. } x\in\Omega, 
\end{equation}
and we define the admissible class
\begin{equation} \label{A}
 \mathscr{A} := \left\{u\in L^\infty(\R^3, \, \QQ)\colon
 \psi_s(u)\in L^1(\Omega), \ 
 u = u_{\mathrm{bd}} \textrm{ a.e. on } \R^3\setminus\Omega\right\} \!.
\end{equation}

\begin{remark}
 {\BBB In order for the assumption~\eqref{hp:bd} to be satisfied,
 it is necessary that the trace of~$u_{\mathrm{bd}}$ on~$\partial\Omega$
 takes its values in the manifold~$\NN$. If~$\NN$ is simply connected,
 then any boundary value that belongs to~$H^{1/2}(\partial\Omega, \, \NN)$
 admits an extension in~$H^1(\Omega, \, \NN)$; 
 this follows from~\cite[Theorem~6.2]{HardtLin}. However, 
 when~$\NN$ is multiply connected (for instance, when~$\NN$ is
 the real projective plane, as in the applications to liquid crystals) 
 there exist boundary values in~$H^{1/2}(\partial\Omega, \, \NN)$ 
 that do not have any extension in~$H^1(\Omega, \, \NN)$.
 (See e.g.~\cite{bethuel2014, MironescuVanSchaftingen-trace}
 for results on the extension problem for manifold-valued Sobolev maps.)
 On the other hand, $\QQ$ is a convex set that contains~$\NN$ and~$0$,
 so any boundary value in~$H^{1/2}(\partial\Omega, \, \NN)$
 has an extension in~$H^1(\R^3\setminus\Omega, \, \QQ)$.}
\end{remark}

In the class~$\mathscr{A}$, the functional~$E_\eps$
has an alterative expression.
For any~$y\in\R^m$, we use the abbreviated 
notation~$y^{\otimes 2} := y\otimes y$. 
We choose
\begin{equation} \label{C_eps}
 C_\eps :=
 \frac{c_0}{\eps^2} \abs{\Omega}
 + \frac{1}{2\eps^2} \int_{\R^3\setminus\Omega}
 \left(\int_{\R^3}K(z)\,\d z\right)
 \cdot u_{\mathrm{bd}}(x)^{\otimes 2} \,\d x,
\end{equation}
where~$\abs{\Omega}$ denotes the volume of~$\Omega$
and~$c_0\in\R$ is the same number as in~\eqref{psib}.
The constant~$C_\eps$ only depends on~$\eps$,
$\Omega$, $K$ and~$u_{\mathrm{bd}}$, 
so it is does not affect minimisers of the functional.
By applying the algebraic identity
\[
 -2K(x-y)u(x)\cdot u(y) 
 = K(x-y)\cdot(u(x) - u(y))^{\otimes 2}
 - K(x-y)\cdot u(x)^{\otimes 2} - K(x-y)\cdot u(y)^{\otimes 2}
\]
and using~\eqref{psib}, \eqref{C_eps},
we re-write~\eqref{energy} as
\begin{equation} \label{energy2}
 \begin{split}
  E_\eps(u) = \frac{1}{4\eps^2}\int_{\R^3\times\R^3} K_\eps(x-y)
     \cdot \left(u(x) - u(y)\right)^{\otimes 2} \, \d x \, \d y  
     +\frac{1}{\eps^2} \int_\Omega\psi_b(u(x)) \, \d x
 \end{split}
\end{equation}
for any~$u\in\mathscr{A}$.
We note that the free energy admits parallels to the Landau-de Gennes energy, with the right-hand term being a corresponding bulk energy and the left-hand term acting as a non-local analogue of the elastic energy, which we shall see is reclaimed in a precise way in the asymptotic limit as~$\eps\to 0$.
Let~$L$ be the unique symmetric fourth-order tensor that satisfies
\begin{equation} \label{L}
 L\xi\cdot \xi := \frac{1}{4}\int_{\R^3} K(z) \cdot (\xi z)^{\otimes 2}\, \d z
 \qquad \textrm{for any } \xi\in\R^{m\times 3}.
\end{equation}
{\BBB The right-hand side of~\eqref{L} is well-defined and finite 
for any~$\xi\in\R^{m\times 3}$ because~$K$ has finite second moment,
thanks to the assumption~\eqref{hp:g_decay}.}
Coordinate-wise, $L$ is defined by
\[
 L_{ij\alpha\beta} = \frac{1}{4} \int_{\R^3} K_{\alpha\beta}(z)\,z_i\,z_j\,\d z
\]
for any~$i$, $j\in\{1, \, 2, \, 3\}$
and~$\alpha$, $\beta\in\{1, 2, \, \ldots, \, m\}$.
Let~$E_0\colon\mathscr{A}\to[0, \, +\infty]$ be given as
\begin{equation} \label{limit}
 E_0(u) := \begin{cases}
            \displaystyle\int_{\Omega} L\nabla u \cdot \nabla u
              & \textrm{if } u\in H^1(\Omega, \, \NN)\cap\mathscr{A} \\
            +\infty &\textrm{otherwise.}
           \end{cases}
\end{equation}
By assumption~\eqref{hp:bd}, the set~$H^1(\Omega, \, \NN)\cap\mathscr{A}$
is non-empty and hence, the functional~$E_0$ is not identically equal to~$+\infty$.
Taylor~\cite{taylor2018oseen} proved that, 
as~$\eps\to 0$, the functional~$E_\eps$ $\Gamma$-converges to~$E_0$
with respect to the $L^2$-topology.
In particular, up to subsequences, minimisers~$u_\eps$ 
of~$E_\eps$ in the class~$\mathscr{A}$
converge $L^2$-strongly to a minimiser $u_0$ of~$E_0$ in~$\mathscr{A}$.
Our aim is to prove a convergence result for minimisers, in a stronger topology.

\paragraph*{Main results.}

Given a Borel set~$G\subseteq\R^3$
and~$u\in L^\infty(G, \, \QQ)$, we define 
\begin{equation} \label{loc_energy}
 F_\eps(u, \, G) := \frac{1}{4\eps^2}\int_{G\times G} 
     K_\eps(x-y) \cdot \left(u(x) - u(y)\right)^{\otimes 2} \, \d x \, \d y  
     +\frac{1}{\eps^2} \int_{G} \psi_b(u(x)) \, \d x.
\end{equation}
For any~$\mu\in (0, \, 1)$, we denote the $\mu$-H\"older
semi-norm of~$u$ on~$G$ as
\[
 [u]_{C^\mu(G)} := \sup_{x, \, y\in G, \ x\neq y} 
 \frac{\abs{u(x) - u(y)}}{\abs{x-y}^\mu}
\]

\begin{mainthm}[Uniform~$\eta$-regularity] \label{th:Holder}
 Assume that~\eqref{hp:Kfirst}--\eqref{hp:Klast},
 \eqref{hp:Hfirst}--\eqref{hp:Hlast} and~\eqref{hp:bd} are satisfied. 
 Then, there exist positive
 numbers~$\eta$, $\eps_*$, $M$ and~$\mu\in (0, \, 1)$ such that
 for any ball~$B_{r_0}(x_0)\subseteq\Omega$, 
 any~$\eps\in (0, \, \eps_* r_0)$,
 and any minimiser~$u_\eps$ of~$E_\eps$ in~$\mathscr{A}$ such that
 \[
  r_0^{-1} F_\eps(u_\eps, \, B_{r_0}(x_0)) \leq \eta^2
 \]
 there holds
 \[
  r_0^{\mu} \, [u_\eps]_{C^\mu(B_{r_0/2}(x_0))} \leq M.
 \]
\end{mainthm}

As a corollary, we deduce a convergence result
for minimisers of~$E_\eps$, in the locally uniform topology.
We recall that any minimiser~$u_0$ for the limit
functional~\eqref{limit} in~$\mathscr{A}$ 
is smooth in~$\Omega\setminus S[u_0]$, where 
\begin{equation} \label{singularset}
 S[u_0] := \left\{x\in\Omega\colon 
 \liminf_{\rho\to 0} \rho^{-1}
 \int_{B_\rho(x)}\abs{\nabla u_0}^2 > 0 \right\} \!.
\end{equation}
Moreover, $S[u_0]$ is a closed set of zero total length 
(see e.g.~\cite{HKL, Luckhaus}).

\begin{mainthm} \label{th:conv}
 Assume that the conditions~\eqref{hp:Kfirst}--\eqref{hp:Klast},
 \eqref{hp:Hfirst}--\eqref{hp:Hlast} and~\eqref{hp:bd} are satisfied. 
 Let~$u_\eps$ be a minimiser of~$E_\eps$ in~$\mathscr{A}$.
 Then, up to extraction of a (non-relabelled) subsequence, we have
 \[
  u_\eps \to u_0 \qquad \textrm{locally uniformly in }
  \Omega\setminus S[u_0], 
 \]
 where $u_0$ is a minimiser of the functional~\eqref{limit}
 in~$\mathscr{A}$.
\end{mainthm}

The strategy of the proof for Theorem~\ref{th:Holder} 
is inspired by~\cite{contreras2018singular}.
Under the assumption~$F_\eps(u_\eps, \, B_1)\leq\eta$, we obtain
an algebraic decay for the mean oscillation of~$u_\eps$, that is
\begin{equation} \label{Campanato}
 \fint_{B_\rho} \abs{u_\eps - \fint_{B_\rho} u_\eps}^2 \leq C\rho^{2\mu}
\end{equation}
for any~$\rho\in (0, \, 1)$ and some positive constants~$C$, $\mu$
that do not depend on~$\rho$, $\eps$.
If the radius~$\rho$ is large enough, 
{\BBB i.e.~$\rho\geq\eps^\gamma$ for some suitable~$\gamma\in (0, \, 1)$,
we obtain an algebraic decay for~$F_\eps(u_\eps, \, B_\rho)$
as a function of~$\rho$ by adapting
analogous arguments for the limit functional~$E_0$
(cf.~Luckhaus' partial regularity results in~\cite{Luckhaus});} then,
we deduce~\eqref{Campanato} via a suitable Poincar\'e inequality
(Proposition~\ref{prop:Poincare}). On the other hand, 
{\BBB if~$\rho\leq\eps^\gamma$} we obtain~\eqref{Campanato} from
the Euler-Lagrange equations for~$E_\eps$ (Proposition~\ref{prop:EL}).
The inequality~\eqref{Campanato} immediately implies the desired 
bound on the H\"older norm of~$u_\eps$, by Campanato embedding.
Once Theorem~\ref{th:Holder} is proven, Theorem~\ref{th:conv}
follows, via the Ascoli-Arzel\`a theorem.

\begin{remark} \label{rk:assumptions}
 {\BBB As we observed before, if we are interested in weaker 
 modes of convergence for the minimisers (e.g., $L^2$-convergence),
 then we may replace~\eqref{hp:g_decay}
 and~\eqref{hp:nabla_K} with the weaker condition
 that~$g\in L^1(\R^3)$ and~$g$ has finite second moment,
 as in~\cite{taylor2018oseen}.
 However, \eqref{hp:g_decay} and~\eqref{hp:nabla_K}
 play a very important r\^ole for us;
 both of them are used in the proof of the estimate~\eqref{Campanato}
 for small radii, $\rho\leq\eps^\gamma$. We do not know whether 
 Theorems~\ref{th:Holder} and~\ref{th:conv} remain true under 
 weaker assumptions.}
\end{remark}

\section{Preliminary results}

\subsection{The Euler-Lagrange equations}

Throughout the paper, we denote by~$C$
several constants that depend only on~$\Omega$, $K$, $m$,
$\psi_s$ and~$u_{\mathrm{bd}}$. We write~$A\lesssim B$ 
as a short-hand for~$A\leq CB$. We also define
$g_\eps(z) := \eps^{-3}g(\eps^{-1}z)$ for~$z\in\R^3$
(where, we recall, $g(z)$ is the 
minimum eigenvalue of~$K(z)$) and
\begin{equation} \label{Lambda}
 \Lambda:= \nabla\psi_s\colon\QQ\to\R^m.
\end{equation}

\begin{proposition} \label{prop:EL}
Consider the free energy~$E_\eps$, given by~\eqref{energy}, 
with $u=u_{\mathrm{bd}}$ on $\mathbb{R}^3\setminus \Omega$.
Then there exists a minimiser~$u_\eps\in L^\infty(\Omega, \, \QQ)$
(identified with its extension by~$u_{\mathrm{bd}}$ to~$\mathbb{R}^3$), and it satisfies the Euler-Lagrange equation, 
\begin{equation} \label{EL}
 \Lambda(u_\eps(x)) = \int_{\mathbb{R}^3} K_\eps(x-y)u_\eps(y)\,\d y
\end{equation}
for a.e.~$x\in\Omega$. 
\end{proposition}

\begin{proof}
By neglecting the additive constant in~\eqref{energy},
and multiplying by~$\eps^2$,
without loss of generality we may consider the functional
\[
 \mathcal{F}(u) := \int_\Omega \psi_s(u(x))\,\d x
 - \int_{\mathbb{R}^3}\int_{\mathbb{R}^3} K_\eps(x-y)u(x)\cdot u(y)\,\d x\,\d y
\] 
instead of~$E_\eps$.
To show existence, we use a direct method argument.
First we show that the bilinear form admits a global lower bound. 
As $u_{\mathrm{bd}}\in L^2(\mathbb{R}^3, \, \QQ)$
and~$u$ admits uniform~$L^\infty$-bounds on~$\Omega$, we have that $u\in L^2(\mathbb{R}^3, \, \overline{\QQ})$, 
$\norm{u}_{L^2(\R^3)}$ is bounded uniformly. We thus have the estimate that 
\begin{equation}
 \begin{split}
 \int_{\mathbb{R}^3}\int_{\mathbb{R}^3} &|K_\eps(x-y)u(x)\cdot u(y)|\,\d x\,\d y\\
 &\lesssim \int_{\mathbb{R}^3}\int_{\mathbb{R}^3} g_\eps(x-y)|u(x)||u(y)|\,\d x\,\d y\\
 &=  \int_{\mathbb{R}^3}\int_{\mathbb{R}^3}
  \left(g_\eps(x-y)^{\frac{1}{2}}|u(x)|\right)
  \left(g_\eps(x-y)^{\frac{1}{2}}|u(y)|\right)\,\d x\,\d y\\
 &\lesssim \left(\int_{\mathbb{R}^3}\int_{\mathbb{R}^3}g_\eps(x-y)|u(x)|^2\,\d x\,\d y\right)^\frac{1}{2}\left(\int_{\mathbb{R}^3}\int_{\mathbb{R}^3}g_\eps(x-y)|u(y)|^2\,\d x\,\d y\right)^\frac{1}{2}\\
 &= \norm{g_\eps}_{L^1(\R^3)}\norm{u}_{L^2(\R^3)}^2
  = \norm{g}_{L^1(\R^3)}\norm{u}_{L^2(\R^3)}^2
\end{split}
\end{equation}
The singular function~$\psi_s$ admits a lower bound pointwise,
hence the functional~$\mathcal{F}$ admits a global lower bound. To show the admissible set is non empty, simply take $u(x)=u_0\in \QQ$ for all $x\in \Omega$, so that $\psi_s(u(x))$ is a non-infinite constant.

The uniform $L^\infty$ bounds on $u$ imply that we have $L^\infty$ weak-* compactness  of a minimising sequence. As~$\psi_s$ is strictly convex, we have weak-* lower semicontinuity of the entropic term. It suffices to show weak-* lower semicontinuity of the bilinear term. First we split the bilinear term into the ``boundary'' and ``bulk'' contributions. That is, we write 
$u=u_{\mathrm{bd}}\,\chi_{\mathbb{R}^3\setminus\Omega} + u\,\chi_{\Omega}$,
where~$\chi_{\R^3\setminus\Omega}$ and~$\chi_\Omega$ are the characteristic functions of~$\R^3\setminus\Omega$ and~$\Omega$ respectively. As $K_\eps*(u_{\mathrm{bd}} \, \chi_{\mathbb{R}^3\setminus\Omega})\in L^1(\Omega)$, 
if $u_j\overset{*}{\rightharpoonup}u$, 
\begin{equation}
 \int_\Omega u_j(x) \, K_\eps*(u_{\mathrm{bd}} \, \chi_{\mathbb{R}^3\setminus\Omega})(x)\,\d x
 \to \int_\Omega u(x) \, K_\eps*(u_{\mathrm{bd}} \, \chi_{\mathbb{R}^3\setminus\Omega})(x)\,\d x. 
\end{equation}
The second term requires a little more care. 
Following~\cite[Corollary 4.1]{eveson1995compactness}, 
the map $L^\infty(\Omega)\ni u\mapsto K_\eps*(u\,\chi_{\Omega})$ is $L^\infty$-to-$L^1$ compact 
if and only if the set $\left\{K_\eps(x-\cdot)\,\chi_\Omega\colon x\in\Omega\right\}$ is relatively $L^1$-compact. This is immediate however as $\Omega$ is a bounded set and $K_\eps$ is integrable. Therefore the map 
\begin{equation}
u\mapsto \int_\Omega\int_\Omega K_\eps(x-y)u(x)\cdot u(y)\,\d x\,\d y
\end{equation}
is in fact continuous with the weak-* $L^\infty$ topology, and therefore the entire bilinear term is continuous also. Therefore the energy functional is lower semicontinuous and minimisers exist by the direct method. 

To show that minimisers satisfy the Euler-Lagrange equation, we note that if $u$ has finite energy, then the measure of the set $\{x\in\Omega : u(x)\in\partial\QQ\}$ is zero. In particular, we may define $U_\delta=\left\{x\in\Omega :\psi_s(u(x))< 1/\delta\right\}$, and we have that 
\begin{equation}
 \Omega = \Gamma\cup \bigcup\limits_{\delta>0}U_\delta,
\end{equation}
where $\Gamma$ is a null set. 
By Assumption~\eqref{hp:blow-up}, for every $\delta>0$,
there exists some~$\gamma>0$ so that if
$\psi_s(\tilde{u})< 1/\delta$, then~$\dist(\tilde{u}, \, \partial\QQ)>\gamma$. 
In particular, for $\phi\in L^\infty(\mathbb{R}^3, \, \mathbb{R}^m)$ supported on~$U_\delta$
and~$\eta$ sufficiently small, $u+\eta\phi$ is bounded away from~$\partial\QQ$ on~$U_\delta$.
Therefore we may take variations without issue, as 
\begin{equation*}
 \begin{split}
 \frac{1}{\eta}\left(\mathcal{F}(u+\eta\phi)-\mathcal{F}(u)\right)
 &= \int_{U_\delta}\frac{1}{\eta}\left(\psi_s(u(x)+\eta\phi(x))-\psi_s(u(x))\right)\,\d x \\
 & -\frac{1}{2}\int_{\mathbb{R}^3}\int_{\mathbb{R}^3}K_\eps(x-y)\cdot\left(2\phi(x)\otimes u(y)+\eta\phi(x)\otimes \phi(y)\right)\,\d x\,\d y.
 \end{split}
\end{equation*}
Now we have no issue taking the limit as~$\eta\to 0$,
as~$\psi_s$ is~$C^2$ away from $\partial\QQ$, to give 
\begin{equation*}
 \begin{split}
 \lim \limits_{\eta\to 0}\frac{1}{\eta}\left(\mathcal{F}(u+\eta\phi)-\mathcal{F}(u)\right)
 =&\int_{U_\delta}\Lambda(u(x))\cdot \phi(x)\,\d x 
  - \int_{\mathbb{R}^3}\int_{\mathbb{R}^3}K_\eps(x-y)u(y)\,\d y\cdot\phi(x)\,\d x\\
 =&\int_{U_\delta}\left(\Lambda(u(x))-\int_{\mathbb{R}^3} K_\eps(x-y)u(y)\,\d y\right)\cdot \phi(x)\,\d x,
\end{split}
\end{equation*}
recalling that $\phi(x)=0$ outside of $U_\delta$. As $\phi$ was arbitrary, this implies that $u$ satisfies 
\begin{equation}
\Lambda(u(x))=\int_{\mathbb{R}^3} K_\eps(x-y)u(y)\,\d y
\end{equation}
on~$U_\delta$, and since $\delta$ was arbitrary, this implies that $u$ satisfies the Euler-Lagrange equation outside of $\Gamma$, which is of measure zero. 
\end{proof}

The Euler-Lagrange equations are particularly useful
when used in combination with the following property.

\begin{lemma} \label{lemma:nabla_inv}
 The map~$\Lambda\colon\QQ\to\R^m$ is invertible
 and its inverse is of class~$C^1$. Moreover,
 \begin{equation} \label{nabla_inv}
  \sup_{z\in\R^m} \norm{\nabla(\Lambda^{-1})(z)} \leq c^{-1},
 \end{equation}
 where~$c$ is the constant given by~\eqref{hp:unif_conv}, and
 \begin{equation} \label{Lambda_blowup}
  \abs{\Lambda(y)}\to +\infty \qquad 
  \textrm{as } \dist(y, \, \partial\QQ)\to 0.
 \end{equation}
\end{lemma}
\begin{proof}
 To prove~\eqref{Lambda_blowup}, it suffices to note that as $\psi_s$ is a closed proper convex function which is $C^1$ on an open domain, so by applying classical results from convex analysis \cite[Theorem 25.1, Theorem 26.1]{rockafellar1970convex}, we see that $\psi_s$ satisfies the property of {\it essential smoothness}, which implies~\eqref{Lambda_blowup}. More so, as $\psi_s$ is also strictly convex on a bounded domain, this implies $\psi_s$ is a Legendre-type function which provides the results that $\Lambda(\mathcal{Q})=\mathbb{R}^m$ \cite[Corollary 13.3.1]{rockafellar1970convex}, and that $\Lambda$ is a $C^0$ bijection from $\mathcal{Q}\to\Lambda(Q)$ \cite[Theorem 26.5]{rockafellar1970convex}. The $C^1$ regularity of $\Lambda^{-1}$ follows immediately from the inverse function theorem, as $\psi_s$ is strongly convex. 
\end{proof}

The Euler-Lagrange equation~\eqref{EL} and Lemma~\ref{lemma:nabla_inv}
have important consequences in terms of
regularity and ``strict physicality'' of minimisers --- that is,
the image of~$u_\eps$ does not touch the boundary of the 
physically admissible set~$\QQ$.

\begin{proposition} \label{prop:physicality}
 Minimisers~$u_\eps$ of the functional~$E_\eps$
 in the class~$\mathscr{A}$ are Lipschitz-continuous
 on~$\Omega$, with~$\norm{\nabla u_\eps}_{L^\infty(\Omega)}\lesssim\eps^{-1}$.
 Moreover, there exists a number~$\delta>0$ such that for any~$\eps>0$
 and any~$x\in\Omega$,
 \begin{equation} \label{physicality}
  \dist(u_\eps(x), \, \partial\QQ) \geq\delta.
 \end{equation}
\end{proposition}
\begin{proof}
 The minimiser~$u_\eps$ takes values in the bounded set~$\QQ$
 and hence, $\|u_\eps\|_{L^\infty(\R^3)}\leq C$,
 where the constant~$C$ does not depend on~$\eps$.
 Moreover, $\|K_\eps\|_{L^1(\R^3)} = \|K\|_{L^1(\R^3)} < +\infty$.
 Then, by applying Young's inequality to~\eqref{EL}, we obtain
 \[
  \norm{\Lambda(u_\eps)}_{L^\infty(\Omega)}
  \leq \|K_\eps\|_{L^1(\R^3)} \|u_\eps\|_{L^\infty(\R^3)} \leq C.
 \]
 On the other hand, we have~$\abs{\Lambda(z)}\to+\infty$
 as~$z\to\partial\QQ$ by~\eqref{Lambda_blowup} and hence,
 \eqref{physicality} follows.
 Since we have assumed that~$K\in W^{1,1}(\R^3, \, \text{Sym}(m))$,
 from the Euler-Lagrange equation~\eqref{EL} we deduce
 \[
  \norm{\nabla(\Lambda\circ u_\eps)}_{L^\infty(\Omega)} 
  = \norm{\nabla K_\eps * u_\eps}_{L^\infty(\Omega)} 
  \leq \eps^{-1} \norm{\nabla K}_{L^1(\Omega)} 
  \norm{u_\eps}_{L^\infty(\Omega)} < + \infty.
 \]
 By Lemma~\ref{lemma:nabla_inv}, we conclude 
 that~$\norm{\nabla u_\eps}_{L^\infty(\Omega)}\lesssim\eps^{-1}$.
\end{proof}

\subsection{A Poincar\'e-type inequality for~$F_\eps$}
\label{sect:Poincare}

The goal of this section is to prove the following inequality
on~$F_\eps$. We recall that the functional~$F_\eps$ 
is defined in~\eqref{loc_energy}.

\begin{proposition} \label{prop:Poincare}
 There exists~$\eps_1 > 0$ such that, for any~$u\in L^\infty(\R^3, \, \R^m)$,
 any~$\rho >0$, any~$x_0\in\R^3$ and any~$\eps\in (0, \, \eps_1\rho]$, there holds
 \[
  \fint_{B_{\rho/2}(x_0)} \abs{u - \fint_{B_{\rho/2}(x_0)}u}^2
  \lesssim \rho^{-1} F_\eps(u, \, B_\rho(x_0)).
 \]
\end{proposition}

To simplify the proof of Proposition~\ref{prop:Poincare},
we will take advantage of the scaling properties of~$F_\eps$: if
$u_\rho\colon B_1\to\R^m$ is defined by~$u_\rho(x):= u(\rho x + x_0)$
for~$x\in B_1$, then a change of variables gives
\begin{align}
 \rho^{-1} F_\eps(u, \, B_\rho(x_0)) &= 
  F_{\eps/\rho}(u_\rho, \, B_1) \label{scaling}
\end{align}
In the proof of Proposition~\ref{prop:Poincare},
we will adapt arguments from~\cite{taylor2018oseen}.
By assumption~\eqref{hp:g}, there exist positive numbers~$\rho_1 < \rho_2$,
$k$ such that $g\geq k$ a.e.~on~$B_{\rho_2}\setminus B_{\rho_1}$.
Let $\varphi\in C^\infty_\mathrm{c}(B_{\rho_2}\setminus B_{\rho_1})$
be a non-negative, radial function
(i.e.~$\varphi(z) = \tilde{\varphi}(\abs{z})$ for~$z\in\R^3$)
such that $\int_{\R^3}\varphi(z)\, \d z = 1$. 
Since~$g$ is bounded away from zero on the support of~$\varphi$, there holds
\[
 \varphi + \abs{\nabla\varphi} \leq C g
 \qquad \textrm{pointwise a.e.~on } \R^3,
\]
for some constant~$C$ that depends on~$g$ and~$\varphi$;
however, $\varphi$ is fixed once and for all, and so is~$C$.
We define~$\varphi_\eps(z) :=\eps^{-3}\varphi(\eps^{-1}z)$
for any~$z\in\R^3$ and~$\eps>0$. Then, $\varphi_\eps\in C^\infty_{\mathrm{c}}(\R^3)$
is non-negative, even, satisfies $\int_{\R^3}\varphi_\eps(z)\, \d z = 1$ and
\begin{equation} \label{mollify}
 \varphi_\eps + \eps \abs{\nabla\varphi_\eps} \leq C g_\eps
 \qquad \textrm{pointwise a.e.~on } \R^3.
\end{equation}

\begin{lemma} \label{lemma:mollify-Morrey}
 There exists~$\eps_2 > 0$ such that, for any~$u\in L^\infty(B_1, \, \R^m)$
 and any~$\eps\in (0, \, \eps_2]$, there holds
 \[
  \int_{B_{1/2}} \abs{\nabla(\varphi_\eps * u)}^2 \lesssim
  \eps^{-2} \int_{B_1\times B_1} K_\eps(x - y)\cdot 
  \left(u(x) - u(y)\right)^{\otimes 2} \d x\, \d y.
 \]
\end{lemma}
\begin{proof}
 We adapt the arguments from~\cite[Lemma~2.1 and 
 Proposition~2.1]{taylor2018oseen}. We define
 \[
  I(y, \, z) := \int_{B_{1/2}}
   \nabla\varphi_\eps(x-y)\cdot \nabla\varphi_\eps(x-z)
   \, \d x \qquad \textrm{for } y, \, z\in\R^3.
 \]
 We express the gradient of~$\varphi_\eps* u$ 
 as~$\nabla(\varphi_\eps * u) = (\nabla\varphi_\eps) * u$.
 By applying the identity $2a\cdot b = -\abs{a - b}^2 + \abs{a}^2 + \abs{b}^2$, we obtain
 \[
  \begin{split}
   \int_{B_{1/2}} \abs{\nabla (\varphi_\eps * u)(x)}^2\d x
   &= \int_{\R^3\times\R^3} u(y)\cdot u(z) \, I(y, \, z) \, \d y \, \d z \\
   &= \underbrace{-\frac{1}{2}\int_{\R^3\times\R^3} \abs{u(y) - u(z)}^2 
   I(y, \, z) \, \d y\, \d z}_{=: I_1} \\
   &+ \underbrace{\frac{1}{2}\int_{\R^3\times\R^3} \abs{u(y)}^2
   I(y, \, z) \, \d y\, \d z}_{=: I_2} 
   + \underbrace{\frac{1}{2}\int_{\R^3\times\R^3} \abs{u(z)}^2
   I(y, \, z) \, \d y\, \d z}_{=: I_3}
  \end{split}
 \]
 We first consider the term~$I_2$. Since~$\varphi_\eps$ is compactly 
 supported, we have $\int_{\R^3}\nabla\varphi_\eps(z)\,\d z = 0$. Therefore, 
 \[
  I_2 = \frac{1}{2}\int_{B_{1/2}\times\R^3} 
  \abs{u(y)}^2 \nabla \varphi_\eps(x-y)\cdot \left(
  \int_{\R^3} \nabla \varphi_\eps(x-z) \, \d z \right) \d x\, \d y = 0,
 \]
 and likewise~$I_3 = 0$. Now, we consider~$I_1$. 
 The gradient~$\nabla\varphi_\eps$ is supported
 in a ball of radius~$C\eps$, where $C$ is an~$\eps$-independent constant.
 This implies
 \[
  \begin{split}
   I_1 
   &= \frac{1}{2}\int_{B_{1/2 + C\eps}\times B_{1/2 + C\eps}}
    \abs{u(y) - u(z)}^2 \left( \int_{B_{1/2}} \nabla\varphi_\eps(x - y)
    \cdot \nabla\varphi_\eps(x - z) \, \d x \right) \, \d y\, \d z \\
   &\leq \int_{B_{1/2}\times B_{1/2 + C\eps}\times B_{1/2 + C\eps}}
    \abs{u(y) - u(x)}^2 \abs{\nabla\varphi_\eps(x - y)} 
    \abs{\nabla\varphi_\eps(x - z)} \, \d x \, \d y\, \d z \\
   &\qquad + \int_{B_{1/2}\times B_{1/2 + C\eps}\times B_{1/2 + C\eps}}
    \abs{u(x) - u(z)}^2 \abs{\nabla\varphi_\eps(x - y)} 
    \abs{\nabla\varphi_\eps(x - z)} \, \d x \, \d y\, \d z \\
   &\leq 2\norm{\nabla\varphi_\eps}_{L^1(\R^3)}
    \int_{B_{1/2}\times B_{1/2 + C\eps}} \abs{u(y) - u(x)}^2 
    \abs{\nabla\varphi_\eps(y - x)}\d x\, \d y
  \end{split}
 \]
 Thanks to~\eqref{mollify}, we obtain
 \[
  \begin{split}
   I_1 \lesssim \eps^{-2} \norm{g}_{L^1(\R^3)}
    \int_{B_{1/2}\times B_{1/2 + C\eps}} 
    \abs{u(y) - u(x)}^2 g_\eps(y - x) \, \d x\, \d y.
  \end{split}
 \]
 For~$\eps$ sufficiently small we have~$1/2 + C\eps<1$,
 and the lemma follows. 
\end{proof}

Given two sets~$A\subseteq\R^3$, $A^\prime\subseteq\R^3$,
we write~$A\csubset A^\prime$ when the \emph{closure}
of~$A$ is contained in~$A^\prime$.

\begin{lemma} \label{lemma:mollify-L2}
 Let~$A$, $A^\prime$ be open sets such 
 that $A\csubset A^\prime\subseteq\R^{3}$.
 Then, there exists~$\eps_3 = \eps_3(A, \, A^\prime)$
 such that, for any~$u\in L^\infty(A^\prime, \, \R^m)$
 and any~$\eps\in (0, \, \eps_3]$, there holds
 \[
  \int_{A} \abs{u - \varphi_\eps * u}^2 \lesssim
  \int_{A^\prime\times A^\prime} 
  K_\eps(x - y)\cdot \left(u(x) - u(y)\right)^{\otimes 2} \d x\, \d y.
 \]
\end{lemma}
\begin{proof}
 Since~$\int_{\R^3}\varphi_\eps(z)\,\d z = 1$, we have
 \[
  I := \int_{A} \abs{u(x) - (\varphi_\eps*u)(x)}^2\d x
  = \int_{A} \abs{\int_{\R^3} \varphi_\eps(x-y) \left(u(x) - u(y)\right) \d y}^2 \d x.
 \]
 We apply Jensen inequality with respect to the 
 probability measure $\varphi_\eps(x-y)\, \d y$:
 \[
  I \leq \int_{A} \left(\int_{\R^3} 
  \varphi_\eps(x-y) \abs{u(x) - u(y)}^2 \d y\right) \d x .
 \]
 Because the support of~$\varphi_\eps$ is contained in a ball 
 of radius~$C\eps$, where~$C$ is an~$\eps$-independent constant, 
 the integrand is equal to zero if~$x\in A$, $\dist(y, \, A) > C\eps$. 
 By applying~\eqref{mollify}, we obtain
 \[
  I \leq \int_{A\times \{y\in\R^3\colon \dist(y, \, A)\leq C\eps\}} 
  g_\eps(x-y) \abs{u(x) - u(y)}^2 \d x \, \d y
 \]
 and, if~$\eps\leq C^{-1}\dist(A, \, \partial A^\prime)$, 
 the lemma follows.
\end{proof}

\begin{proof}[Proof of Proposition~\ref{prop:Poincare}]
 Due to the scaling property~\eqref{scaling}, it suffices to prove that
 \begin{equation} \label{Poincare1}
  \fint_{B_{1/2}} \abs{u - \fint_{B_{1/2}} u}^2
  \lesssim F_{\eps/\rho}(u, \, B_1)
 \end{equation}
 for any~$u\in L^\infty(\R^3, \, \R^m)$ 
 and any~$\eps$, $\rho$ with~$\eps/\rho$ 
 sufficiently small. The triangle inequality
 and the elementary inequality
 $(a + b + c)^2 \leq 3(a^2 + b^2 + c^2)$ imply
 \[
  \fint_{B_{1/2}} \abs{u - \fint_{B_{1/2}} u}^2 \leq
  6 \fint_{B_{1/2}} \abs{u - \varphi_{\eps/\rho} * u}^2 
  + 3 \fint_{B_{1/2}} \abs{\varphi_{\eps/\rho} * u 
    - \fint_{B_{1/2}}\varphi_{\eps/\rho} * u}^2
 \]
 Thanks to the Poincar\'e inequality, we obtain
 \[
  \fint_{B_{1/2}} \abs{u - \fint_{B_{1/2}} u}^2 \lesssim
  \int_{B_{1/2}} \abs{u - \varphi_{\eps/\rho} * u}^2 
  + \int_{B_{1/2}} \abs{\nabla(\varphi_{\eps/\rho} * u)}^2 .
 \]
 If $\eps/\rho$ is sufficiently small, 
 Lemma~\ref{lemma:mollify-Morrey} and Lemma~\ref{lemma:mollify-L2} 
 give
 \[
  \fint_{B_{1/2}} \abs{u - \fint_{B_{1/2}} u}^2 \lesssim
  \left((\eps/\rho)^2 + 1 \right) F_{\eps/\rho}(u, \, B_1),
 \]
 so~\eqref{Poincare1} follows.
\end{proof}

\subsection{Localised $\Gamma$-convergence
for the non-local term}

The $\Gamma$-convergence of the functional~$E_\eps$, as~$\eps\to 0$,
was studied in~\cite{taylor2018oseen}. In this section, 
we adapt the arguments of~\cite{taylor2018oseen} to prove a
localised $\Gamma$-convergence result. We focus on the 
interaction part of the free energy only, since this is all we
need in the proof of Theorem~\ref{th:Holder}.
{\BBB We denote by $F_\eps^\nl$ the non-local interaction 
part of~$F_\eps$, given by
\begin{equation} \label{locintenergy}
 \begin{split}
  F^\nl_\eps(u, \, G) :=& F_\eps(u, \, G) - 
   \frac{1}{\eps^2}\int_G \psi_b(u(x)) \, \d x \\
   =& \frac{1}{4\eps^2} \int_{G\times G} K_\eps(x - y)
   \cdot \left(u(x) - u(y)\right)^{\otimes 2} \,\d x \, \d y
 \end{split}
\end{equation}
for any~$u\in L^\infty(\R^3, \, \QQ)$ and any Borel set~$G\subseteq\R^3$.}

\begin{proposition} \label{prop:Gamma-liminf}
 {\BBB Let~$\rho>0$, $x_0\in\R^3$, 
 and let~$v_\eps\in L^2(B_\rho(x_0), \, \R^m)$, $v_0\in H^1(B_\rho(x_0), \, \R^m)$
 be such that $v_\eps\to v_0$ strongly in $L^2(B_\rho(x_0))$ 
 as~$\eps\to 0$. Then, for any open set~$G\subseteq B_{\rho}(x_0)$ we have
 \begin{equation} \label{liminf}
  \int_{G} L\nabla v_0\cdot\nabla v_0
  \leq \liminf_{\eps\to 0} F_\eps^\nl(v_\eps, \, G)
 \end{equation} }
\end{proposition}

\begin{proposition} \label{prop:Gamma-limsup}
 {\BBB Let~$\rho>0$. $x_0\in\R^3$. Let
 $v_\eps\in H^1(B_\rho(x_0), \, \R^m)$, $v_0\in H^1(B_\rho(x_0), \, \R^m)$
 be such that $v_\eps\to v_0$ strongly in $H^1(B_\rho(x_0))$
 as~$\eps\to 0$. Then
 \begin{equation*}
  \limsup\limits_{\eps\to 0} F_\eps^\nl(v_\eps, \, B_\rho(x_0))
  \leq \int_{B_\rho(x_0)} L\nabla v_0\cdot \nabla v_0
 \end{equation*} }
\end{proposition}

In the proofs of Proposition~\ref{prop:Gamma-liminf}
and~\ref{prop:Gamma-limsup}, we will use the following notation.
Given a vector~$w\in\R^3\setminus\{0\}$ and a 
function~$u$ defined on a subset of~$\R^3$, we define
the difference quotient
\[
  D_w u(x) := \frac{u(x+w) - u(x)}{\abs{w}}
\]
for any $x$ in the domain of~$u$ such that $x+w$ belongs to the domain of~$u$.
If~$\abs{w}\leq h$, $u\in H^1(B_{\rho+h})$, 
and $|\cdot|_*$ is any seminorm on~$\R^{m}$, then
\begin{equation} \label{diffq}
  \int_{B_\rho} |D_{\eps w}u(x)|_*^2\, \d x 
  \leq \int_{B_{\rho+\eps h}}|(\hat{w}\cdot\nabla)u(x)|_*^2\,\d x
\end{equation} 
where~$\hat{w} := w/\abs{w}$. This follows from the same technique
as, e.g.,~\cite[Lemma 7.23]{gilbarg2015elliptic}, for the case  
in which we have the standard 
Euclidean norm. However, we realise the proof only relies
on the convexity of the seminorm, and no further structure.
For convenience, we give the proof
of Proposition~\ref{prop:Gamma-limsup} first.

\begin{proof}[Proof of Proposition~\ref{prop:Gamma-limsup}]
 We assume that~$x_0 = 0$.
 Using a reflection across the boundary of~$B_\rho$ and a cut-off function, we 
 define~$v_\eps$ and~$v_0$ on~$\R^3\setminus B_\rho$, 
 in such a way that $v_\eps\in H^1(\R^3, \, \R^m)$, 
 $v_0\in H^1(\R^3, \, \R^m)$ and 
 $v_\eps\to v_0$ strongly in~$H^1(\R^3)$.
 Let~$t>0$ be a parameter. We have
 \begin{equation*} 
  \begin{split}
   &\frac{1}{4\eps^2} \int_{B_\rho}\int_{B_\rho}
   K_\eps(x-y)\cdot\left(v_\eps(x)-v_\eps(y)\right)^{\otimes 2}\,\d x\,\d y\\
   &\leq \frac{1}{4}\int_{B_\rho}\int_{\mathbb{R}^3}|z|^2K(z)\cdot\left(D_{\eps z}v_\eps(x)\right)^{\otimes 2}\,\d z \,\d x\\
   &= \frac{1}{4}\int_{B_\rho}\int_{B_\frac{t}{\eps}}|z|^2K(z)\cdot\left(D_{\eps z}v_\eps(x)\right)^{\otimes 2}\,\d z \,\d x
   + \frac{1}{4} \int_{B_\rho}\int_{\mathbb{R}^3\setminus B_\frac{t}{\eps}}|z|^2K(z)\cdot\left(D_{\eps z}v_\eps(x)\right)^{\otimes 2}\,\d z\,\d x.
  \end{split}
 \end{equation*}
 To estimate the first integral at the right-hand side, 
 we exchange the order of integration and, for any~$z$,
 we apply~\eqref{diffq} to the
 seminorm~$\abs{\xi}_*^2 := \abs{z}^2K(z) \cdot \xi^{\otimes 2}$;
 for the second integral, we apply~\eqref{hp:lambda_max}:
 \begin{equation} \label{DSJ1}
  \begin{split}
   &\frac{1}{4\eps^2} \int_{B_\rho}\int_{B_\rho}
   K_\eps(x-y)\cdot\left(v_\eps(x)-v_\eps(y)\right)^{\otimes 2}\,\d x\,\d y\\
   &\leq \frac{1}{4} \int_{\mathbb{R}^3}\int_{B_{\rho+t}}K(z)\cdot\left((z\cdot \nabla) v_\eps(x)\right)^{\otimes 2}\,\d x\,\d  z
   +C\int_{B_\rho}\int_{\mathbb{R}^3\setminus B_\frac{t}{\eps}}g(z)|z|^2|D_{\eps z}v_\eps(x)|^2\,\d z\,\d x\\
   &\stackrel{\eqref{L}}{=} 
    \int_{B_{\rho+t}}L\nabla v_\eps (x)\cdot \nabla v_\eps(x) \,\d x + C\int_{B_\rho}\int_{\mathbb{R}^3\setminus B_\frac{t}{\eps}}g(z)|z|^2|D_{\eps z}v_\eps(x)|^2\,\d z\,\d x.
  \end{split}
 \end{equation}
 We now estimate the latter summand independently. 
 {\BBB For $z \in \R^3\setminus B_\frac{t}{\eps}$, $|\eps z|^2>t^2$, so 
 \[
  |D_{\eps z}v_\eps(x)|^2
  \leq \frac{1}{t^2} \abs{v_\eps(x+\eps z) - v_\eps(x)}^2 
  \leq \frac{2}{t^2} \left(\abs{v_\eps(x+\eps z)}^2 + \abs{v_\eps(x)}^2\right)
 \] }
 Therefore, by applying Fubini theorem, we may estimate
 \begin{equation} \label{DSJ2}
  \begin{split}
  \int_{B_\rho}\int_{\mathbb{R}^3\setminus B_\frac{t}{\eps}}g(z)|z|^2|D_{\eps z}v_\eps(x)|^2\,\d z\,\d x
  &\leq \frac{{\BBB 4\norm{v_\eps}_{L^2(\R^3)}^2}}{t^2}\int_{\mathbb{R}^3\setminus B_\frac{t}{\eps}}g(z)|z|^2\,\d z
  \end{split}
 \end{equation}
 As $g$ has finite second moment {\BBB and 
 $\norm{v_\eps}_{L^2(\R^3)}\leq C$}, for fixed $t$ we must have that 
 \begin{equation} \label{DSJ3}
  \lim\limits_{\eps \to 0}\frac{{\BBB 2\norm{v_\eps}^2_{L^2(\R^3)}}}{t^2}\int_{\mathbb{R}^3\setminus B_\frac{t}{\eps}}g(z)|z|^2\,\d z=0. 
 \end{equation}
 Combining~\eqref{DSJ1}, \eqref{DSJ2} and~\eqref{DSJ3} gives
 \begin{equation*}
  \limsup\limits_{\eps \to 0} \frac{1}{4\eps^2} 
  \int_{B_\rho}\int_{B_\rho} K_\eps(x-y) \cdot\left(v_\eps(x)-v_\eps(y)\right)^{\otimes 2}\,\d y\,\d x
  \leq \limsup\limits_{\eps\to 0} 
  \int_{B_{\rho+t}} L\nabla v_\eps(x) \cdot \nabla v_\eps(x)\,\d x.
 \end{equation*}
 As $v_\eps\to v_0$ in $H^1(\R^3)$, 
 this implies 
 \begin{equation*}
  \lim\limits_{\eps\to 0} \int_{B_{\rho+t}}
  L\nabla v_\eps(x) \cdot \nabla v_\eps(x)\,\d x
  = \int_{B_{\rho+t}} L\nabla v_{0}(x) \cdot \nabla v_{0}(x)\,\d x.
 \end{equation*}
 Therefore we have 
 \begin{equation*}
  \limsup\limits_{\eps \to 0} \frac{1}{4\eps^2}
  \int_{B_\rho}\int_{B_\rho}K_\eps(x-y)\cdot\left(v_\eps(x)-v_\eps(y)\right)^{\otimes 2}\,\d y\,\d x
  \leq \int_{B_{\rho+t}} L\nabla v_{0}(x) \cdot \nabla v_{0}(x)\,\d x,
 \end{equation*}
 and passing to the limit as~$t\to 0$ in the right-hand side
 gives the desired result. 
\end{proof}

\begin{proof}[Proof of Proposition~\ref{prop:Gamma-liminf}]
 Again, we assume that~$x_0 = 0$.
 Without loss of generality, we may assume that
 \begin{equation} \label{DSJ0}
  \liminf_{\eps\to 0} 
  \frac{1}{4\eps^2} \int_{G\times G} K_\eps(x - y)
  \cdot \left(v_\eps(x) - v_\eps(y)\right)^{\otimes 2} \,\d x \, \d y < + \infty,
 \end{equation}
 otherwise there is nothing to prove. Up to extraction of a (non-relabelled)
 subsequence, we may also assume that the
 left-hand side of~\eqref{DSJ0} is actually a limit.
 Let $G\subseteq B_{\rho}$ be open and~$G^\prime\csubset G$. 
 Then we may write that 
 \begin{equation} \label{DSJ4}
  \begin{split}
   \frac{1}{4\eps^2} \int_G\int_G & K_\eps(x-y)\cdot \left(v_\eps(x)-v_\eps(y)\right)^{\otimes 2}\,\d y\,\d x\\
   &\geq \frac{1}{4\eps^2} \int_{G^\prime}\int_{G} K_\eps(x-y)\cdot \left(v_\eps(x)-v_\eps(y)\right)^{\otimes 2}\,\d y\,\d x\\
   &= \frac{1}{4} \int_{G^\prime}\int_{\frac{G-x}{\eps}} |z|^2K(z)\cdot \left(D_{\eps z}v_\eps(x)\right)^{\otimes 2}\,\d z\,\d x.
  \end{split}
 \end{equation}
 Let~$G^c := \R^3\setminus G$ and~$\delta := \dist(G^\prime, \, G^c)>0$. 
 We note that 
 \begin{equation*} 
  \begin{split}
   \left|\int_{G^\prime}\int_{\left(\frac{G-x}{\eps}\right)^c}
    |z|^2K(z)\cdot \left(D_{\eps z}v_\eps(x)\right)^{\otimes 2}\,\d y\,\d x\right|
   \lesssim \int_{G^\prime}\int_{B_{\frac{\delta}{\eps}}^c}g(z)|z|^2|D_{\eps z}v_\eps(x)|^2\,\d z\,\d x,
  \end{split}
 \end{equation*}
 which by previous estimates (see~\eqref{DSJ2}, \eqref{DSJ3}) 
 we have seen converges to zero as~$\eps\to 0$. This means 
 \begin{equation} \label{DSJ5}
  \begin{split}
   \liminf\limits_{\eps\to 0}&\int_{G^\prime}\int_{\frac{G-x}{\eps}} |z|^2K(z)\cdot \left(D_{\eps z}v_\eps(x)\right)^{\otimes 2}\,\d z\,\d x\\
   &=\liminf\limits_{\eps\to 0}\int_{G^\prime}\int_{\mathbb{R}^3} |z|^2K(z)\cdot \left(D_{\eps z}v_\eps(x)\right)^{\otimes 2}\,\d z\,\d x\\
  \end{split}
 \end{equation}
 Furthermore, we note this can be written as an $L^2$ norm, by defining $w_\eps: G^\prime\times\mathbb{R}^3\to\mathbb{R}^m$ by $w_\eps(z,x):=|z|K^\frac{1}{2}(z)D_{\eps z}v_\eps(x)$. Thanks to~\eqref{DSJ0}, we immediately see that $w_\eps$ is $L^2$-bounded, so must admit an $L^2$-weakly converging subsequence $w_j:=w_{\eps_j}$ with $\eps_j\to 0$ and $w_j$ has weak-$L^2$ limit $w_0$. Furthermore, we take 
 \begin{equation} \label{DSJ6}
  \liminf\limits_{\eps\to 0}\int_{G^\prime}\int_{\mathbb{R}^3} |z|^2K(z)\cdot \left(D_{\eps z}v_\eps(x)\right)^{\otimes 2}\,\d z\,\d x
  =\liminf\limits_{j\to\infty} \norm{w_j}^2_{L^2(G^\prime\times\mathbb{R}^3)}
  \geq \norm{w_0}^2_{L^2(G^\prime\times\mathbb{R}^3)} \! .
 \end{equation}
 It remains to identify the limit~$w_0$.
 We may do this by integrating against test functions.
 Let~$\phi\in C^\infty_{\mathrm{c}}(G^\prime\times \mathbb{R}^3)$.
 There exists some~$R_0>0$ such that, for any~$(y, \, z)\in\R^3\times\R^3$
 with~$|z|>R_0$, $\phi(y, \, z)=0$. 
 Furthermore, there exists some $\delta>0$ so that if 
 $\dist(y, \, (G^\prime)^c)<\delta$, then~$\phi(y, \, z)=0$.
 In particular, if~${\eps_j}< \frac{\delta}{R_0}$
 and~$(x - \eps_j z, \, z)\in\mathrm{supp}(\phi)$, then~$x\in G^\prime$.
 Therefore 
 \begin{equation*}
  \begin{split}
   \langle w_j, \, \phi\rangle
   &=\int_{G^\prime}\int_{\mathbb{R}^3}\phi(x, \, z)|z|K^\frac{1}{2}(z)D_{{\eps_j} z}v_{\eps_j}(x)\,\d z\,\d x\\
   &= \frac{1}{{\eps_j}}\int_{G^\prime}\int_{\mathbb{R}^3}\Big(\phi(x-{\eps_j} z, \, z)-\phi(x, \, z)\Big)K^\frac{1}{2}(z)v_{\eps_j}(x)\,\d z\,\d x,
  \end{split}
 \end{equation*}
 and we may exploit the fact that
 \[
  \frac{1}{{\eps_j}}\Big(\phi(x-{\eps_j} z, \, z)-\phi(x, \, z)\Big)
  \to (-z\cdot \nabla_x )\phi(x, \, z)
  \qquad \textrm{uniformly on } G^\prime\times\R^3 \textrm{ as } j\to+\infty,
 \]
 with the assumed $L^2$ convergence of $v_{\eps_j}\to v_0$, to give that 
 \begin{equation*}
  \begin{split}
   \lim\limits_{j\to\infty} \langle w_j, \, \phi\rangle
   &= \lim\limits_{j\to\infty}\frac{1}{\eps_j}\int_{G^\prime}\int_{\mathbb{R}^3}\Big(\phi(x-{\eps_j} z, \, z)-\phi(x, \, z)\Big)K^\frac{1}{2}(z)v_{\eps_j}(x)\,\d z\,\d x,\\
   &=\int_{G^\prime}\int_{\mathbb{R}^3}(-z\cdot \nabla_x )\phi(x, \, z)K^\frac{1}{2}(z)v_0(x)\,\d z\,\d x\\
   &=\int_{G^\prime}\int_{\mathbb{R}^3}\phi(x, \, z)K^\frac{1}{2}(z)(z\cdot \nabla )v_0(x)\,\d z\,\d x=\langle w_0, \, \phi\rangle.
  \end{split}
 \end{equation*}
 Therefore $w_0(x, \, z)=K^\frac{1}{2}(z)(z\cdot \nabla )v_0(x)$,
 and by~\eqref{DSJ4}, \eqref{DSJ5}, \eqref{DSJ6} we have
 \begin{equation*}
  \begin{split}
   \liminf\limits_{\eps\to 0} 
   &\frac{1}{4\eps^2}\int_G\int_G K_\eps(x-y)\cdot \left(v_\eps(x)-v_\eps(y)\right)^{\otimes 2}\,\d y\,\d x\\
   &\geq \frac{1}{4} \liminf\limits_{j\to\infty} \norm{w_j}^2_{L^2(G^\prime\times \mathbb{R}^3)}\\
   &\geq \frac{1}{4}\norm{w_0}^2_{L^2(G^\prime\times \mathbb{R}^3)}\\
   &= \frac{1}{4}\int_{G^\prime}\int_{\mathbb{R}^3}K(z)\cdot\Big((z\cdot \nabla )v_0(x)\Big)^{\otimes 2}\,\d z\,\d x\\
   &\stackrel{\eqref{L}}{=} \int_{G^\prime}L\nabla v_0(x)\cdot \nabla v_0(x)\,\d x.
  \end{split}
 \end{equation*}
 As the set $G^\prime\csubset G$ was arbitrary,
 by monotonicity the lower bound~\eqref{liminf} holds.
\end{proof}

\subsection{{\BBB Other auxiliary results}}

{\BBB In this section, we collect some auxiliary results that will be useful 
in the proof of Theorem~\ref{th:Holder}. Our first result is a 
remark on the kernel~$K$, which will be used repeatedly.
We will use the notation~$a_\eps = \o(b_\eps)$
if there exists a positive sequence~$c_\eps$, depending on~$K$ and~$\QQ$
only, such that~$\abs{a_\eps} \leq c_\eps \abs{b_\eps}$ and~$c_\eps\to 0$
as~$\eps\to 0$.

\begin{lemma} \label{lemma:decayK}
 For any~$\sigma>0$, there holds
 \[
  \sup_{x\in\R^3}\left(\frac{1}{\eps^2}
  \int_{\{y\in\R^3\colon\abs{x - y}\geq\sigma\}}
  g_\eps(x - y) \, \d y\right)
  = \o\!\left(\frac{\eps^{q-2}}{\sigma^2}\right)
 \]
 where~$q>7/2$ is the number given by Assumption~\eqref{hp:g_decay}.
\end{lemma}
\begin{proof}
 By the change of variable~$y = x + \eps z$, we obtain
 \[
  \begin{split}
   \frac{1}{\eps^2} \int_{\{y\in\R^3\colon\abs{x - y}\geq\sigma\}}
    g_\eps(x - y) \, \d y
   &= \frac{1}{\eps^2} \int_{\{z\in\R^3\colon\abs{z}\geq \sigma/\eps\}}
    g(z) \, \d z \\
   &\leq \frac{\eps^{q-2}}{\sigma^q}
   \int_{\{z\in\R^3\colon\abs{z}\geq \sigma/\eps\}}
    g(z) \abs{z}^q \, \d z
  \end{split}
 \]
 By assumption~\eqref{hp:g_decay}, $g$ has finite moment of order~$q$,
 so the lemma follows.
\end{proof}

Given a function~$u\in L^\infty(\R^3, \, \QQ)$ and
Borel sets~$G\subseteq\R^3$, $G^\prime\subseteq\R^3$, we define
\begin{equation} \label{locdefect}
 \Gamma_\eps(u, \, G, \, G^\prime) := 
  \frac{1}{2\eps^2}\int_{G\times G^\prime} 
     K_\eps(x-y) \cdot \left(u(x) - u(y)\right)^{\otimes 2} \, \d x \, \d y 
\end{equation}
In the terminology introduced by Alberti and
Bellettini~\cite{AlbertiBellettini},
$\Gamma_\eps$ is called the `locality defect'.
Indeed, if~$G$ and~$G^\prime$ are disjoint
and~$F^\nl_\eps$ is defined as in~\eqref{locintenergy}, then
\begin{equation} \label{locdefsplit}
 F^\nl_\eps(u, \, G \cup G^\prime) = F^\nl_\eps(u, \, G) 
 + F^\nl_\eps(u, \, \, G^\prime)
 + \Gamma_\eps(u, \, G, \, G^\prime)
\end{equation}
because~$K$ is assumed to be an even function (by~\eqref{hp:K_even}).
Moreover, for any~$G\subseteq\R^3$, $G^\prime\subseteq\R^3$ we have
$\Gamma_\eps(u, \, G, \, G^\prime) = \Gamma_\eps(u, \, G^\prime, \, G)$.
Given a set~$G\subseteq\R^3$ and a number~$\sigma>0$,
we define
\begin{equation} \label{partialsigma}
 \partial_\sigma G := \left\{x\in\Omega\colon 
  \dist(x, \, \partial G)<\sigma\right\} \!.
\end{equation}
Our next result is an estimate for the `locality defect'~$\Gamma_\eps$.}

\begin{lemma} \label{lemma:locdefect}
 {\BBB Let~$u\in L^\infty(\R^3, \, \QQ)$, let~$G\subseteq G^\prime \subseteq \R^3$
 be bounded Borel sets, and let~$\sigma>0$. Then,
 \begin{equation*} \label{glue0}
  \Gamma_\eps(u, \, G, \, G^\prime\setminus G) 
   \lesssim F_\eps^\nl(u, \, \partial_\sigma G)
    + \o\!\left(\frac{\eps^{q-2}}{\sigma^q}\right) \inf_{\zeta\in\R^m}
    \norm{u - \zeta}^2_{L^2(G^\prime)}
 \end{equation*}
 where~$q>7/2$ is given by~\eqref{hp:g_decay}.}
\end{lemma}
\begin{proof}
 {\BBB If two points~$x\in G$, $y\in G^\prime\setminus G$ 
 satisfy~$\abs{x - y}\leq\sigma$, then necessarily~$x\in\partial_\sigma G$,
 $y\in\partial_\sigma G$. Then, 
 \[
  \begin{split}
   \Gamma_\eps(u, \, G, \, G^\prime\setminus G) 
   &\leq 2F^\nl_\eps(u, \, \partial_\sigma G) \\
   &\qquad + \underbrace{\frac{1}{2\eps^2}
   \int_{\{x\in G^\prime, \, y\in G^\prime\colon
   \abs{x - y} \geq \sigma\}} K_\eps(x-y) \cdot
   \left(u(x) - u(y)\right)^{\otimes 2} \, \d x \, \d y}_{=: I}
  \end{split}
 \]
 We need to estimate the term~$I$. Let~$\zeta\in\R^m$ be a constant.
 The inequality $\abs{u(x) - u(y)}^2 \leq 2\abs{u(x) - \zeta}^2 
 + 2\abs{u(y) - \zeta}^2$, and the assumption that~$K$ is even
 (see~\eqref{hp:K_even}), imply
 \begin{equation} \label{locdefect1}
  \begin{split}
   I&\lesssim \frac{1}{2\eps^2} 
    \int_{\{x\in G^\prime, \, y\in G^\prime\colon
   \abs{x - y} \geq \sigma\}}
    g_\eps(x - y)\abs{u(x) - u(y)}^2 \d x \, \d y \\
    &\lesssim \frac{2}{\eps^2} 
    \int_{\{x\in G^\prime, \, y\in G^\prime\colon
   \abs{x - y} \geq \sigma\}}
    g_\eps(x - y)\abs{u(x) - \zeta}^2 \d x \, \d y
  \end{split}
 \end{equation}
 Then, by Lemma~\ref{lemma:decayK}, we obtain
 \[
  I \lesssim \o\!\left(\frac{\eps^{q-2}}{\sigma^q}\right) \inf_{\zeta\in\R^m}
    \norm{u - \zeta}^2_{L^2(G^\prime)}
 \]
 and the lemma follows.}
\end{proof}

\begin{lemma} \label{lemma:diffGamma}
 {\BBB Let~$u\in L^\infty(\R^3, \, \QQ)$, $\xi\in L^\infty(\R^3, \, \QQ)$,
 let~$G\subseteq \R^3$ be a bounded Borel set such that~$u=\xi$ 
 a.e. in~$\R^3\setminus G$, and let~$\sigma>0$. Then,
 \begin{equation*} 
  \begin{split}
   \Gamma_\eps(\xi, \, G, \, \R^3\setminus G) 
   - \Gamma_\eps(u, \, G, \, \R^3\setminus G) 
    \lesssim F_\eps^\nl(\xi, \, \partial_\sigma G)
     + \o\!\left(\frac{\eps^{q-2}}{\sigma^q}\right)
     \abs{G}^{1/2} \norm{\xi - u}_{L^2(G)}
  \end{split}  
 \end{equation*}
 where~$q>7/2$ is given by~\eqref{hp:g_decay}.}
\end{lemma}

\begin{remark} \label{rk:diffGamma}
 {\BBB The right-hand side of Lemma~\ref{lemma:diffGamma}
 contains the $L^2$-norm of~$\xi - u$, \emph{not} the $L^2$-norm squared.
 This loss of a power will be responsible for 
 additional technicalities later on in the proof 
 (see Lemma~\ref{lemma:decay} below). }
\end{remark}

\begin{proof}[Proof of Lemma~\ref{lemma:diffGamma}]
 {\BBB Let~$H_\sigma := \{(x, \, y)\in G\times(\R^3\setminus G)
 \colon \abs{x-y}\geq\sigma\}$. We have
 \[
  G\times (\R^3\setminus G)\subseteq 
   (\partial_\sigma G\times \partial_\sigma G)\cup H_\sigma
 \]
 and hence,
 \[
  \begin{split}
   &\Gamma_\eps(\xi, \, G, \, \R^3\setminus G)
    - \Gamma_\eps(u, \, G, \, \R^3\setminus G) \\
   &\qquad \leq 2F_\eps^\nl(\xi, \, \partial_\sigma G)
    + \underbrace{\frac{1}{2\eps^2} \int_{H_\sigma} K_\eps(x-y)
    \cdot\left(\left(\xi(x) - u(y)\right)^{\otimes 2} 
    - \left(u(x) - u(y)\right)^{\otimes 2}\right) \d x \, \d y}_{=:I}
  \end{split}
 \]
 Using the identity $K_\eps(x-y)\cdot (a^{\otimes 2} - b^{\otimes 2}) 
 = K_\eps(x-y)(a - b)\cdot(a + b)$, we obtain
 \[
  \begin{split}
   I &= \frac{1}{2\eps^2} \int_{H_\sigma} K_\eps(x-y)
      \left(\xi(x) - u(x)\right) \cdot
      \left(\xi(x) + u(x) - 2u(y)\right) \d x \, \d y 
  \end{split}
 \]
 Since~$u$, $\xi$ take their values in the bounded set~$\QQ$,
 we deduce
 \[
  \begin{split}
   &I \lesssim \frac{1}{\eps^2} \int_{H_\sigma} g_\eps(x-y)
      \abs{\xi(x) - u(x)} \d x \, \d y \\
   &\lesssim \sup_{x\in G} \left( \frac{1}{\eps^2} 
     \int_{\{y\in\R^3\colon\abs{x-y}\geq\sigma\}} g_\eps(x-y)
      \d y \right) \norm{\xi - u}_{L^1(G)}
  \end{split}
 \]
 By applying Lemma~\ref{lemma:decayK} and the H\"older inequality 
 at the right-hand side, the result follows.}
\end{proof}

\begin{lemma} \label{lemma:glue}
 {\BBB Let~$G\subseteq\R^3$ be a Borel set.
 Let~$u_1\in L^\infty(\R^3, \, \QQ)$, 
 $u_2\in L^\infty(\R^3, \, \QQ)$. Then,
 \[
   F_\eps^\nl(u_2, \, G) \lesssim F_\eps^\nl(u_1, \, G) 
   + \frac{1}{\eps^2} \norm{u_2 - u_1}^2_{L^2(G)} 
 \]}
\end{lemma}
\begin{proof}
 {\BBB By writing~$u_2(x) - u_2(y) = (u_1(x) - u_1(y)) 
 + (u_2(x) - u_1(x)) + (u_1(y) - u_2(y))$, and using that~$K$ is even
 (by assumption~\eqref{hp:K_even}), we obtain
 \begin{equation*}
  \begin{split}
   F^\nl_\eps(u_2, \, G) 
   &\lesssim F^\nl_\eps(u_1, \, G) + \frac{2}{\eps^2}\int_{G\times G} 
    K_\eps(x - y)\cdot \left(u_2(x) - u_1(x)\right)^{\otimes 2} \d x \, \d y 
  \end{split}
 \end{equation*}
 and the lemma follows.}
\end{proof}

\begin{lemma} \label{lemma:glueH1}
 {\BBB Let~$G\csubset G^\prime\csubset\R^3$ be open sets
 and~$\sigma\in (0, \, \dist(G, \, \partial G^\prime))$. Then,
 for any~$u\in H^1(G^\prime, \, \QQ)$, we have
 \[
  \begin{split}
   F_\eps^\nl(u, \, G) &\lesssim 
   \int_{G\cup\partial_\sigma G} \abs{\nabla u}^2
   + \o\!\left(\frac{\eps^{q-2}}{\sigma^q}\right) \inf_{\zeta\in\R^m}
    \norm{u - \zeta}^2_{L^2(G)}
  \end{split}
 \] 
 where~$q>7/2$ is the number given by~\eqref{hp:g_decay}.}
\end{lemma}
\begin{proof}
 {\BBB This lemma is a variant of
 Proposition~\ref{prop:Gamma-limsup}. We have
 \begin{equation*} 
  \begin{split}
   F_\eps^\nl(u, \, G) 
   &= \frac{1}{4\eps^2}\int_{\{x\in G, \, y\in G\colon \abs{x - y}\leq\sigma\}}
   K_\eps(x-y)\cdot\left(u(x) - u(y)\right)^{\otimes 2} \,\d x\, \d y \\
   &\qquad + \frac{1}{4\eps^2}
   \int_{\{x\in G, \, y\in G\colon \abs{x - y}\geq\sigma\}}
   K_\eps(x-y)\cdot\left(u(x) - u(y)\right)^{\otimes 2} \,\d x\, \d y 
   =: I_1 + I_2
  \end{split}
 \end{equation*}
 For the first term~$I_1$, we may repeat the very same argument 
 from the proof of Proposition~\ref{prop:Gamma-limsup}, which gives
 \begin{equation*}
  \begin{split}
   I_1 \leq \frac{1}{4\eps^2}\int_{G\times B_{\sigma/\eps}}
   K(z)\cdot\left(u(x + \eps z) - u(x)\right)^{\otimes 2} \,\d x\, \d z
   \lesssim \int_{G\cup\partial_\sigma G} \abs{\nabla u}^2
  \end{split}
 \end{equation*}
 On the other hand, the estimate~\eqref{locdefect1}
 in the proof of Lemma~\ref{lemma:locdefect} shows that
 \[
  I_2 \lesssim \o\!\left(\frac{\eps^{q-2}}{\sigma^q}\right)
  \inf_{\zeta\in\R^m} \norm{u - \zeta}^2_{L^2(G)}
  \qedhere
 \]}
\end{proof}

The next lemma is an estimate on the potential~$\psi_b$.

\begin{lemma} \label{lemma:psib}
 For any~$\delta>0$, there exists a constant~$C_\delta>0$
 such that, for any~$y_1\in\QQ$, $y_2\in\QQ$ with
 $\dist(y_2, \, \partial\QQ)\geq\delta$, we have
 \begin{equation} \label{in_psib}
  \psi_b(y_2) \leq C_\delta 
  \left(\psi_b(y_1) + \abs{y_1 - y_2}^2\right).
 \end{equation}
\end{lemma}
\begin{proof}
 The assumption~\eqref{hp:non_degeneracy} implies,
 via a Taylor expansion and a compactness argument, that
 there exist~$\gamma>0$, $\kappa_1>0$, $\kappa_2>0$ 
 so that if $\dist(y, \, \NN)<\gamma$, then
 \begin{equation} \label{nondeg}
  \kappa_1 \dist^2(y, \, \NN)\leq 
  \psi_b(y)\leq \kappa_2 \dist^2(y, \, \NN).
 \end{equation}
 To prove the result we exhaust three cases, 
 \begin{enumerate}
  \item $\dist(y_1, \, \NN)\geq \frac{1}{2}\gamma$.
  \item $\dist(y_1, \, \NN)< \frac{1}{2}\gamma$, $\dist(y_2, \, \NN)\geq \gamma$.
  \item $\dist(y_1, \, \NN)< \frac{1}{2}\gamma$, $\dist(y_2, \, \NN)< \gamma$.
 \end{enumerate}
 In the case of (1), we have that such~$y_1$ satisfy
 $\psi_b(y_1)>c_1$ for a constant $c_1>0$
 (that depends on~$\gamma$), as~$y_1$ is bounded away
 from the minimising manifold. We furthermore have 
 that~$\psi_b(y_2)\leq c_2$ because $\dist(y_2, \, \partial\QQ)>\delta$
 (and the constant~$c_2$ will depend on~$\delta$). 
 Therefore the inequality~\eqref{in_psib} holds 
 trivially with $C_\delta=\frac{c_1}{c_2}$.
 
 In the case of (2), since $\dist(y_1, \, \NN)<\frac{1}{2}\gamma$,
 $\dist(y_2, \, \NN)\geq \gamma$, we must have $|y_1-y_2|^2\geq \frac{1}{4}\gamma^2$,
 then we use the upper bound on $\psi_b(y_2)$ as before. 

 In the case of (3), we note that since $y_1,y_2$ are
 both sufficiently close to $\NN$,
 \begin{equation*}
  \begin{split}
   \psi_b(y_2) &\stackrel{\eqref{nondeg}}{\lesssim} \dist^2(y_2, \, \NN)
   \lesssim \dist^2(y_1, \, \NN) + |y_1-y_2|^2 
   \stackrel{\eqref{nondeg}}{\lesssim} \psi_b(y_1)+|y_1-y_2|^2. \qedhere
  \end{split}
 \end{equation*}
\end{proof}

Finally, we conclude this section with interpolation (or extension)
results. The first one is a classical interpolation result for~$H^1$-maps;
it constructs a suitable map in an annulus, with prescribed values on
the boundary. 

\begin{lemma}[\cite{Luckhaus,contreras2018convergence}]
\label{lemma:interpolation}
 For any~$M>0$, there exists~$\eta = \eta(M)>0$
 such that the following statement holds.
 Let~$\QQ_0\csubset\QQ$ be a convex, open set that contains~$\NN$. 
 Let~$\rho$, $\lambda$ be positive numbers with~$\lambda < \rho$,
 and let~$u\in H^1(\partial B_\rho, \, \QQ_0)$,
 $v\in H^1(\partial B_\rho, \, \NN)$ be such that
 \[
  \int_{\partial B_\rho} 
   \left(\abs{\nabla u}^2 + \abs{\nabla v}^2\right)
   \d\H^2\leq M, \qquad 
   \int_{\partial B_\rho} \abs{u - v}^2
   \,\d\H^2 \leq \eta\lambda^2. 
 \]
 Then, there exists a map 
 $w\in H^1(B_{\rho} \setminus B_{\rho - \lambda}, \, \QQ_0)$
 such that $w(x) = u(x)$ for $\H^2$-a.e.~$x\in\partial B_\rho$,
 $w(x) = v(\rho x/(\rho - \lambda))$ 
 for $\H^2$-a.e.~$x\in\partial B_{\rho - \lambda}$, and
 \begin{gather*}
  \int_{B_\rho\setminus B_{\rho-\lambda}} \abs{\nabla w}^2
   \lesssim \lambda \int_{\partial B_\rho}
   \left(\abs{\nabla u}^2 + \abs{\nabla v}^2 +
   \frac{\abs{u - v}^2}{\lambda^2}\right)\d\H^2 \\
  \int_{B_\rho\setminus B_{\rho-\lambda}} \psi_b(w)
   \lesssim \lambda \int_{\partial B_\rho} 
   \psi_b(u)\,\d\H^2
 \end{gather*}
\end{lemma}

\begin{remark} \label{rk:interpolation}
 Lemma~\ref{lemma:interpolation}, in case~$\psi_b=0$,
 was first proven by Luckhaus~\cite[Lemma~1]{Luckhaus}.
 Up to a scaling, the statement given here is essentially
 the same as~\cite[Lemma~B.2]{contreras2018convergence}.
 However, in~\cite{contreras2018convergence} 
 the potential is assumed to be finite and smooth 
 on the whole of~$\R^m$, while our potential~$\psi_b$ is 
 singular out of~$\QQ$. Nevertheless, the proof carries over
 to our setting. Indeed, the map~$w$ constructed 
 in~\cite{contreras2018convergence} takes values in a
 neighbourhood of~$\NN$, whose thickness can be made arbitrarily
 small by choosing~$\eta$ small (see also~\cite[Lemma~1]{Luckhaus}). 
 In particular, we can make sure that the
 image of~$w$ is contained in the set~$\QQ_0$, where
 the function~$\psi_b$ is finite and smooth,
 and the arguments of~\cite{contreras2018convergence} carry over.
 Incidentally, Lemma~\ref{lemma:interpolation} crucially depends
 on the non-degeneracy assumption~\eqref{hp:non_degeneracy}
 for the bulk potential~$\psi_b$.
\end{remark}

We give a variant of Lemma~\ref{lemma:interpolation}
which is adapted to our non-local setting. 

\begin{lemma} \label{lemma:nonlocinterp}
 {\BBB Let~$\QQ_0\csubset\QQ$ be an open convex set.
 Let~$u_\eps$, $u\in L^\infty(\R^3, \, \QQ_0)$
 and~$u^*_\eps$, $u^*\in H^1(B_{1/2}, \, \NN)$
 satisfy the following conditions:
 \begin{gather}
  M_\eps := \int_{B_{1/2}} \abs{\nabla u^*_\eps}^2
   + F_\eps(u_\eps, \, B_1) + \norm{u^*_\eps - u_\eps}^2_{L^2(B_t)}
   \quad \textrm{is bounded} \label{hp:interp3} \\
  u_\eps\to u \ \textrm{ strongly in } L^2(B_{1/2}),
   \quad u^*_\eps\to u^* \label{hp:interp1}
   \ \textrm{ strongly in } H^1(B_{1/2}) \textrm{ as } \eps\to 0 \\
  u^* = u \quad \textrm{ a.e. in } B_1\setminus B_s
   \ \textrm{ for some } s\in (1/4, \, 1/2) \label{hp:interp2}
 \end{gather}
 Let~$\sigma\in (0, \, 1/10)$.
 Then, up to extraction of a (non-relabelled) subsequence,
 there exist maps~$\xi_\eps\in L^\infty(\R^3, \, \QQ_0)$
 and radii~$r$, $t$ with~$s < r < t < 1/2$
 that satisfy the following conditions:
 \begin{enumerate}[label=(\roman*)]
  \item $\xi_\eps = u_\eps$ a.e.~in~$\R^3\setminus B_t$;
  \item ${\xi_\eps}_{|B_r}\in H^1(B_r, \, \QQ_0)$
  and~${\xi_\eps}_{|B_r}\to u^*_{|B_r}$ strongly in~$H^1(B_r)$;
  \item there holds
  \[
   \begin{split}
    &F_\eps(\xi_\eps, \, B_t) 
     + \Gamma_\eps(\xi_\eps, \, B_t, \, \R^3\setminus B_t)
     - \Gamma_\eps(u_\eps, \, B_t, \, \R^3\setminus B_t) \\ 
    &\qquad\qquad \leq F_\eps^\nl(\xi_\eps, \, B_r) 
     + C\int_{B_r\setminus B_{r-\sigma}} \abs{\nabla\xi_\eps}^2 
     + C\sigma M_\eps 
     + \o\!\left(\frac{\eps^{q-2} M_\eps^{1/2}}{\sigma^q}\right)
   \end{split}
  \]
  where~$q>7/2$ is given by~\eqref{hp:g_decay}.
 \end{enumerate} }
\end{lemma}
{\BBB As we will see in the proof, the maps~$\xi_\eps$
agree with~$u^*_\eps$ on~$B_r$, up to rescaling
and interpolation near the boundary of~$B_r$.}
\begin{proof}[Proof of Lemma~\ref{lemma:nonlocinterp}]
 {\BBB We split the proof into several steps.
 
 \medskip
 \begin{step}[Construction of~$\xi_\eps$]
  Let~$\varphi_\eps\in C^{\infty}_{\mathrm{c}}(\R^3)$ 
  be a sequence of mollifiers, defined as in~\eqref{mollify}. 
  Lem\-ma~\ref{lemma:mollify-Morrey} implies
  \begin{equation} \label{energybd}
   \int_{B_{1/2}} \abs{\nabla(\varphi_\eps * u_\eps)}^2
   \lesssim F(u_\eps, \, B_1) \stackrel{\eqref{hp:interp3}}{\leq} M_\eps
  \end{equation}
  for~$\eps$ small enough. Let~$N\geq 1$ be an integer number such that
  \[
   \frac{1}{18\sigma} \leq N \leq \frac{1}{6\sigma}
  \]
  Such a number exists, because of the assumption that~$0 < \sigma < 1/10$.
  We divide the annulus~$B_{1/2}\setminus\bar{B}_{s}$
  into~$N$ concentric sub-annuli:
  \[
   A_i := B_{s + i\frac{1/2 - s}{N}} 
   \setminus
   \bar{B}_{s + (i-1)\frac{1/2 - s}{N}}
   \qquad \textrm{for } i = 1, \, 2, \, \ldots, \, N.
  \]
  We have
  \[
   \sum_{i=1}^N \left(F_\eps(u_\eps, \, A_i)
    + \int_{A_i} \abs{\nabla(\varphi_\eps * u_\eps)}^2
    + \int_{A_i} \abs{\nabla u^*_\eps}^2\right)
    \stackrel{\eqref{energybd}}{\lesssim} M_\eps
  \]
  As a consequence, for any~$\eps$
  we can choose an index~$i(\eps)$ such that
  \begin{equation} \label{comp1}
   F_\eps(u_\eps, \, A_{i(\eps)})
   + \int_{A_i} \abs{\nabla(\varphi_\eps * u_\eps)}^2
   + \int_{A_i} \abs{\nabla u^*_\eps}^2
   \lesssim \frac{M_\eps}{N} \lesssim  \sigma M_\eps
  \end{equation}
  Passing to a subsequence, we may also assume
  that all the indices~$i(\eps)$ are the same, so from now on,
  we write~$i$ instead of~$i(\eps)$. We take 
  positive numbers~$a < b$ such that 
  $A^\prime := B_b\setminus \bar{B}_a\csubset A_i$
  and~$b - a > 5\sigma$.
  Then, Lemma~\ref{lemma:mollify-L2} gives
  \begin{equation} \label{comp2}
   \frac{1}{\eps^2} \int_{A^\prime} 
    \abs{\varphi_\eps * u_\eps - u_\eps}^2
    \lesssim F_\eps(u_\eps, \, A_i) 
    \stackrel{\eqref{comp1}}{\lesssim} \sigma M_\eps
  \end{equation}
  for~$\eps$ small enough. We have assumed that~$u_\eps$
  takes its values in the convex set~$\QQ_0\csubset\QQ$;
  it follows that the image of~$\varphi_\eps*u_\eps$
  is contained in~$\overline{\QQ_0}\csubset\QQ$. Thus, we may apply
  Lemma~\ref{lemma:psib} to estimate the integral 
  of~$\psi_b(\varphi_\eps *u_\eps)$:
  \begin{equation} \label{comp3}
   \begin{split}
    \frac{1}{\eps^2} \int_{A^\prime} 
     \psi_b(\varphi_\eps *u_\eps) 
    &\lesssim
    \frac{1}{\eps^2} \int_{A^\prime} \left(\psi_b(u_\eps)  
     + \abs{\varphi_\eps * u_\eps - u_\eps}^2\right) 
    \stackrel{\eqref{comp2}}{\lesssim}
    \sigma M_\eps
   \end{split}
  \end{equation}
  Using Fatou's lemma, we see that
  \begin{equation*} 
   \begin{split}
    &\int_a^b \left(\liminf_{\eps\to 0} \int_{\partial B_r}
     \abs{\nabla u^*_\eps}^2 + \abs{\nabla(\varphi_\eps * u_\eps)}^2 
     + \frac{1}{\eps^2} \psi_b(\varphi_\eps * u_\eps) 
     \, \d\H^2 \right)\d r \\ 
    &\qquad\qquad\qquad \leq 
     \liminf_{\eps\to 0} \int_{A^\prime}
     \left(\abs{\nabla u^*_\eps}^2 + \abs{\nabla(\varphi_\eps * u_\eps)}^2 
     + \frac{1}{\eps^2} \psi_b(\varphi_\eps * u_\eps)\right)
    \stackrel{\eqref{energybd}, \eqref{comp3}}{\lesssim} \sigma M_\eps
   \end{split}
  \end{equation*}
  By Fubini theorem, there exists a radius~$r\in (a + \sigma, \, b - 3\sigma)$ 
  and a (non-relabelled) subsequence~$\eps\to 0$ such that
  \begin{equation}  \label{comp4}
   \begin{split}
    &\int_{\partial B_r} \left(\abs{\nabla  u^*_\eps}^2 + 
     \abs{\nabla(\varphi_\eps * u_\eps)}^2 
     + \frac{1}{\eps^2} \psi_b(\varphi_\eps * u_\eps) 
     \right)\, \d\H^2 \lesssim M_\eps
   \end{split}
  \end{equation}
  By a similar argument, we may also assume that
  \begin{equation} \label{comp5}
   \begin{split}
    \int_{\partial B_r} \abs{\varphi_\eps * u_\eps -  u^*_\eps}^2
    \d\H^2 \lesssim \eps^2 M_\eps 
    + \frac{1}{\sigma} \int_{A^\prime} \abs{ u^*_\eps - u_\eps}^2
   \end{split}
  \end{equation}
  Due to~\eqref{hp:interp1} and~\eqref{hp:interp2}, 
  we have~$ u^*_\eps - u_\eps \to  u^* - u$
  in~$L^2(A^\prime)$ and~$ u^* = u$ in~$A^\prime$,
  so the right-hand side of~\eqref{comp5} tends to zero. Let
  \begin{equation} \label{def_lambda}
   \lambda_\eps := \left(\eps^2 M_\eps 
    + \frac{1}{\sigma} 
    \int_{A^\prime} \abs{u^*_\eps - u_\eps}^2\right)^{1/4} > 0
  \end{equation}
  We have~$\lambda_\eps\to 0$ as~$\eps\to 0$. Moreover,
  for this choice of~$\lambda_\eps$, the assumptions of
  Lemma~\ref{lemma:interpolation} are satisfied for~$\eps$
  small enough. By applying Lemma~\ref{lemma:interpolation},
  we construct a map
  $w_\eps\in H^1(B_{r} \setminus B_{r - \lambda_\eps}, \, \QQ_0)$
  such that $w_\eps(x) = (\varphi * u_\eps)(x)$ 
  for~$x\in\partial B_r$,
  $w_\eps(x) = v(r x/(r - \lambda_\eps))$ 
  for~$x\in\partial B_{r - \lambda_\eps}$, and
  \begin{equation} \label{comp6}
   \begin{split}
    &\int_{B_r\setminus B_{r-\lambda_\eps}}
     \left(\abs{\nabla w_\eps}^2 
     + \frac{1}{\eps^2} \psi_b(w_\eps)\right)
    \lesssim \lambda_\eps M_\eps 
     + \frac{1}{\lambda_\eps} \left(\eps^2 M_\eps
     + \frac{1}{\sigma} \int_{A^\prime} \abs{ u^*_\eps - u_\eps}^2\right)
   \end{split} 
  \end{equation} 
  and
  \begin{equation} \label{competitorbulk}
   \frac{1}{\eps^2} 
    \int_{B_r\setminus B_{r-\lambda_\eps}} \psi_b(w_\eps)
   \lesssim \frac{\lambda_\eps}{\eps^2} 
    \int_{\partial B_r} \psi_b(w_\eps) \, \d\H^2 
   \lesssim \lambda_\eps M_\eps
  \end{equation}
  The right-hand sides of~\eqref{comp6}, \eqref{competitorbulk} converge
  to zero as~$\eps\to 0$. Finally, we take~$t\in (r + 2\sigma, \, b - \sigma)$
  and we define
  \begin{equation} \label{def_xi}
   \xi_\eps(x) := \begin{cases}
                u_\eps(x) & \textrm{if } x\in\R^3\setminus B_t \\
                (\varphi_\eps * u_\eps)(x)
                 & \textrm{if } B_t\setminus B_r \\
                w_\eps(x) 
                 & \textrm{if } x\in B_r\setminus B_{r - \lambda_\eps}\\
                 u^*_\eps\left(\dfrac{rx}{r-\lambda_\eps}\right) 
                 & \textrm{if } x\in B_{r - \lambda_\eps}
               \end{cases}
  \end{equation}
  By construction, $\xi_\eps=u_\eps$ out of~$B_t$. Moreover,
  a routine computation, based on~\eqref{comp6},
  shows that $\xi_\eps\to u^*_\eps$ strongly in~$H^1(B_r)$
 \end{step}
 
 \medskip
 \begin{step}[Bounds on~$\norm{\xi_\eps - u_\eps}_{L^2(B_t)}$]
  By construction, we have
  \begin{equation*}
   \begin{split}
    \norm{\xi_\eps - u_\eps}_{L^2(B_t\setminus B_r)}
    = \norm{\varphi_\eps * u_\eps - u_\eps}_{L^2(A^\prime)}
    \stackrel{\eqref{comp2}}{\lesssim} \eps \sigma^{1/2} M_\eps^{1/2}
   \end{split}
  \end{equation*}
  On the other hand,
  \begin{equation*}
   \begin{split}
    \norm{\xi_\eps - u_\eps}_{L^2(B_r)}
    \lesssim \norm{w_\eps - \tau_\eps u^*_\eps}_{L^2(B_r\setminus 
    B_{r-\lambda_\eps})}
    + \norm{\tau_\eps u^*_\eps -  u^*_\eps}_{L^2(B_r)}
    + \norm{ u^*_\eps - u_\eps}_{L^2(B_r)}
   \end{split}
  \end{equation*}
  where~$\tau_\eps u^*_\eps(x) :=  u^*_\eps(rx/(r - \lambda_\eps))$.
  We recall that $\norm{ u^*_\eps - u_\eps}_{L^2(B_{1/2})}\leq M_\eps^{1/2}$.
  The norm of~$w_\eps - \tau_\eps u^*_\eps$
  can be estimated using the Poincar\'e inequality and~\eqref{comp6}, 
  because we know that~$w_\eps = \tau_\eps u^*_\eps$ 
  on~$\partial B_{r-\lambda_\eps}$:
  \begin{equation*}
   \begin{split}
    \norm{w_\eps - \tau_\eps u^*_\eps}
     _{L^2(B_r\setminus B_{r-\lambda_\eps})}
    &\lesssim 
    \lambda_\eps \norm{\nabla(w_\eps - \tau_\eps u^*_\eps)}
     _{L^2(B_r\setminus B_{r-\lambda_\eps})}
    \lesssim \left(\lambda_\eps 
    + \frac{\lambda_\eps^{1/2}}{\sigma^{1/2}}\right) M_\eps^{1/2} 
   \end{split}
  \end{equation*}
  A classical argument (see e.g. \cite[Lemma~7.23]{gilbarg2015elliptic}
  for an analogous result) gives
  \begin{equation*}
   \begin{split}
    \norm{\tau_\eps u^*_\eps -  u^*_\eps}_{L^2(B_r)}
    \lesssim \lambda_\eps \norm{\nabla u^*_\eps}_{L^2(B_r)}
    \lesssim \lambda_\eps M_\eps^{1/2}
   \end{split}
  \end{equation*}
  Combining the inequalities above, and 
  recalling that~$\lambda_\eps\to 0$, we obtain
  \begin{equation} \label{compL2}
   \norm{\xi_\eps - u_\eps}_{L^2(B_t)}\lesssim M_\eps^{1/2}
  \end{equation}
  As a consequence, we deduce
  \begin{equation*} 
   \begin{split}
    \inf_{\zeta\in\R^m} \norm{\xi_\eps - \zeta}_{L^2(B_t)}
    &\leq \norm{\xi_\eps - u_\eps}_{L^2(B_t)}
    + \inf_{\zeta\in\R^m} \norm{u_\eps - \zeta}_{L^2(B_t)}\\
    &\lesssim M_\eps^{1/2} + \inf_{\zeta\in\R^m} \norm{u_\eps - \zeta}_{L^2(B_t)}
   \end{split}
  \end{equation*}
  and hence, by Proposition~\ref{prop:Poincare},
  \begin{equation} \label{compL2bis}
   \begin{split}
    \inf_{\zeta\in\R^m} \norm{\xi_\eps - \zeta}_{L^2(B_t)}
    &\lesssim M_\eps^{1/2} + F_\eps(u_\eps, \, B_1)^{1/2} 
    \lesssim M_\eps^{1/2}
   \end{split}
  \end{equation}
 \end{step}

 \medskip
 \begin{step}[Bounds on~$F_\eps(\xi_\eps, \, B_t)$]
  Since~$K_\eps$ is an even function (see~\eqref{hp:K_even}), we have
  \begin{equation} \label{interp7}
   \begin{split}
    F_\eps^\nl(\xi_\eps, \, B_t) &\leq F_\eps^\nl(\xi_\eps, \, B_r)
     + F_\eps^\nl(\xi_\eps, \, B_t\setminus B_r)
     + \Gamma_\eps(\xi_\eps, \, B_r, \, B_t\setminus B_r)  \\
    &\leq F_\eps^\nl(\xi_\eps, \, B_r)
     + F_\eps^\nl(\varphi_\eps* u_\eps, \, A^\prime)
     + \Gamma_\eps(\xi_\eps, \, B_r, \, B_t\setminus B_r) 
   \end{split}
  \end{equation}
  where~$\Gamma_\eps(\xi_\eps, \, B_r, B_t\setminus B_r)$ is defined 
  as in~\eqref{locdefect}. Lemma~\ref{lemma:glue} implies
  \begin{equation} \label{interp8}
   \begin{split}
    F_\eps^\nl(\varphi_\eps* u_\eps, \, A^\prime) 
    &\lesssim F_\eps^\nl(u_\eps, \, A^\prime) 
     + \frac{1}{\eps^2} \norm{\varphi_\eps *u_\eps - u_\eps}^2_{L^2(A^\prime)}
     \stackrel{\eqref{comp1}, \ \eqref{comp3}}{\lesssim}
     \sigma M_\eps
   \end{split}
  \end{equation}
  Now, we estimate~$\Gamma_\eps(\xi_\eps, \, B_r, \, B_t\setminus B_r)$.
  By construction, we have $a + \sigma < r < t < b - \sigma$ and hence
  $\partial_\sigma B_r\subseteq A^\prime$.
  We apply Lemma~\ref{lemma:locdefect},
  Lemma~\ref{lemma:glueH1} and~\eqref{compL2bis}:
  \begin{equation} \label{interp9}
   \begin{split}
    \Gamma_\eps(\xi_\eps, \, B_r, \, B_t\setminus B_r) 
    &\lesssim F_\eps^\nl(\xi_\eps, \, \partial_{\sigma/2} B_r)
    + \o\!\left(\frac{\eps^{q - 2}}{\sigma^q}\right) 
    \inf_{\zeta\in\R^m} \norm{\xi_\eps - \zeta}_{L^2(B_t)}^2 \\
    &\lesssim \int_{\partial_{\sigma} B_r} \abs{\nabla\xi_\eps}^2
    + \o\!\left(\frac{\eps^{q - 2}}{\sigma^q}\right)
    \inf_{\zeta\in\R^m} \norm{\xi_\eps - \zeta}_{L^2(B_t)}^2 \\
    &\lesssim \int_{\partial_{\sigma} B_r} \abs{\nabla\xi_\eps}^2
    + \o\!\left(\frac{\eps^{q - 2}M_\eps}{\sigma^q}\right)
   \end{split}
  \end{equation}
  The gradient term at the right-hand side can be
  further estimated by~\eqref{comp2}:
  \begin{equation} \label{interp9.5}
   \begin{split}
    \int_{\partial_{\sigma} B_r} \abs{\nabla\xi_\eps}^2
    \lesssim  \int_{B_r\setminus B_{r-\sigma}} \abs{\nabla\xi_\eps}^2
    + \int_{A^\prime} \abs{\nabla(\varphi_\eps * u_\eps)}^2
    \lesssim \int_{B_r\setminus B_{r-\sigma}} \abs{\nabla\xi_\eps}^2
    + \sigma M_\eps
   \end{split}
  \end{equation}
  Combining~\eqref{interp7}, \eqref{interp8}, \eqref{interp9}
  and~\eqref{interp9.5}, we obtain
  \begin{equation} \label{interp10}
   \begin{split}
    F^\nl_\eps(\xi_\eps, \, B_t) 
    \leq F^\nl_\eps(\xi_\eps, \, B_r)
    &+ \int_{B_r\setminus B_{r-\sigma}} \abs{\nabla\xi_\eps}^2
    + \sigma M_\eps 
    + \o\!\left(\frac{\eps^{q - 2} M_\eps}{\sigma^q}\right)
   \end{split}
  \end{equation}
  We estimate the local term of the energy, i.e.~the 
  integral of~$\psi_b(\xi_\eps)$ in~$B_t$. By construction,
  $\xi_\eps$ restricted to~$B_{r-\lambda_\eps}$ takes 
  its values in~$\NN$. As a consequence, 
  \begin{equation} \label{interp11}
   \begin{split}
    \frac{1}{\eps^2} \int_{B_t} \psi_b(\xi_\eps)
     &= \frac{1}{\eps^2} \int_{B_t\setminus B_r} \psi_b(\varphi_\eps * u_\eps)
     + \frac{1}{\eps^2} \int_{B_r \setminus B_{r - \lambda_\eps}}\psi_b(w_\eps)
     \stackrel{\eqref{comp3}, \ \eqref{competitorbulk}}{\lesssim} 
     \sigma M_\eps
   \end{split}
  \end{equation}
 \end{step}
 
 \medskip
 \begin{step}[Bounds on~$\Gamma_\eps(\xi_\eps, \, B_t, \, \R^3\setminus B_t)$]
  By construction, $\xi_\eps = u_\eps$ out of~$B_t$. Then, 
  Lemma~\ref{lemma:diffGamma} implies
  \begin{equation*}
   \begin{split}
    \Gamma_\eps(\xi_\eps, \, B_t, \, \R^3\setminus B_t)
    - \Gamma_\eps(u_\eps, \, B_t, \, \R^3\setminus B_t)
    &\lesssim F^\nl_\eps(\xi_\eps, \, \partial_\sigma B_t) 
    + \o\!\left(\frac{\eps^{q-2}}{\sigma^q}\right)
     \norm{\xi_\eps - u_\eps}_{L^2(B_t)}
   \end{split}
  \end{equation*}
  Reasoning as in~\eqref{interp8}, and applying~\eqref{comp1},
  we deduce
  \[
   F^\nl_\eps(\xi_\eps, \, \partial_\sigma B_t) \lesssim \sigma M_\eps
  \]
  Then, due to~\eqref{compL2}, we conclude that
  \begin{equation} \label{interp12}
   \begin{split}
    &\Gamma_\eps(\xi_\eps, \, B_t, \, \R^3\setminus B_t)
    - \Gamma_\eps(u_\eps, \, B_t, \, \R^3\setminus B_t) 
    \lesssim \sigma M_\eps 
    + \o\!\left(\frac{\eps^{q-2} M_\eps^{1/2}}{\sigma^q}\right)
   \end{split}
  \end{equation}
  Combining~\eqref{interp10}, \eqref{interp11} and~\eqref{interp12},
  we obtain the estimate~(iii) in the statement of the lemma, and
  the proof is complete.
  \qedhere
 \end{step} 
 }
\end{proof}

\begin{remark} \label{rk:nonlocinterp}
 {\BBB In addition to~\eqref{hp:interp3}, 
 \eqref{hp:interp1} and~\eqref{hp:interp2},
 suppose there exist points~$\zeta_\eps\in\NN$, a positive sequence
 $\eta_\eps\to 0$ and maps~$v\in H^1(B_{1/2}, \, \R^m)$,
 $v^*\in H^1(B_{1/2}, \, \R^m)$ such that~$v = v^*$ out of~$B_s$,
 $M_\eps\lesssim \eta^2_\eps$ and
 \[
  \frac{u_\eps - \zeta_\eps}{\eta_\eps} \to v \
  \textrm{strongly in } L^2(B_{1/2}), \quad 
  \frac{u^*_\eps - \zeta_\eps}{\eta_\eps} \to v^* \
  \textrm{strongly in } H^1(B_{1/2})
 \]
 as~$\eps\to 0$.
 Then, the same sequence~$\xi_\eps$ constructed above satisfies
 \[
  \frac{\xi_\eps - \zeta_\eps}{\eta_\eps} \to v^* \qquad
  \textrm{strongly in } L^2(B_r),
 \]
 \emph{so long as} we choose~$\lambda_\eps$ in a suitable way
 (see in particular Equations~\eqref{def_xi} and~\eqref{comp6}).
 For instance, instead of~\eqref{def_lambda}, we may take
 \[
  \lambda_\eps := \left(\eps^2 \frac{M_\eps}{\eta^2_\eps} 
  + \frac{1}{\sigma \eta_\eps^2} 
    \int_{A^\prime} \abs{u^*_\eps - u_\eps}^2\right)^{1/4}
 \]
 }
\end{remark}

\section{Proof of the main results}

\subsection{A compactness result for $\omega$-minimisers}

The goal of this section is to prove a compactness result for 
minimisers of~$E_\eps$, subject to variable ``boundary conditions'',
as~$\eps\to 0$. For later convenience, we state our result
in terms of ``almost minimisers''
--- or, more precisely, $\omega$-minimisers,
as defined below. This will be useful
to study variants of our original minimisation problem, 
as we will do in Section~\ref{sect:bd}.

\begin{definition} \label{def:omegamin}
 Let~$\Omega\subseteq\R^3$ be a bounded domain.
 Let~$\omega\colon [0, \, +\infty)\to [0, \, +\infty)$
 be an increasing function such that $\omega(s) \to 0$ as~$s\to 0$.
 We say that a function~$u\in L^\infty(\R^3, \, \QQ)$
 is an $\omega$-minimiser of~$E_\eps$ in~$\Omega$ if, for any 
 ball~$B_\rho(x_0)\subseteq\Omega$ and any~$v\in L^\infty(\R^3, \, \QQ)$
 such that~$v = u$ a.e.~on~$\R^3\setminus B_\rho(x_0)$,
 there holds
 \[
  E_\eps(u) \leq E_\eps(v) + \omega(\eps) \, \rho.
 \]
\end{definition}

By definition, a minimiser for~$E_\eps$ in the
class~$\mathscr{A}$ defined by~\eqref{A}
is also a~$\omega$-minimiser in~$\Omega$, for any~$\omega\geq 0$.
$\omega$-minimisers behave nicely with respect to scaling.
Given~$u\in L^\infty(\R^3, \, \QQ)$,
an increasing function~$\omega\colon [0, \, +\infty)\to [0, \, +\infty)$,
$x_0\in\R^3$ and~$\rho>0$, we define $u_\rho\colon\R^3\to\QQ$
and~$\omega_\rho\colon [0, \, +\infty)\to [0, \, +\infty)$
as~$u_\rho(y) := u(x_0 + \rho y)$ for~$y\in\R^3$ 
and~$\omega_\rho(s) := \omega(\rho s)$ for~$s\geq 0$,
respectively. A scaling argument implies

\begin{lemma} \label{lemma:scaleomega}
 If~$u$ is an~$\omega$-minimiser for~$E_\eps$ 
 in a bounded domain~$\Omega\subseteq\R^3$, then~$u_\rho$
 is an~$\omega_\rho$-minimiser for~$E_{\eps/\rho}$ in~$(\Omega - x_0)/\rho$.
\end{lemma}

\begin{proposition} \label{prop:compactness}
 Let~$\Omega\subseteq\R^3$ be a bounded domain.
 Let~$\omega\colon [0, \, +\infty)\to [0, \, +\infty)$
 be an increasing function such that $\omega(s) \to 0$ as~$s\to 0$.
 Let~$u_\eps$ be a sequence of $\omega$-minimisers
 of~$E_\eps$ in~$\Omega$. {\BBB Suppose that there exists an open,
 convex set~$\QQ_0\csubset\QQ$ such that~$u_\eps(x)\in\QQ_0$ 
 for any~$\eps>0$ and a.e.~$x\in\Omega$.}
 Let~$B_\rho(x_0)\subseteq\Omega$ be a ball
 such that $\sup_{\eps>0} F_\eps(u_\eps, \, B_\rho(x_0))<+\infty$.
 Then, up to extraction of a non-relabelled subsequence,
 $u_\eps$ converge $L^2(B_{\rho/2}(x_0))$-strongly to a
 map~$u_0\in H^1(B_{\rho/2}(x_0), \, \NN)$, which minimises the functional
 \[
  w\in H^1(B_{\rho/2}(x_0), \, \NN) \mapsto
  \int_{B_{\rho/2}(x_0)} L\nabla w \cdot \nabla w
 \]
 subject to its own boundary conditions. 
 Moreover, for any~$s\in(0, \, \rho/2)$ there holds
 \begin{equation} \label{strongconv}
  \lim_{\eps\to 0} F_\eps(u_\eps, \, B_s(x_0)) 
  = \int_{B_s(x_0)} L\nabla u_0\cdot\nabla u_0
 \end{equation}
\end{proposition}

Proposition~\ref{prop:compactness} differs from the 
results in~\cite{taylor2018oseen} in that no 
``boundary condition'' is prescribed: each~$u_\eps$ 
minimises the functional~$E_\eps$
(possibily up to a small error, which is
quantified by the function~$\omega$) subject to 
its own ``boundary condition''. 

\begin{proof}[Proof of Proposition~\ref{prop:compactness}]
 If~$u$ is an~$\omega$-minimiser of~$E_\eps$ in~$\Omega$
 and~$B_\rho(x_0)\subseteq\Omega$, then~$u_\rho$
 is an~$\omega_\rho$-minimiser for~$E_{\eps/\rho}$ in~$(\Omega - x_0)/\rho$.
 Since we have assumed that~$\Omega$ is bounded,
 the radius~$\rho$ is bounded too --- say, $\rho\leq R_0$,
 where~$R_0$ depends only on~$\Omega$. The function~$\omega$
 is increasing, so~$\omega_\rho(s) = \omega(\rho s)
 \leq\omega(R_0 s) =: \omega_0(s)$. As a consequence, 
 $u_\rho$ is also an~$\omega_0$-minimiser. Since~$\omega_0$
 is independent of~$\rho$, by a scaling argument
 (see Equation~\eqref{scaling}) we may assume without 
 loss of generality that~$\rho = 1$ and~$x_0 = 0$.
 
  Let~$\varphi_\eps\in C^{\infty}_{\mathrm{c}}(\R^3)$ 
  be defined as in Section~\ref{sect:Poincare}. 
  Lem\-ma~\ref{lemma:mollify-Morrey} and Lemma~\ref{lemma:mollify-L2}
  imply that
  \begin{equation*} 
   \int_{B_{1/2}} \abs{\nabla(\varphi_\eps * u_\eps)}^2
   \lesssim F(u_\eps, \, B_1), \qquad
   \int_{B_{1/2}} \abs{\varphi_\eps * u_\eps - u_\eps}^2
   \lesssim \eps^2 F(u_\eps, \, B_1)
  \end{equation*}
  for~$\eps$ small enough. Since~$F(u_\eps, \, B_1)$ is bounded,
  we can extract a (non-relabelled) subsequence so that
  $\varphi_\eps * u_\eps\rightharpoonup u_0$ weakly in~$H^1(B_{1/2})$,
  $u_\eps\to u_0$ strongly in~$L^2(B_{1/2})$. The map~$u_0$
  takes its values in~$\NN$, because
  \[
   \int_{B_{1/2}} \psi_b(u_0) 
   \leq \liminf_{\eps\to 0} \eps^2 \int_{B_{1/2}} \psi_b(u_\eps) =0 
  \]
  by Fatou lemma. We must show that
  \begin{equation} \label{comp-min}
   \int_{B_{1/2}} L\nabla u_0\cdot\nabla u_0 
    \leq \int_{B_{1/2}} L\nabla v\cdot\nabla v
  \end{equation}
  for any~$v\in H^1(B_{1/2}, \, \NN)$ such that $v = u_0$
  on~$\partial B_{1/2}$. By an approximation argument, 
  it suffices to prove~\eqref{comp-min} in case~$v = u_0$
  in a neighbourhood of~$\partial B_{1/2}$.
  Therefore, we fix~$s\in (0, \, 1/2)$ and we 
  take a map~$v\in H^1(B_{1/2}, \, \NN)$ such that $v = u_0$
  on~$B_{1/2}\setminus\bar{B}_s$. The map~$v$ is not an admissible
  competitor for~$u_\eps$, because in general~$u_0 \neq u_\eps$
  on~$\R^3\setminus B_1$. However, we may construct suitable
  competitors by applying Lemma~\ref{lemma:nonlocinterp}.
  {\BBB We can indeed do so, because we have assumed that 
  all the~$u_\eps$'s take their values
  in an open, convex set~$\QQ_0\csubset\QQ$.}
  
  {\BBB Let~$\sigma\in (0, \, 1/10)$.
  By applying Lemma~\ref{lemma:nonlocinterp} (with~$u_\eps^* = v$ for any~$\eps$),
  we find maps~$\xi_\eps\in L^\infty(\R^3, \, \QQ_0)$ and radii~$r$, $t$
  with~$\max(s, \, 1/4) < r < t < 1/2$, so that 
  $\xi_\eps = u_\eps$ a.e.~in~$\R^3\setminus B_t$, 
  $\xi_\eps\to v$ strongly in~$H^1(B_r)$ and
  \begin{equation*} 
   \begin{split}
     F_\eps(\xi_\eps, \, B_t) 
      &+ \Gamma_\eps(\xi_\eps, \, B_t, \, \R^3\setminus B_t)
      - \Gamma_\eps(u_\eps, \, B_t, \, \R^3\setminus B_t) \\
     &\leq F_\eps^\nl(\xi_\eps, \, B_r) 
      + C\int_{B_r\setminus B_{r-\sigma}}\abs{\nabla\xi_\eps}^2
      + C\sigma + \o\!\left(\frac{\eps^{q-2}}{\sigma^q}\right)
   \end{split}
  \end{equation*}
  (where~$q>7/2$ is given by~\eqref{hp:g_decay}).
  Since~$u_\eps$ is an~$\omega$-minimiser for~$E_\eps$
  and~$\xi_\eps = u_\eps$ out of~$B_t$, we have
  \begin{equation*} 
   \begin{split}
    F_\eps(u_\eps, \, B_t) 
    + \Gamma_\eps(u_\eps, \, B_t, \, \R^3\setminus B_t)
    &\leq F_\eps(\xi_\eps, \, B_t) 
    + \Gamma_\eps(\xi_\eps, \, B_t, \, \R^3\setminus B_t) + \omega(\eps) 
   \end{split}
  \end{equation*}
  and hence, 
  \begin{equation} \label{comp7}
   \begin{split}
    F_\eps(u_\eps, \, B_t) 
    &\leq F_\eps^\nl(\xi_\eps, \, B_r) 
    + C\int_{B_r\setminus B_{r-\sigma}}\abs{\nabla\xi_\eps}^2
    + C\sigma + \o\!\left(\frac{\eps^{q-2}}{\sigma^q}\right) + \omega(\eps)
   \end{split}
  \end{equation}
  We apply Proposition~\ref{prop:Gamma-liminf} and 
  Proposition~\ref{prop:Gamma-limsup} to pass to the limit 
  in both sides of~\eqref{comp7},
  first as~$\eps\to 0$, then as~$\sigma\to 0$. We obtain
  \begin{equation*} 
   \int_{B_r} L\nabla u_0\cdot\nabla u_0 
   \leq \int_{B_r} L\nabla v\cdot\nabla v
  \end{equation*} 
  which implies~\eqref{comp-min}. In case~$v = u_0$,
  the same argument shows that 
  \begin{equation} \label{comp-lim1}
   \begin{split}
    \limsup_{\eps\to 0} F_\eps(u_\eps, \, B_r)
    \leq \int_{B_r} L\nabla u_0\cdot\nabla u_0
   \end{split}
  \end{equation}
  On the other hand, Proposition~\ref{prop:Gamma-liminf} implies
  \begin{equation} \label{comp-lim2}
   \begin{split}
    \int_{G} L\nabla u_0\cdot\nabla u_0 
    \leq \liminf_{\eps\to 0} F_\eps(u_\eps, \, G) \qquad
    \textrm{for any open set } G\subseteq B_{1/2}
   \end{split}
  \end{equation}
  Combining~\eqref{comp-lim1} with~\eqref{comp-lim2},
  we deduce~\eqref{strongconv}. }
\end{proof}

\subsection{A decay lemma for~$F_\eps$}

The aim of this section is to prove a decay property 
for the localised energy~$F_\eps$:

\begin{lemma} \label{lemma:decay}
 {\BBB There exist~$\eta>0$, $\theta\in (0, \, 1/2)$
 and~$\eps_*>0$ such that, for any
 ball~$B_{\rho}(x_0)\subseteq\Omega$, any~$\eps\in (0, \, \eps_*\rho)$
 and any minimiser~$u_\eps$ of~$E_\eps$ in~$\mathscr{A}$
 such that 
 \[
  F_\eps(u_\eps, \, B_{\rho}(x_0)) 
  \leq \eta^2 \rho,
 \]
 there holds
 \begin{equation} \label{decay}
  \begin{split}
   & F_\eps(u_\eps, \, B_{\theta\rho}(x_0))
   \leq \frac{\theta}{2} F_\eps(u_\eps, \, B_{\rho}(x_0))
   + \left(\frac{\eps}{\rho}\right)^{2q-4} \rho.
  \end{split}
 \end{equation} }
\end{lemma}

{\BBB Compared with analogous results in
the regularity theory for Oseen-Frank minimisers ---
see, for instance, Proposition~1 in~\cite{Luckhaus} or 
Theorem~2.4 in~\cite{HKL} --- the estimate~\eqref{decay} contains 
an extra term at the right-hand side. This additional term
controls the contibutions from the `locality defect'~$\Gamma_\eps$,
defined by~\eqref{locdefect} (cf. Lemma~\ref{lemma:diffGamma}
and Remark~\ref{rk:diffGamma}).
However, this term will introduce additional issues 
in the proof of Theorem~\ref{th:Holder}, which we are only 
able to resolve in case~$q>7/2$.}

As a first step towards the proof of Lemma~\ref{lemma:decay},
we check that the limit tensor~$L$ (defined by~\eqref{L})
is elliptic.

\begin{proposition} \label{prop:L}
 There exists a constant~$\lambda>1$ so that 
 \begin{equation*}
  \lambda^{-1} |\xi|^2 \leq  L\xi\cdot \xi \leq \lambda|\xi|^2
 \end{equation*}
 for all $\xi\in\mathbb{R}^{m\times 3}$. 
\end{proposition}
\begin{proof}
 The upper bound comes trivially, as 
 \begin{equation*}
  \begin{split}
   4L\xi\cdot \xi = \int_{\mathbb{R}^3}K(z)(\xi z)\cdot (\xi z)\,\d z
   \lesssim \int_{\mathbb{R}^3}g(z)|\xi z|^2\,\d z
   \lesssim \left(\int_{\mathbb{R}^3}g(z)|z|^2\,\d z\right)|\xi|^2
  \end{split}
 \end{equation*}
 and the constant at the right-hand side is finite, due to~\eqref{hp:g_decay}.
 For the lower bound, recall that $g$ is non-negative and satisfies $g(z)\geq k$ for $\rho_1<|z|<\rho_2$. Then we have that 
 \begin{equation*}
  \begin{split}
   4L\xi\cdot \xi = \int_{\mathbb{R}^3} K_{ij}(z) z_\alpha z_\beta \, \xi_{i,\alpha} \, \xi_{j,\beta}\,\d z
   &\geq \int_{\mathbb{R}^3} g(z)z_\alpha z_\beta \, \xi_{i,\alpha} \, \xi_{i,\beta}\,\d z \\
   &\geq k \int_{B_{\rho_2}\setminus B_{\rho_1}}z_\alpha z_\beta\,\d z \, \xi_{i,\alpha} \, \xi_{i,\beta}
  \end{split}
 \end{equation*}
 We may evaluate the inner integral as 
 \begin{equation*}
  \begin{split}
   \int_{B_{\rho_2}\setminus B_{\rho_1}} z_\alpha z_\beta\,\d z
   =& \int_{\rho_1}^{\rho_2} \int_{\mathbb{S}^2} r^2p_\alpha p_\beta \,\d p\,\d r\\
   =&\int_{\rho_1}^{\rho_2}r^2\,\d r \int_{\mathbb{S}^2} p_\alpha p_\beta \,\d p\\
   =& \left(\frac{\rho_2^3-\rho_1^3}{3}\right)\frac{4\pi}{3}\delta_{\alpha\beta}
  \end{split}
 \end{equation*}
This gives a lower bound on the bilinear form as 
 \begin{equation*}
  L\xi\cdot \xi \geq \frac{k\pi(\rho_2^3-\rho_1^3)}{9}\delta_{\alpha\beta}\xi_{i\alpha}\xi_{i\beta}=\frac{k\pi(\rho_2^3-\rho_1^3)}{9}|\xi|^2 \qedhere
 \end{equation*}
\end{proof}

{\BBB
We will prove Lemma~\ref{lemma:decay} by blow-up,
adapting Luckhaus' arguments from~\cite{Luckhaus}. By a scaling argument
(see~\eqref{scaling}), we may assume without loss of generality that
$x_0=0$ and~$\rho=1$.
Suppose, towards a contradiction, that Lemma~\ref{lemma:decay}
does not hold. Then, for any choice of the parameter~$\theta\in (0, \, 1/2)$,
we find a sequence~$\eps_j\to 0$ and minimisers~$u_j$ such that
\begin{gather}
 \eta_j^2 := F_{\eps_j}(u_j, \, B_1) \to 0 
  \qquad \textrm{as } j\to+\infty, \label{smallenergy} \\
 F_{\eps_j} (u_j, \, B_\theta) > \frac{\theta\eta_j^2}{2} + \eps_j^{2q-4}
  \label{nodecay}
\end{gather}
We denote by~$\T_\zeta\NN$ the tangent space to the manifold~$\NN$
at a point~$\zeta\in\NN$ (regarded as a linear subspace of~$\R^m$,
i.e.~$0\in\T_{\zeta}\NN$). We recall that, for any point~$y\in\R^m$
sufficiently close to the manifold~$\NN$, there exists a unique
point~$\pi(y)\in\NN$ such that~$\dist(y, \, \NN) = \abs{y - \pi(y)}$.
Moreover, there exists a constant~$\delta_*(\NN) > 0$ such that
the map~$y \mapsto \pi(y)$ is smooth (with bounded derivatives) in
\begin{equation} \label{nearestproj}
 \mathscr{U} :=\{z\in\R^m\colon \dist(z, \, \NN)\leq \delta_*(\NN)\}
\end{equation} 
(see e.g.~\cite[Section~2.12.3]{Simon-Harmonic}).

\begin{lemma} \label{lemma:average}
 For any~$j\in\N$ large enough, there exists a 
 constant~$\zeta_j\in\NN$ and a measurable 
 set~$G_j\subseteq B_{1/2}$ such that 
 \begin{gather}
  \int_{B_{1/2}} \abs{u_j(x) - \zeta_j}^2 \, \d x
   \lesssim \eta_j^2 \label{average1} \\
  \eta_j^{-1} \int_{B_{1/2}\setminus G_j} \dist(u_j(x) - \zeta_j, \, \T_{\zeta_j}\NN) 
   \, \d x \to 0 \label{average2}
 \end{gather}
 and~$\abs{G_j}\to 0$ as~$j\to+\infty$.
\end{lemma}
\begin{proof}
 Let~$\tilde{\zeta}_j := \fint_{B_{1/2}} u_j$. 
 By the Poincar\'e-type inequality,
 Proposition~\ref{prop:Poincare}, we have
 \begin{equation} \label{avrg1}
  \int_{B_{1/2}} \abs{u_j - \tilde{\zeta}_j}^2 
   \lesssim F_{\eps_j}(u_j, \, B_1) 
   \stackrel{\eqref{smallenergy}}{=} \eta_j^2.
 \end{equation}
 On the other hand, Proposition~\ref{prop:physicality} 
 and Lemma~\ref{lemma:psib} imply
 \begin{equation} \label{avrg2}
  \psi_b(\tilde{\zeta}_j) \lesssim
  \fint_{B_{1/2}} \psi_b(u_j) 
  + \fint_{B_{1/2}} \abs{u_j - \tilde{\zeta}_j}^2 
  \stackrel{\eqref{avrg1}}{\lesssim} 
  (\eps_j^2 + 1) \eta_j^2 \to 0.
 \end{equation}
 In particular, the distance between~$\tilde{\zeta}_j$ and~$\NN$
 tends to zero as~$j\to+\infty$. Then, for~$j$ large enough,
 the projection~$\zeta_j := \pi(\tilde{\zeta}_j)\in\NN$ is well-defined.
 Moreover, recalling~\eqref{nondeg}, we have
 \begin{equation} \label{avrg3}
  \abs{\zeta_j - \tilde{\zeta}_j}^2 = \dist^2(\tilde{\zeta}_j, \, \NN)
  \lesssim \psi_b(\tilde{\zeta}_j) \stackrel{\eqref{avrg2}}{\lesssim} \eta_j^2.
 \end{equation}
 The estimate~\eqref{average1} follows from~\eqref{avrg1} and~\eqref{avrg3}.
 
 Let us consider the set
 \[
  G_j := \left\{x\in B_{1/2}\colon \dist(u_j(x), \, \NN) 
  \geq \delta_*(\NN) \right\}
 \]
 where~$\delta_*(\NN)$ is the constant from~\eqref{nearestproj}.
 On the set~$G_j$, the function~$\psi_b(u_j)$ is bounded from
 below by a strictly positive constant, which depends only on~$\delta_*(\NN)$
 and~$\psi_b$. Then,
 \[
  \abs{G_j} \lesssim \int_{B_1} \psi_b(u_j) 
   \lesssim \eps_j^2 \, F_{\eps_j}(u_j, \, B_1) 
   = \eps_j^2 \, \eta_j^2  \to 0.
 \]
 Moreover, the projection~$\pi\circ u_j$ is well-defined
 on~$B_{1/2}\setminus G_j$ and
 \[
  \begin{split}
   \int_{B_{1/2}\setminus G_j} \dist(u_j - \zeta_j, \, \T_{\zeta_j}\NN)
   &\leq \int_{B_{1/2}\setminus G_j} \dist(\pi\circ u_j - \zeta_j, \, \T_{\zeta_j}\NN)
    + \int_{B_{1/2}\setminus G_j} \abs{u_j - \pi\circ u_j} \\
   &\stackrel{\eqref{nondeg}}{\lesssim} \int_{B_{1/2}\setminus G_j} 
    \dist(\pi\circ u_j - \zeta_j, \, \T_{\zeta_j}\NN)
    + \int_{B_{1}} \psi_b^{1/2}(u_j) \\
   &\lesssim \int_{B_{1/2}\setminus G_j} 
    \dist(\pi\circ u_j - \zeta_j, \, \T_{\zeta_j}\NN)
    + \eps_j \, \eta_j \\ 
  \end{split}
 \]
 Therefore, in order to prove~\eqref{average2}, it suffices to show that
 \begin{equation} \label{average3}
  \eta_j^{-1} \int_{B_{1/2}\setminus G_j} 
    \dist(\pi\circ u_j - \zeta_j, \, \T_{\zeta_j}\NN) \to 0 
    \qquad \textrm{as } j\to+\infty.
 \end{equation}
 To this end, we fix a large number~$M > 0$ and we consider the sets
 \[
  A_j := \left\{x\in B_{1/2}\setminus G_j\colon
   \abs{(\pi\circ u_j)(x) - \zeta_j} \leq M\eta_j \right\} \!,
   \qquad B_j := B_{1/2}\setminus (G_j\cup A_j).
 \]
 We first estimate the contribution from the set~$A_j$.
 Since the manifold $\NN$ is compact and smooth, we have
 \begin{equation} \label{avrg4}
  \dist(y-z, \, \T_{z}\NN) \lesssim \abs{y - z}^2
  \qquad \textrm{for any } y\in\NN, \ z\in\NN
 \end{equation}
 (For~$y$ sufficiently close to~$z$,
 say~$\abs{y - z}\leq\eta_0 = \eta_0(\NN)$,
 the inequality~\eqref{avrg4} can be obtained 
 by writing~$\NN$ as the graph of a smooth function,
 locally around~$z$, and using a Taylor expansion. 
 If~$\abs{y - z}\geq\eta_0$, we remark that the left-hand
 side of~\eqref{avrg4} is bounded from above because~$\NN$ is compact.) 
 Then,
 \begin{equation}\label{avrg5}
  \begin{split}
   \eta_j^{-1} \int_{A_j} \dist(\pi\circ u_j - \zeta_j, \, \T_{\zeta_j}\NN)
    \lesssim \eta_j^{-1} \int_{A_j} |\pi\circ u_j - \zeta_j|^2 
    \leq M^2 \eta_j \abs{A_j} \to 0
  \end{split}
 \end{equation}
 as~$j\to+\infty$. Now, we estimate the contribution from~$B_j$.
 By construction, the image~$u_j(B_{1/2}\setminus G_j)$ is contained
 in the neighbourhood~$\mathscr{U}$ given by~\eqref{nearestproj}.
 Moreover, $\pi$ is Lipschitz-continuous on~$\mathscr{U}$
 and~$\pi(\zeta_j) = \zeta_j$, so
 \begin{equation} \label{avrg5.5}
  \abs{(\pi\circ u_j)(x) - \zeta_j} \lesssim
  \abs{u_j(x) - \zeta_j} \qquad \textrm{for a.e. } x\in B_{1/2}\setminus G_j
 \end{equation}
 By definition, we have $\abs{\pi\circ u_j - \zeta_j}\geq M\eta_j$
 on~$B_j$ and hence, 
 \begin{equation} \label{avrg6}
  \abs{B_j} \lesssim M^{-2} \eta_j^{-2} 
   \int_{B_{1/2}\setminus G_j} \abs{\pi\circ u_j - \zeta_j}^2
   \stackrel{\eqref{avrg5.5}}{\lesssim} M^{-2} \eta_j^{-2} 
   \int_{B_{1/2}\setminus G_j} \abs{u_j - \zeta_j}^2
   \stackrel{\eqref{average1}}{\lesssim} M^{-2}
 \end{equation}
 By applying the inequality $\dist(y-z, \, \T_{z}\NN)\leq |y - z|$,
 which holds for any~$y\in\NN$, $z\in\NN$, we obtain
 \begin{align*}
  \eta_j^{-1} \int_{B_j} \dist(\pi\circ u_j - \zeta_j, \, \T_{\zeta_j}\NN) 
   \leq \eta_j^{-1} \int_{B_j} \abs{\pi\circ u_j - \zeta_j} 
  \stackrel{\eqref{avrg5.5}}{\lesssim}
   \eta_j^{-1} \int_{B_j} \abs{u_j - \zeta_j} 
  \end{align*}
  and hence,
  \begin{equation} \label{avrg7}
  \eta_j^{-1} \int_{B_j} \dist(\pi\circ u_j - \zeta_j, \, \T_{\zeta_j}\NN) 
   \lesssim \eta_j^{-1} \abs{B_j}^{1/2} 
   \left(\int_{B_{1/2}} \abs{u_j - \zeta_j}^2\right)^{1/2}
   \stackrel{\eqref{average1}, \, \eqref{avrg6}}{\lesssim}  M^{-1}
  \end{equation}
  By combining~\eqref{avrg5} with~\eqref{avrg7}, and passing 
  to the limit as~$M\to+\infty$, we deduce~\eqref{average3}.
\end{proof}

Let~$\zeta_j\in\NN$, $G_j\subseteq B_{1/2}$ be as in Lemma~\ref{lemma:average}.
We consider the maps $v_j\colon B_{1/2}\to\R^3$ given by
\[
 v_j := \frac{1}{\eta_j} \left(u_j - \zeta_j\right)
\]
Thanks to~\eqref{average1}, the sequence~$v_j$
is bounded in~$L^2(B_{1/2})$.

\begin{lemma} \label{lemma:L2conv}
 There exist a (non-relabelled) subsequence, a point~$\zeta\in\NN$
 and a map~$v\in H^1(B_{1/2}, \, \R^m)$ such that $\zeta_j\to \zeta$,
 $v_j\to v$ strongly in~$L^2(B_{1/2})$ and a.e.~as~$j\to+\infty$.
 Moreover,
 \begin{equation} \label{tg}
  v(x)\in \T_\zeta\NN \qquad \textrm{for a.e. } x\in B_{1/2}.
 \end{equation}
\end{lemma}
\begin{proof}
 Let~$\varphi_{\eps_j}$ be a sequence of mollifiers, as in~\eqref{mollify}.
 Since~$\zeta_j$ is constant, we have $\varphi_\eps*\zeta_j = 
 \zeta_j \int_{\R^3} \varphi_{\eps_j} = \zeta_j$ and hence,
 by Lemma~\ref{lemma:mollify-Morrey},
 \[
  \int_{B_{1/2}} \abs{\nabla (\varphi_{\eps_j}*v_j)}^2
  = \frac{1}{\eta_j^2} \int_{B_{1/2}} \abs{(\nabla \varphi_{\eps_j})*u_j}^2 
  \lesssim \frac{F_{\eps_j}(u_j, \, B_1)}{\eta_j^2} 
  \stackrel{\eqref{smallenergy}}{=} 1
 \]
 Moreover, by Lemma~\ref{lemma:mollify-L2},
 \[
  \int_{B_{1/2}} \abs{v_j - \varphi_{\eps_j}*v_j}^2
  = \frac{1}{\eta_j^2} \int_{B_{1/2}} \abs{u_j - \varphi_{\eps_j}*u_j}^2
  \lesssim \frac{\eps_j^2 \, F_{\eps_j}(u_j, \, B_1)}{\eta_j^2} 
  \stackrel{\eqref{smallenergy}}{=} \eps_j^2 \to 0
 \]
 Then, we may extract a subsequence in such a way that
 $\varphi_{\eps_j}*v_j\rightharpoonup v$ weakly in~$H^1(B_{1/2})$
 and~$v_j \to v$ strongly in~$L^2(B_{1/2})$ and a.e.
 We may also assume that~$\zeta_j\to \zeta\in\NN$, because~$\NN$ is compact.
 
 It remains to prove~\eqref{tg}. Let~$\delta > 0$ be a small parameter.
 By Lemma~\ref{lemma:average}, we know that~$\abs{G_j} \to 0$,
 so we may extract a subsequence~$j_k\to+\infty$ in such a way that
 $\sum_{k\in\N} |G_{j_k}| \leq \delta$. Let~$G := \cup_{k\in\N} G_{j_k}$.
 The estimate~\eqref{average2} implies
 \[
  \int_{B_{1/2}\setminus G} \dist(v_{j_k}, \, \T_{\zeta_{j_k}}\NN)
  \leq \int_{B_{1/2}\setminus G_{j_k}} \dist(v_{j_k}, \, \T_{\zeta_{j_k}}\NN) \to 0
 \]
 as~$k\to+\infty$. By Fatou lemma, we deduce that~$v(x)\in\T_{\zeta}\NN$
 for a.e.~$x\in B_{1/2}\setminus G$. This implies
 \[
  \abs{\left\{x\in B_{1/2}\colon v(x)\notin\T_\zeta\NN\right\}}
  \leq \abs{G} \leq \sum_{k\in\N} \abs{G_{j_k}} \leq \delta.
 \]
 Since~$\delta>0$ is arbitrary, \eqref{tg} follows.
\end{proof}

Next, we show that~$v$ minimises the gradient energy 
associated with the tensor~$L$, subject to its own boundary conditions.

\begin{lemma} \label{lemma:minim}
 Let~$v^*\in H^1(B_{1/2}, \, \T_\zeta\NN)$ be such that
 $v^* = v$ on~$\partial B_{1/2}$ (in the sense of traces).
 Then,
 \begin{equation} \label{blmin}
  \int_{B_{1/2}} L\nabla v\cdot \nabla v
   \leq \int_{B_{1/2}} L\nabla v^*\cdot\nabla v^*
 \end{equation}
 Moreover, for any~$s\in(0, \, 1/2)$ there holds
 \begin{equation} \label{blstrongconv}
  \lim_{\eps\to 0} \eta_j^{-2} F_{\eps_j}(u_j, \, B_s) 
  = \int_{B_s} L\nabla v\cdot\nabla v 
 \end{equation}
\end{lemma}
\begin{proof}
 By a density argument, it suffices to prove~\eqref{blmin}
 in case $v^* = v$ in a neighbourhood of~$\partial B_{1/2}$.
 Let us fix~$s\in (0, \, 1/2)$ and~$v^*\in H^1(B_{1/2}, \, \T_{\zeta}\NN)$
 such that~$v^*=v$ a.e.~in~$B_{1/2}\setminus\bar{B}_s$.
 For any~$a>0$, we let $\mathscr{U}_a := \{y\in\R^m\colon 
 \dist(y, \, \NN)\leq a\}$. Let~$z\in\NN$, $R>0$ and~$w\in\T_z\NN$
 be such such that~$\abs{w}\leq R$. Since~$\nabla\pi(z)$ is the 
 orthogonal projection onto~$\T_z\NN$ 
 (see e.g.~\cite[Section~2.12.3]{Simon-Harmonic}), we have
 \begin{equation} \label{pi1}
  \abs{z + \eta_jw - \pi(z - \eta_jw)} 
  \lesssim  \eta_j^2 \, R^2 \, \|\nabla^2\pi\|_{L^\infty(\mathscr{U}_{\eta_j R})} 
 \end{equation}
 Moreover, for any~$X\in\T_z\NN$ we have
 \begin{equation} \label{pi2}
  \abs{X - \nabla\pi(z + \eta_jw)X}
  \leq \norm{\nabla\pi(z) - \nabla\pi(z+\eta_jw)}\abs{X}
  \lesssim \eta_j R \abs{X}\, 
  \|\nabla^2\pi\|_{L^\infty(\mathscr{U}_{\eta_j R})} 
 \end{equation}
 We choose a positive sequence~$R_j\to+\infty$ such that~$\eta_j R_j^2\to 0$
 and we define
 \[
  v^*_j := \frac{R_jv^*}{\max(R_j, \, \abs{v^*})}, \qquad
  u^*_j := \pi(\zeta_j + \eta_j v^*_j).
 \]
 Using~\eqref{pi1} and~\eqref{pi2}, a routine computation shows that
 \[
  \frac{u^*_j - \zeta_j}{\eta_j} \to v^*
  \qquad \textrm{strongly in } H^1(B_{1/2})
 \]
 Now, we apply Lemma~\ref{lemma:nonlocinterp} (and Remark~\ref{rk:nonlocinterp}).
 For any~$\sigma\in (0, \, 1/10)$, we find radii~$r$, $t$ 
 with~$\max(1/2, \, s) < r < t < 1/2$
 and maps~$\xi_j\in L^\infty(\R^3, \, \QQ)$ such that
 $\xi_j = u_j$ a.e.~in~$\R^3\setminus B_t$,
 \begin{equation} \label{Xi}
  \Xi_j := \frac{\xi_j - \zeta_j}{\eta_j} \to v^*
  \qquad \textrm{strongly in } H^1(B_{r})
 \end{equation}
 and
 \[
  \begin{split}
   F_{\eps_j}(\xi_j, \, B_t) 
   &+ \Gamma_{\eps_j}(\xi_j, \, B_t, \, \R^3\setminus B_t)
   - \Gamma_{\eps_j}(u_j, \, B_t, \, \R^3\setminus B_t) \\
   &\leq F_{\eps_j}^\nl(\xi_j, \, B_r) 
     + C\int_{B_r\setminus B_{r-\sigma}} \abs{\nabla\xi_j}^2 
     + C\sigma \eta_j^2 + \o\!\left(\frac{\eps_j^{q-2} \, \eta_j}{\sigma^q}\right)
  \end{split}
 \]
 Since~$u_j$ is a minimiser of~$E_{\eps_j}$, we have
 \[
  F_{\eps_j}(u_j, \, B_t) 
  + \Gamma_{\eps_j}(u_j, \, B_t, \, \R^3\setminus B_t)
  \leq F_{\eps_j}(\xi_j, \, B_t) 
  + \Gamma_{\eps_j}(\xi_j, \, B_t, \, \R^3\setminus B_t)
 \]
 and hence,
 \[
  F_{\eps_j}(u_j, \, B_t) \leq F_{\eps_j}^\nl(\xi_j, \, B_r) 
     + C\int_{B_r\setminus B_{r-\sigma}} \abs{\nabla\xi_j}^2 
     + C\sigma \eta_j^2 + \o\!\left(\frac{\eps_j^{q-2} \, \eta_j}{\sigma^q}\right)
 \]
 We divide both sides of this inequality by~$\eta^2_j$ and obtain
 \begin{equation} \label{blmin1}
  \begin{split}
    F^\nl_{\eps_j}(v_j, \, B_t)
   &\leq  \frac{1}{\eta^2_j} F_{\eps_j}(u_j, \, B_t) \\
   &\leq F_{\eps_j}^\nl\left(\Xi_j, \, B_r\right) 
     + C\int_{B_r\setminus B_{r-\sigma}} \abs{\nabla\Xi_j}^2 
     + C\sigma + \o\!\left(\frac{\eps_j^{q-2}}{\sigma^q \, \eta_j}\right)
  \end{split}
 \end{equation}
 The assumptions~\eqref{smallenergy}, \eqref{nodecay}
 imply that~$\eta_j^2 \geq \eps_j^{2q-4}$, that is
 $\eta_j \geq \eps_j^{q-2}$. Then, recalling~\eqref{Xi},
 Proposition~\ref{prop:Gamma-liminf} and Proposition~\ref{prop:Gamma-limsup},
 we may pass to the limit in~\eqref{blmin1}, 
 first as~$j\to+\infty$, then as~$\sigma\to 0$. We obtain
 \[
  \int_{B_t} L\nabla v\cdot\nabla v \leq 
  \limsup_{j\to +\infty} \frac{1}{\eta^2_j} F_{\eps_j}(u_j, \, B_t)
  \leq \int_{B_t} L\nabla v^*\cdot\nabla v^*
 \]
 and~\eqref{blmin} follows. By choosing~$v^* = v$,
 we also deduce~\eqref{blstrongconv}, exactly as in
 Proposition~\ref{prop:compactness}.
\end{proof}

Now, we are in position to complete the proof of Lemma~\ref{lemma:decay}.

\begin{proof}[Proof of Lemma~\ref{lemma:decay}]
 The assumption~\eqref{smallenergy}, Lemma~\ref{lemma:L2conv}
 and Proposition~\ref{prop:Gamma-liminf} imply that
 \begin{equation} \label{minim0}
  \int_{B_{1/2}}L\nabla v\cdot\nabla v 
  \leq \liminf_{j\to+\infty} \frac{1}{\eta_j^2} F_{\eps_j}(u_j, \, B_{1/2})\leq 1
 \end{equation}
 By Lemma~\ref{lemma:minim}, $v\in H^1(B_{1/2}, \, \T_\zeta\NN)$
 minimises a quadratic functional among maps with values
 in the linear space~$\T_\zeta\NN$. In particular, $v$
 is a solution of the system
 \[
  -\div\left(L\nabla v\right) = 0 \qquad \textrm{in } B_{1/2}
 \]
 This system has constant coefficients and is (strongly) elliptic,
 by Proposition~\ref{prop:L}. Then, elliptic 
 regularity theory (see e.g.~\cite[Section~III.2, 
 Theorem~2.1, Remarks~2.2 and~2.3]{Giaquinta83}) 
 implies that~$v$ is smooth and
 there exists a number~$\theta\in (0, \, 1)$ 
 (depending only on~$L$) such that
 \begin{equation*} 
  \int_{B_{\theta}}L\nabla v\cdot\nabla v
  \leq \frac{\theta}{4} \int_{B_{1/2}}L\nabla v\cdot\nabla v
  \stackrel{\eqref{minim0}}{\leq} \frac{\theta}{4}
 \end{equation*}
 However, \eqref{nodecay} and~\eqref{blstrongconv} imply
 \begin{equation*} 
  \int_{B_{\theta}}L\nabla v\cdot\nabla v
  = \lim_{j\to+\infty} \frac{1}{\eta_j^2} F_{\eps_j}(u_j, \, B_{1/2})
  \geq \frac{\theta}{2}
 \end{equation*}
 Thus, we have obtained a contradiction, and the lemma follows.
\end{proof}
}

\subsection{Proof of Theorem~\ref{th:Holder} and Theorem~\ref{th:conv}}

\begin{proof}[Proof of Theorem~\ref{th:Holder}]
 Let~$\eta$, $\theta$ and~$\eps_*$ be given by Lemma~\ref{lemma:decay}.
 Let~$B_{r_0}(x_0)\csubset\Omega$ be a ball, $\eps\in (0, \, \eps_* r_0)$,
 and let~$u_\eps$ be a minimiser of~$E_\eps$ in~$\mathscr{A}$ such that
 \begin{equation} \label{Holder0}
  F_\eps(u_\eps, \, B_{r_0}(x_0)) \leq \eta^2\,r_0.
 \end{equation}
 By a scaling argument, using~\eqref{scaling}, we can assume
 without loss of generality that~$x_0 = 0$ and~$r_0 = 1$.
 {\BBB We fix a parameter~$\gamma\in (0, \, 1)$.}
 
 \setcounter{step}{0}
 \begin{step}[Campanato estimate for radii~$\rho\geq\eps^\gamma$]
  {\BBB In this case, the result follows by Lemma~\ref{lemma:decay}
  combined with a classical iteration argument.
  Let~$k\geq 1$ be an integer such that~$\theta^k\geq \eps^\gamma$.
  Thanks to~\eqref{Holder0}, we can apply 
  Lemma~\ref{lemma:decay} iteratively, first on the ball~$B_1$,
  then on~$B_\theta$, on~$B_{\theta^2}$, and so on. We obtain
  \begin{equation} \label{holder-largerad1}
   F_\eps(u_\eps, \, B_{\theta^k}) 
   \leq \frac{\theta^k}{2^k} F_\eps(u_\eps, \, B_1)
   + \eps^{2q - 4}\theta^{k-1} 
   \sum_{j=0}^{k-1} 2^{-j} \theta^{(4 - 2q)(k - 1 - j)}
  \end{equation}
  At each step of the iteration, the assumptions of
  Lemma~\ref{lemma:decay} remain satisfied, so long as~$\theta^k\geq \eps^\gamma$
  and~$\eps$ is small enough. Indeed, \eqref{holder-largerad1} implies
  \[
   F_\eps(u_\eps, \, B_{\theta^k}) 
     \leq \frac{\theta^k}{2^k} F_\eps(u_\eps, \, B_1)
     + \eps^{2q - 4}\,\theta^{(k-1)(5 - 2q)}
     \sum_{j=0}^{k-1} \left(\frac{\theta^{2q-4}}{2}\right)^j
  \]
  The series at the right-hand side converges,
  because $\theta \leq 1/2$,
  and its sum is less than~$2$. The assumption~\eqref{Holder0} implies
  \begin{equation} \label{holder-largerad2}
   \begin{split}
    F_\eps(u_\eps, \, B_{\theta^k}) 
    \leq \theta^k 
     \left(\frac{\eta^2}{2^k} + 2\theta^{2q-5}
     \left(\frac{\eps}{\theta^k}\right)^{2q-4} \right) 
   \end{split}
  \end{equation}
  Under the assumption that~$\theta^k\geq\eps^\gamma$, we can 
  further estimate the right-hand side as
  \begin{equation*} 
   \begin{split}
    F_\eps(u_\eps, \, B_{\theta^k}) 
    &\leq \theta^k \left(\frac{\eta^2}{2^k}
    + 2\theta^{2q-5} \, \eps^{(2q - 4)(1 - \gamma)}\right) 
   \end{split}
  \end{equation*}
  and we can make sure that $F_\eps(u_\eps, \, B_{\theta^k}) 
  \leq \eta^2\,\theta^k$ by taking~$\eps$ small enough
  (depending on~$\eta$, $\theta$ and~$\gamma$ only).
  Now, take a radius~$\rho$ such that~$\eps^\gamma \leq \rho < 1$.
  Let~$k = k(\rho)\geq 0$
  be the unique integer such that $\theta^{k+1} \leq \rho < \theta^k$.
  The inequality~\eqref{holder-largerad2} implies
  \begin{equation} \label{holder-largerad3}
   \begin{split}
    F_\eps(u_\eps, \, B_{\rho}) 
    &\leq \frac{\rho}{\theta}
     \left(\frac{\eta^2}{2^k} + 2\theta^{2q-5}
     \left(\frac{\rho^{1/\gamma}}{\rho}\right)^{2q - 4} \right)
    \leq C\rho \left(\rho^{\alpha} + \rho^{(2q-4)(1/\gamma - 1)}\right)
   \end{split}
  \end{equation}
  for~$\alpha := \log 2/\abs{\log\theta} \in (0, \, 1)$
  and some constant~$C$ that depends only on~$\eta$, $\theta$.
  We assume that
  \begin{equation} \label{holderalpha}
   0 < \gamma < \frac{2q - 4}{\alpha + 2q - 4}
  \end{equation}
  which implies $\alpha < (2q - 4)(1/\gamma - 1)$.
  By applying Proposition~\ref{prop:Poincare}, and possibly 
  modifying the value of~$C$, from~\eqref{holder-largerad3} we deduce
  \begin{equation} \label{Holder-largeradii}
   \fint_{B_{\rho}} \abs{u_\eps - \fint_{B_{\rho}}u_\eps}^2
   \leq C\rho^{\alpha} \qquad \textrm{for  } \eps^\gamma \leq \rho < 1.
  \end{equation}  }
 \end{step}
 
 \begin{step}[Campanato estimate for radii~$\rho\leq\eps^\gamma$]
  {\BBB We need to show that an estimate similar to~\eqref{Holder-largeradii}
  holds for~$\rho<\eps^\gamma$ as well. To this end,
  we consider the number~$\nu>1$ given by 
  Assumption~\eqref{hp:nabla_K} and we define
  \begin{equation} \label{beta,p}
   p := 3 + \alpha - \frac{\alpha}{\nu}, \qquad
   \beta := 1 - \frac{\alpha}{(\alpha + 3)\nu}
  \end{equation}
  We have $p > 3$, $0 < \beta < 1$. We claim that 
  we can choose the parameter~$\gamma$ so as to 
  satisfy~\eqref{holderalpha} and
  \begin{equation} \label{betap1}
   \beta < \gamma
  \end{equation}
  Indeed, a straightforward algebraic manipulation shows that
  \begin{equation} \label{betap2}
   1 - \frac{\alpha}{(\alpha + 3)\nu} < \frac{2q - 4}{\alpha + 2q - 4}
   \qquad \Longleftrightarrow \qquad 
   \nu < \frac{\alpha + 2q - 4}{\alpha + 3}
  \end{equation}
  The assumption that~$q > 7/2$ guarantees
  that $(\alpha + 2q - 4)/(\alpha + 3) > 1$. Moreover, we have assumed
  that~$\nabla K$ is integrable (see~\eqref{hp:K_decay}),
  so there is no loss of generality in taking a smaller value
  for~$\nu$ (so long as~$\nu > 1$). Therefore, we may assume
  that~$\nu>1$ satisfies~\eqref{betap2} and hence, we 
  can choose~$\gamma$ that satisfies~\eqref{holderalpha}
  and~\eqref{betap1}.}
  
  Let 
  \[
   m_\eps :=\fint_{B_{2\eps^\beta}} u_\eps
  \]
  and let~$\chi_\eps$ be the characteristic function of the ball~$B_{\eps^\beta}$.
  Since~$\nabla K_\eps$ has zero average, from 
  the Euler-Lagrange equation (Proposition~\ref{prop:EL})
  we obtain
  \[
   \begin{split}
    \nabla(\Lambda\circ u_\eps) 
    &= (\nabla K_\eps) * (u_\eps - m_\eps) \\
    &= (\chi_\eps\nabla K_\eps) * (u_\eps - m_\eps)
    + \left((1 - \chi_\eps)\nabla K_\eps\right) * (u_\eps - m_\eps).
   \end{split}
  \]
  Let~$\tilde{\chi}_\eps$ be the characteristic function of the
  ball~$B_{2\eps^\beta}$. Since~$\chi_\eps\nabla K_\eps$ is
  supported on~$B_{\eps^\beta}$, we deduce
  \[
   \begin{split}
    \nabla(\Lambda\circ u_\eps) 
    &= (\chi_\eps\nabla K_\eps) * (\tilde{\chi}_\eps(u_\eps - m_\eps))
    + \left((1 - \chi_\eps)\nabla K_\eps\right) * (u_\eps - m_\eps)
    \quad \textrm{in } B_{\eps^\beta}.
   \end{split}
  \]
  We apply H\"older's inequality, and 
  then Young's inequality for the convolution:
  \begin{equation} \label{Holder1}
   \begin{split}
    \norm{\nabla (\Lambda\circ u_\eps)}_{L^p(B_{\eps^\beta})}
    &\lesssim \|(\chi_\eps\nabla K_\eps)
    *(\tilde{\chi}_\eps(u_\eps - m_\eps))\|_{L^p(\R^3)} \\ 
    &\qquad\qquad + \eps^{3\beta/p}
    \norm{((1-\chi_\eps)\nabla K_\eps)
    * (u_\eps - m_\eps)}_{L^\infty(\R^3)} \\
    &\lesssim \norm{\nabla K_\eps}_{L^1(B_{\eps^\beta})}
    \norm{u_\eps - m_\eps}_{L^p(B_{2\eps^\beta})} \\ 
    &\qquad\qquad + \eps^{3\beta/p}
    \norm{\nabla K_\eps}_{L^1(\R^3\setminus B_{\eps^\beta})} 
    \norm{u_\eps - m_\eps}_{L^\infty(\R^3)}
   \end{split}
  \end{equation}
  We bound the terms at the right-hand side. 
  {\BBB We have~$\beta < \gamma$,
  so~$2\eps^\beta\geq\eps^\gamma$ and hence}
  \begin{equation*}
   \norm{u_\eps - m_\eps}^2_{L^2(B_{2\eps^\beta})}
   \lesssim \eps^{(3 + \alpha)\beta}
  \end{equation*}
  {\BBB due to~\eqref{Holder-largeradii}.}
  Since $\|u_\eps\|_{L^\infty(\R^3)}\leq C$, by interpolation
  we obtain
  \begin{equation} \label{Holder2}
    \begin{split}
     \norm{u_\eps - m_\eps}_{L^p(B_{2\eps^\beta})} &\lesssim 
     \norm{u_\eps - m_\eps}_{L^2(B_{2\eps^\beta})}^{2/p} 
     \lesssim \eps^{(3 + \alpha)\beta/p}.
   \end{split}
  \end{equation}
  By a change of variable, we have 
  \begin{equation} \label{Holder3}
   \norm{\nabla K_\eps}_{L^1(B_{\eps^\beta})} \leq 
   \eps^{-1}\norm{\nabla K}_{L^1(\R^3)}
  \end{equation}
  and
  \begin{equation} \label{Holder4}
   \begin{split} 
    \norm{\nabla K_\eps}_{L^1(\R^3\setminus B_{\eps^\beta})}
    &= \eps^{-1}\int_{\R^3\setminus B_{\eps^{\beta-1}}}
    \norm{\nabla K (z)} \d z \\
    &\leq \eps^{-1}\int_{\R^3\setminus B_{\eps^{\beta-1}}}
    \norm{\nabla K (z)}
     {\BBB \frac{\abs{z}^\nu}{\eps^{\nu\beta - \nu}} \, \d z }\\
    &\leq {\BBB \eps^{\nu - \nu\beta - 1}} \int_{\R^3}
    \norm{\nabla K (z)} \abs{z}^{\BBB \nu} \d z,
   \end{split}
  \end{equation}
  where the integral at the right-hand side is finite 
  by Assumption~\eqref{hp:nabla_K}. Combining~\eqref{Holder1},
  \eqref{Holder2}, \eqref{Holder3} and~\eqref{Holder4},
  and using that $u_\eps$ is bounded in~$L^\infty(\R^3)$,
  we obtain
  \begin{equation*} 
   \norm{\nabla (\Lambda\circ u_\eps)}_{L^p(B_{\eps^\beta})}
   \lesssim \eps^{(3 + \alpha)\beta/p - 1}
   + {\BBB \eps^{3\beta/p + \nu - \nu\beta - 1} }
  \end{equation*}
  By simple algebra, from~\eqref{beta,p} we obtain
  \begin{equation*} 
   \frac{(3+\alpha)\beta}{p} - 1 =
   {\BBB \frac{3\beta}{p} + \nu - \nu\beta - 1} = 0
  \end{equation*}
  so $\|\nabla(\Lambda\circ u_\eps)\|_{L^p(B_{\eps^\beta})}$
  is bounded. Thanks to~\eqref{nabla_inv}, we deduce that 
  $\|\nabla u_\eps\|_{L^p(B_{\eps^\beta})}$
  is bounded too. Since~$p>3$, {\BBB we have the Sobolev embedding
  $W^{1,p}(B_{\eps^\beta})\hookrightarrow C^\mu(B_{\eps^\beta})$,
  where
  \begin{equation} \label{Holder6}
   \mu := 1 - \frac{3}{p} 
   = {\BBB \frac{\alpha(\nu - 1)}{\alpha(\nu - 1) + 3\nu}}
  \end{equation}
  Moreover, the constant~$C$ in 
  the Sobolev inequality $[u]_{C^\mu(B_r)} \leq C\|\nabla u\|_{L^p(B_r)}$
  is independent of~$r>0$, as demonstrated
  by a scaling argument. Then, we obtain 
  \begin{equation*} 
   [u_\eps]_{C^\mu(B_{\eps^\beta})} \leq C
  \end{equation*}
  and hence, }
  \begin{equation} \label{Holder-smallradii}
   \begin{split}
    \fint_{B_\rho}\abs{u_\eps - \fint_{B_\rho}u_\eps}^2
    &\lesssim \rho^{2\mu} \qquad \textrm{for any } 
    \rho\in (0, \, \eps^\beta].
   \end{split}
  \end{equation}
 \end{step}
 
 \begin{step}[Conclusion]
  By combining~\eqref{Holder-largeradii}, \eqref{Holder6}
  and~\eqref{Holder-smallradii}, we deduce that
  \[
   \fint_{B_\rho}\abs{u_\eps - \fint_{B_\rho}u_\eps}^2 \leq 
   C \rho^{\min(\alpha, \, 2\mu)} = C \rho^{2\mu}
  \]
  for any radius~$\rho\in (0, \, 1)$ and for
  some constant~$C>0$ that does not depend on~$\eps$, $\rho$.
  Then, Campanato embedding gives an $\eps$-independent bound on the 
  $\mu$-H\"older semi-norm of~$u_\eps$ on~$B_{1/2}$.
  This completes the proof. \qedhere
 \end{step}
\end{proof}

\begin{proof}[Proof of Theorem~\ref{th:conv}]
 Let~$u_\eps$ be a minimiser of~$E_\eps$ in~$\mathscr{A}$.
 By the results of~\cite{taylor2018oseen},
 there exists a (non-relabelled)
 subsequence such that $u_\eps\to u_0$ strongly in~$L^2(\Omega)$,
 where~$u_0$ is a minimiser of the limit functional~\eqref{limit}.
 Take a point~$x_0\in\Omega\setminus S[u_0]$, where~$S[u_0]$
 is defined by~\eqref{singularset}.
 By definition of~$S[u_0]$,
 there exists a number~$r_0>0$ such that
 \[
  r_0^{-1}\int_{B_{r_0}(x_0)} L\nabla u_0\cdot\nabla u_0
  \leq \frac{\eta^2}{2},
 \]
 where~$\eta$ is given by Theorem~\ref{th:Holder}.
 Proposition~\ref{prop:compactness} implies
 \[
  r_0^{-1} F_\eps(u_\eps, \, B_{r_0}(x_0))
  \leq \eta^2
 \]
 for any~$\eps$ small enough and hence, by Theorem~\ref{th:Holder},
 $[u_\eps]_{C^\mu(B_{r_0}(x_0))}$ is uniformly bounded.
 Then, Ascoli-Arzel\`a's theorem implies that $u_\eps\to u_0$
 uniformly in~$B_{r_0}(x_0)$.
\end{proof}

\section{Generalisation to finite-thickness boundary conditions}
\label{sect:bd}

In this section, we discuss a variant of the minimisation
problem, where we prescribe~$u$ in a neighbourhood
of~$\partial\Omega$ only. Let~$\Omega_\eps\supset\!\supset\Omega$
be a larger domain, possibly depending on~$\eps$. We consider 
the functional
\begin{equation} \label{energytilde}
 \begin{split}
  \tilde{E}_\eps(u) := - \frac{1}{2\eps^2}\int_{\Omega_\eps\times\Omega_\eps} 
     K_\eps(x-y)u(x)\cdot u(y) \, \d x \, \d y 
     + \frac{1}{\eps^2} \int_\Omega\psi_s(u(x)) \, \d x
     + {\BBB \tilde{C}_\eps,}
 \end{split}
\end{equation}
{\BBB where~$\tilde{C}_\eps$ is a constant. 
The value of~$\tilde{C}_\eps$ is irrelevant for the
purposes of minimisation, but we will make a specific 
choice of~$\tilde{C}_\eps$ (see~\eqref{Ctilde} below)
for convenience. }
As before, we take a map~$u_{\mathrm{bd}}\in H^1(\R^3, \, \QQ)$
that satisfies~\eqref{hp:bd} and define the admissible class
\begin{equation} \label{Aeps}
 \tilde{\mathscr{A}}_\eps := \left\{u\in L^\infty(\Omega_\eps, \, \QQ)\colon
 \psi_s(u)\in L^1(\Omega), \
 u = u_{\mathrm{bd}} \textrm{ a.e. on } \Omega_\eps\setminus\Omega\right\} \!.
\end{equation}
The thickness of the boundary layer~$\Omega_\eps\setminus\Omega$
must be related to the decay properties of the kernel~$K$. 
More precisely, in addition to~\eqref{hp:Kfirst}--\eqref{hp:Klast}, 
\eqref{hp:Hfirst}--\eqref{hp:Hlast}, we assume that
\begin{enumerate}[label = (K$^\prime$), ref = K$^\prime$]
 \item \label{hp:Omega_eps} 
 {\BBB There exist numbers~$q > 7/2$, 
 $\beta\in\left(0, \, 1 - \frac{7}{2q}\right)$
 and~$\tau>0$ such that
 \begin{gather*}
  \int_{\R^3} g(z)\abs{z}^{q} \d z < +\infty
 \end{gather*}
 and~$\dist(\Omega, \, \partial\Omega_\eps) \geq\tau\eps^\beta$
 for any~$\eps> 0$.}
\end{enumerate}

\begin{remark} \label{rk:Omega_eps}
 In case the kernel~$K$ is compactly supported, 
 we can allow for a boundary layer of thickness proportional to~$\eps$.
 More precisely, we can replace the assumption~\eqref{hp:Omega_eps}
 with the following: there exists~$R_0>0$ such that
 $\mathrm{supp}(K)\subseteq B_{R_0}$ 
 and~$\dist(\Omega, \, \partial\Omega_\eps)\geq R_0\eps$
 for any  $\eps > 0$. The proofs in this case remain essentially unchanged.
\end{remark}

Under these assumptions, we can prove the analogues
of Theorems~\ref{th:Holder} and~\ref{th:conv}.

\begin{theorem} \label{th:Holder-tilde}
 Assume that the conditions~\eqref{hp:Kfirst}--\eqref{hp:Klast},
 \eqref{hp:Hfirst}--\eqref{hp:Hlast},
 \eqref{hp:bd} and~\eqref{hp:Omega_eps}
 are satisfied. Then, there exist positive
 numbers~$\eta$, $\eps_*$, $M$ and~$\mu\in (0, \, 1)$ such that,
 for any ball~$B_{r_0}(x_0)\csubset\Omega$, 
 any~$\eps\in (0, \, \eps_* r_0)$,
 and any minimiser~$\tilde{u}_\eps$ of~$\tilde{E}_\eps$ 
 in~$\tilde{\mathscr{A}}_\eps$, there holds
 \[
  r_0^{-1} F_\eps(\tilde{u}_\eps, \, B_{r_0}(x_0)) \leq \eta
  \qquad\Longrightarrow\qquad
  r_0^{\mu} \, [\tilde{u}_\eps]_{C^\mu(B_{r_0/2}(x_0))} \leq M.
 \]
\end{theorem}

\begin{theorem} \label{th:conv-tilde}
 Assume that the conditions~\eqref{hp:Kfirst}--\eqref{hp:Klast},
 \eqref{hp:Hfirst}--\eqref{hp:Hlast}, 
 \eqref{hp:bd} and~\eqref{hp:Omega_eps} are satisfied. 
 Let~$\tilde{u}_\eps$ be a minimiser of~$\tilde{E}_\eps$ in~$\tilde{\mathscr{A}}_\eps$.
 Then, up to extraction of a (non-relabelled) subsequence, we have
 \[
  \tilde{u}_\eps \to u_0 \qquad \textrm{locally uniformly in }
  \Omega\setminus S[u_0], 
 \]
 where $u_0$ is a minimiser of the functional~\eqref{limit}
 in~$\mathscr{A}$ and~$S[u_0]$ is defined by~\eqref{singularset}.
\end{theorem}

The proofs of Theorem~\ref{th:Holder-tilde} and~\ref{th:conv-tilde}
are largely similar to those of Theorem~\ref{th:Holder}
and~\ref{th:conv}. {\BBB First, we show that
the energy~$\tilde{E}_\eps$ has an alternative expression,
which is analogous to~\eqref{energy2}. We define
\begin{equation} \label{Ctilde}
 \tilde{C}_\eps := \frac{c_0}{\eps^2}\abs{\Omega}
 + \frac{1}{2\eps^2} \int_{\Omega_\eps\setminus\Omega}
 \left(\int_{(\Omega_\eps - x)/\eps} K(z)\,\d z\right)
 \cdot u_{\mathrm{bd}}(x)^{\otimes 2} \, \d x,
\end{equation}
where~$c_0$ is the same number as in~\eqref{psib}, and
\begin{equation}  \label{Heps}
 H_\eps(x) := \frac{1}{2\eps^2} 
 \int_{\R^3\setminus(\Omega_\eps - x)/\eps} K(z) \, \d z
 \qquad \textrm{for any } x\in\R^3.
\end{equation}

\begin{lemma} \label{lemma:Etilde2}
 For any~$u\in\tilde{\mathscr{A}}_\eps$, we have
 \[
  \begin{split}
   \tilde{E}_\eps(u) &= \frac{1}{4\eps^2} 
   \int_{\Omega_\eps\times\Omega_\eps} K_\eps(x-y)
   \cdot\left(u(x) - u(y)\right)^{\otimes 2} \d x \, \d y \\
   &\qquad\qquad\qquad + \int_{\Omega} 
   H_\eps(x)\cdot u(x)^{\otimes 2} \, \d x
   + \frac{1}{\eps^2} \int_{\Omega} \psi_b(u(x)) \, \d x
  \end{split}
 \]
\end{lemma}
\begin{proof}
 We inject the algebraic identity
 \[
  -2K(x-y)u(x)\cdot u(y) 
  = K(x-y)\cdot(u(x) - u(y))^{\otimes 2}
  - K(x-y)\cdot u(x)^{\otimes 2} - K(x-y)\cdot u(y)^{\otimes 2}
 \]
 in the expression for~$\tilde{E}_\eps$. 
 Since~$K$ is even, we obtain
 \[
  \begin{split}
   \tilde{E}_\eps(u) 
   &= F^\nl_\eps(u, \, \Omega_\eps)
   - \frac{1}{2\eps^2} \int_{\Omega_\eps\times\Omega_\eps} 
   K_\eps(x -y)\cdot u(x)^{\otimes 2} \,\d x
   + \frac{1}{\eps^2} \int_{\Omega} \psi_s(u(x)) \, \d x + \tilde{C}_\eps \\
   &= F^\nl_\eps(u, \, \Omega_\eps)
   - \frac{1}{2\eps^2} \int_{\Omega} \left(\int_{(\Omega_\eps - x)/\eps} 
   K(z) \, \d z\right) \cdot u(x)^{\otimes 2} \,\d x \\
   &\qquad  - \frac{1}{2\eps^2} 
   \int_{\Omega_\eps\setminus\Omega} \left(\int_{(\Omega_\eps - x)/\eps} 
   K(z) \, \d z\right) \cdot u_{\mathrm{bd}}(x)^{\otimes 2} \,\d x
   + \frac{1}{\eps^2} \int_\Omega \psi_s(u(x)) \, \d x + \tilde{C}_\eps
  \end{split}
 \]
 where~$F^\nl_\eps$ is defined by~\eqref{locintenergy}.
 Due to~\eqref{Ctilde}, we deduce
 \[
  \tilde{E}_\eps(u) 
   = F^\nl_\eps(u, \, \Omega_\eps)
   - \frac{1}{2\eps^2} \int_{\Omega} \left(\int_{(\Omega_\eps - x)/\eps} 
   K(z) \, \d z\right) \cdot u(x)^{\otimes 2} \,\d x 
   + \frac{1}{\eps^2} \int_\Omega \left(\psi_s(u(x)) + c_0\right) \d x 
 \]
 and, thanks to~\eqref{psib} and~\eqref{Heps}, the lemma follows.
\end{proof}
}

\begin{lemma} \label{lemma:ELtilde}
 For any~$\eps$, there exists a minimiser~$\tilde{u}_\eps$
 for~$\tilde{E}_\eps$ in~$\tilde{\mathscr{A}}_\eps$
 and it satisfies the Euler-Lagrange equation, 
 \begin{equation} \label{ELtilde}
  \Lambda(\tilde{u}_\eps(x)) = 
  \int_{\Omega_\eps} K_\eps(x-y)\tilde{u}_\eps(y)\,\d y
 \end{equation}
 for a.e.~$x\in\Omega$. 
\end{lemma}

The proof of Lemma~\ref{lemma:ELtilde} is identical to that 
of Proposition~\ref{prop:EL}, so we skip it for brevity.
We remark that the equation~\eqref{ELtilde} can be written as
\[
 \Lambda(\tilde{u}_\eps) = K_\eps * (\tilde{u}_\eps\chi_\eps)
 \qquad \textrm{a.e. on } \Omega,
\]
where~$\chi_\eps$ is the characteristic function of~$\Omega_\eps$.
In particular, the uniform strict physicality of~$\tilde{u}_\eps$
follows from~\eqref{ELtilde}, exactly as in Proposition~\ref{prop:physicality}.

\begin{lemma} \label{lemma:H}
 Under the assumption~\eqref{hp:Omega_eps},
 {\BBB the tensor field~$H_\eps$ defined by~\eqref{Heps} satisfies
 \[
  \sup_{x\in\Omega} \norm{H_\eps(x)} = \o(\eps^{q - \beta q - 2}) 
 \]
 as~$\eps\to 0$.}
\end{lemma}
{\BBB Under the assumption~\eqref{hp:Omega_eps}, we have~$\beta < 1 - 7/(2q)$,
which implies~$q - \beta q - 2 > 3/2$. In particular,
$H_\eps$ converges to zero uniformly in~$\Omega$ as~$\eps\to 0$.}
\begin{proof}[Proof of Lemma~\ref{lemma:H}]
 For any~$x\in\Omega$, we have
 {\BBB $B_{\tau\eps^\beta}(x)\subseteq\Omega_\eps$}
 by \eqref{hp:Omega_eps}, and hence 
 {\BBB $B_{\tau\eps^{\beta-1}}\subseteq(\Omega_\eps - x)/\eps$},
 {\BBB $\R^3\setminus (\Omega_\eps - x)/\eps \subseteq
 \R^3\setminus B_{\tau\eps^{\beta - 1}}$.}
 Then, the definition~\eqref{Heps} of~$H_\eps$ gives
 \[
  \begin{split}
   \norm{H_\eps(x)}
   &\leq \frac{1}{\eps^2} 
    \int_{\R^3\setminus (\Omega_\eps - x)/\eps} \norm{K(z)} \d z \\
   &\leq \frac{1}{\eps^2} 
   {\BBB \int_{\R^3\setminus B_{\tau\eps^{\beta-1}}} \norm{K(z)} \d z }\\
   &\leq {\BBB \frac{\eps^{q - \beta q - 2}}{\tau^q}
    \int_{\R^3\setminus B_{\tau\eps^{\beta-1}}} 
    \norm{K(z)} \abs{z}^q \, \d z}
  \end{split}
 \]
 {\BBB and the lemma follows, due to~\eqref{hp:Omega_eps}}.
\end{proof}

\begin{lemma} \label{lemma:omega}
 Let~$\tilde{u}_\eps$ be a minimiser for~$\tilde{E}_\eps$
 in~$\tilde{\mathscr{A}}_\eps$, identified with its extension
 by~$u_{\mathrm{bd}}$ to~$\R^3$. Then, $\tilde{u}_\eps$ is
 an~$\omega$-minimiser for~$E_\eps$ in~$\Omega$, where 
 {\BBB $\omega(s) := C s^{q - \beta q - 2}$, $C>0$ is a constant
 that depends only on~$\Omega$, $K$, $\QQ$ and~$q$, $\beta$
 are given by~\eqref{hp:Omega_eps}.}
\end{lemma}

{\BBB Lemma~\ref{lemma:omega} guarantees
that our compactness result, Proposition~\ref{prop:compactness},
applies to minimisers of~$\tilde{E}_\eps$.}

\begin{proof}[Proof of Lemma~\ref{lemma:omega}]
 We write~$\Omega_\eps^c := \R^3\setminus\Omega_\eps$.
 Let~$\tilde{u}_\eps$ be a minimiser for~$\tilde{E}_\eps$
 in~$\tilde{\mathscr{A}}_\eps$.
 Let~$B := B_\rho(x_0)\subseteq\Omega$ be a ball,
 and let~$v\in L^\infty(\R^3, \, \QQ)$ be such
 that~$v = \tilde{u}_\eps$~a.e.~on~$\R^3\setminus B$.
 By comparing~\eqref{energy} with~\eqref{energytilde}, 
 and using~\eqref{hp:K_even}, we obtain
 \begin{equation} \label{omega0}
  \begin{split}
   &E_\eps(v) - \tilde{E}_\eps(v) {\BBB - C_\eps + \tilde{C}_\eps} \\
   &= -\frac{1}{\eps^2} 
   \int_{\Omega_\eps\times\Omega_\eps^c} K_\eps(x-y)v(x)\cdot v(y)\,\d x\,\d y
   -\frac{1}{2\eps^2} 
   \int_{\Omega_\eps^c\times\Omega_\eps^c} K_\eps(x-y)v(x)\cdot v(y)\,\d x\,\d y\\
   &= -\frac{1}{\eps^2} 
   \int_{B\times\Omega_\eps^c} K_\eps(x-y)v(x)\cdot u_{\mathrm{bd}}(y)\,\d x\,\d y
   -\frac{1}{\eps^2} 
   \int_{(\Omega\setminus B)\times\Omega_\eps^c} 
    K_\eps(x-y)\tilde{u}_\eps(x)\cdot u_{\mathrm{bd}}(y)\,\d x\,\d y \\
   &\qquad\qquad\qquad -\frac{1}{2\eps^2} 
   \int_{\Omega_\eps^c\times\Omega_\eps^c} 
    K_\eps(x-y)u_{\mathrm{bd}}(x)\cdot u_{\mathrm{bd}}(y)\,\d x\,\d y
  \end{split}
 \end{equation}
 The second and third integral at the right-hand side are independent of~$v$.
 We bound the first integral by making the change of variable $y = x + \eps z$
 and applying Fubini theorem:
 \[
  \begin{split}
   \frac{1}{\eps^2} 
   \abs{\int_{B\times\Omega_\eps^c} 
   K_\eps(x-y)v(x)\cdot u_{\mathrm{bd}}(y)\,\d x\,\d y}
   \leq \norm{u_{\mathrm{bd}}}_{L^\infty(\R^3)}
   \int_{B} \norm{H_\eps(x)} \abs{v(x)} \,\d x
  \end{split}
 \]
 where~$H_\eps$ is defined by~\eqref{Heps}. 
 Since~$u_{\mathrm{bd}}$ and~$v$ both take values in the 
 bounded set~$\QQ$, and since~$\abs{B}\lesssim\rho^3\lesssim\rho$,
 by Lemma~\ref{lemma:H} {\BBB we have}
 \begin{equation} \label{omega1}
  \begin{split}
   \frac{1}{\eps^2} 
   \abs{\int_{B\times\Omega_\eps^c} 
   K_\eps(x-y)v(x)\cdot u_{\mathrm{bd}}(y)\,\d x\,\d y}
   {\BBB \lesssim \eps^{q - \beta q - 2}\rho}
  \end{split}
 \end{equation}
 From~\eqref{omega0} and~\eqref{omega1}, we deduce
 \begin{equation} \label{omega2}
  \begin{split}
   E_\eps(\tilde{u}_\eps) - \tilde{E}_\eps(\tilde{u}_\eps)  
   \leq E_\eps(v) - \tilde{E}_\eps(v) + {\BBB C \eps^{q - \beta q - 2}\rho}
  \end{split}
 \end{equation}
 On the other hand, 
 \begin{equation} \label{omega3}
  \tilde{E}_\eps(\tilde{u}_\eps) \leq \tilde{E}_\eps(v) ,
 \end{equation}
 because~$\tilde{u}_\eps$ is a minimiser for~$\tilde{E}_\eps$
 and~$v\in\tilde{\mathscr{A}}_\eps$. Combining~\eqref{omega2} and~\eqref{omega3},
 the lemma follows.
\end{proof}

{\BBB Finally, we prove the analogue of the decay lemma,
Lemma~\ref{lemma:decay}. }

\begin{lemma} \label{lemma:decaytilde}
 {\BBB There exist~$\eta>0$, $\theta\in (0, \, 1/2)$
 and~$\eps_*>0$ such that, for any
 ball~$B_{\rho}(x_0)\subseteq\Omega$, any~$\eps\in (0, \, \eps_*\rho)$
 and any minimiser~$\tilde{u}_\eps$ of~$\tilde{E}_\eps$ 
 in~$\tilde{\mathscr{A}}_\eps$ such that 
 \[
  F_\eps(\tilde{u}_\eps, \, B_{\rho}(x_0)) 
  \leq \eta^2 \rho,
 \]
 there holds
 \begin{equation} \label{decaytilde}
  \begin{split}
   & F_\eps(\tilde{u}_\eps, \, B_{\theta\rho}(x_0))
   \leq \frac{\theta}{2} F_\eps(\tilde{u}_\eps, \, B_{\rho}(x_0))
   + \left(\frac{\eps}{\rho}\right)^{2q - 2\beta q - 4} \rho.
  \end{split}
 \end{equation} }
\end{lemma}

{\BBB The estimate~\eqref{decaytilde}
is slightly worse that the corresponding estimate from
Lemma~\ref{lemma:decay}, i.e.~Equation~\eqref{decay},
because the exponent of~$\eps/\rho$ at the right-hand side
is strictly less than~$2q - 4$.
Nevertheless, \eqref{decaytilde} is still enough for our
purposes. Indeed, the assumption~$\beta< 1 - 7/(2q)$ implies that
$\bar{q} := q - \beta q > 7/2$. Then, the proofs of
Theorem~\ref{th:Holder} and Theorem~\ref{th:conv}
carry over; it suffices to replace all the occurrences of~$q$ with~$\bar{q}$.
In particular,} once Lemma~\ref{lemma:decaytilde} is proved, 
Theorem~\ref{th:Holder-tilde} and Theorem~\ref{th:conv-tilde}
follow. 

\begin{proof}[Proof of Lemma~\ref{lemma:decaytilde}]
 {\BBB Suppose that~$B_\rho(x_0)$ is a ball contained in~$\Omega$
 and that~$u\colon\Omega_\eps\to\QQ$. We 
 consider~$U_\rho := (\Omega - x_0)/\rho$, 
 $U_{\eps,\rho} := (\Omega_\eps - x_0)/\rho$ and define
 $u_{\rho}\colon U_\rho\to\QQ$
 as $u_{\rho}(y) := u(x_0 + \rho y)$. Then,
 \[
  \begin{split}
   \rho^{-1}\tilde{E}(u)
   = F_{\eps/\rho}^\nl(u_\rho, \, U_{\eps,\rho})
   + \frac{1}{(\eps/\rho)^2} \int_{U_\rho} \psi_b(u_\rho(x)) \, \d x
   + \int_{U_\rho} H_{\eps/\rho}^{\rho, x_0}(x) 
   \cdot u_\rho(x)^{\otimes 2} \, \d x
  \end{split}
 \]
 where~$F^\nl_\eps$ is defined by~\eqref{locintenergy} and
 \[
  H_{\eps/\rho}^{\rho, x_0}(x) := \frac{1}{(\eps/\rho)^2}
  \int_{\R^3\setminus (\Omega_\eps - x_0 - \rho x)/\eps}K(z) \, \d z
  = \rho^2 H_{\eps}(x_0 + \rho x)
 \]
 Lemma~\ref{lemma:H} implies that
 \[
  \sup_{x\in U_\rho}\|H_{\eps/\rho}^{\rho, x_0}(x)\| 
  = \o(\rho^2 \eps^{q - \beta q - 2}) 
  = \o\!\left((\eps/\rho)^{q - \beta q - 2}\right)
 \]
 As a consequence, if we prove the lemma in case~$x_0 = 0$, $\rho = 1$,
 then the general statement will follow, by a scaling argument.
 
 Suppose the lemma does not hold.
 Then, for any~$\theta\in (0, \, 1/2)$, there exist
 a sequence~$\eps_j\to 0$ and minimisers~$\tilde{u}_j$
 of~$\tilde{E}_{\eps_j}$ such that
 \begin{gather}
  \eta_j^2 := F_{\eps_j}(\tilde{u}_j, \, B_1) \to 0 
   \qquad \textrm{as } j\to+\infty, \label{smallenergytilde} \\
  F_{\eps_j} (\tilde{u}_j, \, B_\theta) 
   > \frac{\theta\eta_j^2}{2} + \eps_j^{2q - 2\beta q - 4}
   \geq  \frac{\theta\eta_j^2}{2} + \eps_j^{2q - 4}
   \label{nodecaytilde}
 \end{gather}
 Lemma~\ref{lemma:average} and Lemma~\ref{lemma:L2conv} carry over,
 so there exists a sequence of points~$\zeta_j\in\NN$ such that, 
 up to a non-relabelled subsequence, $\zeta_j\to\zeta$ and
 \begin{equation} \label{decaytilde0}
  \tilde{v}_j := \frac{\tilde{u}_j - \zeta_j}{\eta_j}\to \tilde{v}
  \qquad \textrm{strongly in } L^2(B_{1/2}) \textrm{ and a.e.}
 \end{equation}
 Moreover, the limit map~$\tilde{v}$ belongs to~$H^1(B_{1/2}, \, \T_{\zeta}\NN)$.
 Let~$s\in (0, \, 1/2)$ and let~$\tilde{v}^*\in H^1(B_{1/2}, \, \T_\zeta\NN)$
 be such that~$\tilde{v}^* = \tilde{v}$ a.e.~in~$B_{1/2}\setminus\bar{B}_s$.
 We construct admissible competitors~$\tilde{\xi}_\eps\in
 L^\infty(\Omega_\eps, \, \QQ)$ exactly as before,
 by applying Lemma~\ref{lemma:nonlocinterp}. We recall that there 
 exist radii~$r$, $t$ with~$\max(1/2, \, s) < r < t < 1/2$
 such that~$\tilde{\xi}_\eps = \tilde{u}_\eps$ out of~$B_t$ and
 \begin{equation} \label{decaytilde2}
  \frac{\tilde{\xi}_j - \zeta_j}{\eta_j} \to \tilde{v}^*
  \qquad \textrm{strongly in } H^1(B_r)
 \end{equation}
 However, $\tilde{u}_j$ is now a minimiser of~$\tilde{E}_{\eps_j}$,
 not of~$E_{\eps_j}$. As a result, we have the inequality
 \[
  \begin{split}
   F_{\eps_j}(\tilde{u}_j, \, B_t) 
    + \Gamma_{\eps_j}(\tilde{u}_j, \, B_t, \, \Omega_\eps\setminus B_t)
   &\leq F_{\eps_j}(\tilde{\xi}_j, \, B_t)
    + \Gamma_{\eps_j}(\tilde{\xi}_j, \, B_t, \, \Omega_\eps\setminus B_t) \\
   &\qquad + \int_{B_t} H_{\eps_j}(x)\cdot 
    \left(\tilde{\xi}_j(x)^{\otimes 2} 
    - \tilde{u}_j(x)^{\otimes 2}\right) \d x,
  \end{split}
 \]
 with an additional term at the right-hand side.
 (We recall that~$\Gamma_\eps$ is defined by~\eqref{locdefect}.)
 We claim that
 \begin{equation} \label{decaytilde1}
  I_j := \frac{1}{\eta^2_j} \int_{B_t} H_{\eps_j}(x)
  \cdot \left(\tilde{\xi}_j(x)^{\otimes 2}
  - \tilde{u}_j(x)^{\otimes 2}\right) \d x\to 0
 \end{equation}
 as~$j\to+\infty$. Once~\eqref{decaytilde1} is proved,
 the rest of the arguments carry over.
 
 Equation~\eqref{Heps} implies that the matrix~$H_\eps(x)$ 
 is symmetric, for any~$x\in B_1$. Then,
 \[
  I_j = \frac{1}{\eta^2_j} \int_{B_t} H_{\eps_j}(x)
  \left(\tilde{\xi}_j(x) - \tilde{u}_j(x)\right)
  \cdot \left(\tilde{\xi}_j(x) + \tilde{u}_j(x)\right) \d x
 \]
 As both~$\tilde{\xi}_j$ and~$\tilde{u}_j$
 take their values in the bounded set~$\QQ$, the 
 H\"older inequality gives 
 \[
  I_j \lesssim \eta_j^{-2} \|H_{\eps_j}\|_{L^\infty(B_t)}
  \left(\|\tilde{\xi}_\eps - \zeta_j\|_{L^2(B_{1/2})} 
  + \norm{\tilde{u}_j - \zeta_j}_{L^2(B_{1/2})}\right)
 \]
 We can further estimate the right-hand side
 by applying~\eqref{decaytilde0}, \eqref{decaytilde2} 
 and Lemma~\ref{lemma:H}:
 \[
  I_j = \o\!\left(\frac{\eps_j^{q - \beta q - 2}}{\eta_j}\right)
 \]
 However, \eqref{smallenergytilde} and~\eqref{nodecaytilde}
 imply that~$\eta_j^2 \geq \eps^{2q - 2\beta q - 4}$,
 so~$\eta_j \geq \eps^{q - \beta q - 2}$ and~\eqref{decaytilde1} follows.}
\end{proof}

\paragraph*{Acknowledgements.}
The authors are grateful to Arghir D. Zarnescu,
who brought the problem to their attention.
Part of this work was carried out when the authors were visiting
the \emph{Centro Internazionale per la Ricerca Matematica} (CIRM)
in Trento (Italy), supported by the Research-in-Pairs program.
The authors would like to thank the CIRM for its hospitality.
This research has been partially supported by the Basque Government through the BERC 2018-2021 program; and by Spanish Ministry of Economy and Competitiveness MINECO through BCAM Severo Ochoa excellence accreditation SEV-2017-0718 and through project MTM2017-82184-R funded by (AEI/FEDER, UE) and acronym ``DESFLU''.
G. C. was partially supported by GNAMPA-INdAM.


\bibliographystyle{plain}
\bibliography{bib4}

\end{document}